\documentclass[leqno,11pt,oneside,a4paper]{amsart}
\usepackage[textwidth=14cm,textheight=20.5cm]{geometry} 
\usepackage[english]{babel}
\usepackage[utf8]{inputenc}
\usepackage{ulem}
\usepackage{amsmath,amssymb,amsthm,stmaryrd,appendix,fontawesome}
\usepackage{microtype}
\usepackage{eqparbox}
\usepackage{graphicx,xcolor}
\usepackage{float}
\usepackage{etoolbox}

\usepackage[all,2cell]{xy}

\usepackage{datetime}

\usepackage{enumitem}
\setlist[enumerate]{label=(\textit{\roman*}\hspace{.08em}),font=\rm,itemsep=.25em}

\newcommand{\itemi}{(\textit{i}\hspace{.08em})}
\newcommand{\itemii}{(\textit{ii}\hspace{.08em})}
\newcommand{\itemiii}{(\textit{iii}\hspace{.08em})}
\newcommand{\itemiv}{(\textit{iv})}
\newcommand{\itemv}{(\textit{v})}

\newcommand{\hair}{\ifmmode\mskip1.5mu\else\kern0.08em\fi}

\usepackage[colorlinks]{hyperref} 
\hypersetup{
  pdftitle   = {Deformations of partially wrapped Fukaya categories of surfaces},
  pdfauthor  = {Severin Barmeier, Sibylle Schroll, Zhengfang Wang},
  pdfcreator = {}
}

\DeclareRobustCommand{\SkipTocEntry}[5]{}

\usepackage{tikz}
\usepackage{mathtools}
\usetikzlibrary{backgrounds,matrix,arrows.meta,decorations.markings,decorations.pathmorphing,shapes,positioning,calc}
\tikzset{>=stealth}
\renewcommand{\to}{\mathrel{\tikz[baseline]\draw[ ->,line width=.4pt] (0ex,0.65ex) -- (3ex,0.65ex);}}
\renewcommand{\mapsto}{\mathrel{\tikz[baseline]\draw[|->,line width=.4pt] (0pt,0.65ex) -- (3ex,0.65ex);}}
\renewcommand{\rightarrow}{\mathrel{\tikz[baseline]\draw[->,line width=.4pt] (0ex,0.65ex) -- (3ex,0.65ex);}}

\renewcommand{\leftarrow}{\mathrel{\tikz[baseline]\draw[<-,line width=.4pt] (0ex,0.65ex) -- (3ex,0.65ex);}}

\renewcommand{\longrightarrow}{\mathrel{\tikz[baseline]\draw[->,line width=.4pt] (0ex,0.65ex) -- (4ex,0.65ex);}}

\renewcommand{\leftrightarrow}{\mathrel{\tikz[baseline]\draw[<->,line width=.4pt] (0ex,0.65ex) -- (3.5ex,0.65ex);}}

\newcommand{\toarg}[1]{\mathrel{\tikz[baseline]\path[->,line width=.4pt] (0ex,0.65ex) edge node[above=-.4ex, overlay, font=\scriptsize] {$#1$} (3.5ex,.65ex);}}

\newcommand{\leftarrowarg}[1]{\mathrel{\tikz[baseline]\path[<-,line width=.4pt] (0ex,0.65ex) edge node[above=-.4ex, overlay, font=\scriptsize,pos=.6] {$#1$} (3.5ex,.65ex);}}

\definecolor{stopcolour}{HTML}{000000}
\definecolor{arccolour}{HTML}{00AB57}

\renewcommand{\theta}{\vartheta}
\renewcommand{\phi}{\varphi}

\numberwithin{equation}{section}
\newtheorem{theorem}[equation]{Theorem}
\newtheorem*{theorem*}{Theorem}
\newtheorem{proposition}[equation]{Proposition}
\newtheorem{lemma}[equation]{Lemma}
\newtheorem*{lemma*}{Lemma}

\newtheorem*{corollary*}{Corollary}
\theoremstyle{definition}
\newtheorem{definition}[equation]{Definition}

\newtheorem{notation}[equation]{Notation}
\theoremstyle{remark}
\newtheorem{remark}[equation]{Remark}

\newcommand{\add}{\operatorname{add}}
\renewcommand{\b}{\mathrm{b}}
\renewcommand{\Bar}{\operatorname{Bar}}

\newcommand{\e}{\mathrm{e}}
\DeclareMathOperator{\Ext}{Ext}
\DeclareMathOperator{\End}{End}

\renewcommand{\H}{\operatorname{H}}

\DeclareMathOperator{\HH}{HH}
\DeclareMathOperator{\Hom}{Hom}

\newcommand{\id}{\mathrm{id}}

\DeclareMathOperator{\rad}{rad}
\DeclareMathOperator{\per}{per}
\DeclareMathOperator{\proj}{proj}

\newcommand{\s}{\mathrm{s}}

\DeclareMathOperator{\Spec}{Spec}

\DeclareMathOperator{\tw}{tw}
\DeclareMathOperator{\Tw}{Tw}

\DeclareMathOperator{\w}{w}

\newcommand{\I}{\textup{\textsc{i}}}
\newcommand{\II}{\textup{\textsc{i\hspace{-.9pt}i}}}
\newcommand{\III}{\textup{\textsc{i\hspace{-.9pt}i\hspace{-.9pt}i}}}
\newcommand{\IV}{\textup{\textsc{i\hspace{-.9pt}v}}}
\newcommand{\V}{\textup{\textsc{v}}}
\newcommand{\VI}{\textup{\textsc{v\hspace{-.9pt}i}}}

\DeclareFontFamily{U}{mathc}{}
\DeclareFontShape{U}{mathc}{m}{it}%
{<->s*[1.01] mathc10}{}
\DeclareMathAlphabet{\mathcal}{U}{mathc}{m}{it}

\begin{document}


\title[Deformations of Fukaya categories of surfaces]{Deformations of partially wrapped Fukaya categories of surfaces}

\author{Severin Barmeier}
\address{University of Cologne, Mathematical Institute, Weyertal 86-90, 50931 Köln, Germany}
\email{s.barmeier@gmail.com}
\author{Sibylle Schroll}
\address{University of Cologne, Mathematical Institute, Weyertal 86-90, 50931 Köln, Germany}
\email{schroll@math.uni-koeln.de}
\author{Zhengfang Wang}
\address{Nanjing University, School of Mathematics, Nanjing 210093, Jiangsu, China}
\email{zhengfangw@nju.edu.cn}



\begin{abstract}
We give a complete description of the A$_\infty$ deformation theory of partially wrapped Fukaya categories of graded surfaces. We show that any abstract A$_\infty$ deformation is ``geometric'', namely it is equivalent to the partially wrapped Fukaya category of an orbifold surface obtained as a partial compactification of the original surface. For certain genus $0$ surfaces, these deformations are generically Fukaya categories of compact pillowcases.

We introduce the notion of a weak dual and use unbounded twisted complexes to overcome the curvature problem that naturally arises when some of the boundary components are fully wrapped.

Our results provide a first account of the relationship between A$_\infty$ deformations of Fukaya categories and partial compactifications, as advocated in P.~Seidel's ICM 2002 address, in the presence of stop data. All of our results also hold when the original surface has finitely many order $2$ orbifold points.
\end{abstract}

\maketitle

\setcounter{tocdepth}{2}
\tableofcontents

\section{Introduction}

Deformations of derived categories are a deeply fascinating topic at the crossroads of algebra and geometry. Kontsevich's homological mirror symmetry conjecture \cite{kontsevich1} entails an equivalence between derived categories of coherent sheaves and Fukaya categories, as well as an equivalence of the deformation theories on both sides. On the B~side, deformations of categories of coherent sheaves give rise not only to deformations of the underlying variety, but also to ``noncommutative varieties'', such deformations being intimately tied to quantizations of Poisson structures \cite{kontsevich2,lowenvandenbergh1,kaledin2}. On the A~side, A$_\infty$ deformations of Fukaya categories arise from partial compactifications of the underlying symplectic manifold or from deformations of the symplectic form \cite{seidel1}. For example, results by Auroux, Katzarkov and Orlov \cite{aurouxkatzarkovorlov1,aurouxkatzarkovorlov2} relate noncommutative deformations of $\mathbb P^2$ and del Pezzo surfaces to non-exact deformations of the exact symplectic form on their mirror Landau--Ginzburg models.

In order to describe the deformation theory of derived categories, one has to overcome several hurdles, both technical and conceptual. The deformation theory of triangulated categories starts behaving well only at the enhanced level, either in a DG or A$_\infty$ enhancement or a stable $\infty$-categorical enhancement. Unless the category in question admits a tilting object, such deformations will involve not only associative deformations, but general DG or A$_\infty$ deformations. In this setting, one quickly encounters the curvature problem which presents serious obstacles to passing back to the triangulated level.

Deformation theory itself involves an often intricate obstruction calculus which works either infinitesimally or formally, that is, over a local $\Bbbk$-algebra. However, to compare the deformed category with the original category on equal footing, one ideally would hope to pass from an infinitesimal or formal deformation to a ``strict'' deformation, by evaluating the deformation parameters to nonzero constants. This requires one to construct algebraic or at least convergent deformations. Lastly, even if successful, it is generally still unclear how to interpret any such deformations geometrically. For example, it is rather difficult to associate a classical geometric object, such as a locally ringed space, to a ``noncommutative'' deformation of a derived category of coherent sheaves --- at least if one asks this geometric object to contain enough information to recover the deformed derived category. Likewise, it is a priori unclear if a general abstract A$_\infty$ deformation of a Fukaya category has an intrinsic symplectic interpretation, say as another Fukaya category.

In this paper, we are able to give a most satisfactory solution to all of these problems for one particular class of categories --- the partially wrapped Fukaya categories of graded surfaces. These Fukaya categories are derived equivalent to gentle and skew-gentle algebras studied in the representation theory of algebras. For every partially wrapped Fukaya category of a graded surface (possibly with isolated order $2$ orbifold singularities), we construct a semi-universal family of A$_\infty$ deformations parametrized by an affine space associated to Hochschild $2$-cocycles. This family can be viewed as an algebraization of the formal deformation problem, recovering the formal deformation theory in a formal neighbourhood of the central fiber of the family. A natural description of deformations of (enhanced) categories is via deformations of the endomorphism algebra of a classical generator. However, whenever the surface has ``fully wrapped'' boundary components (without stops) and winding number $1$ or $2$, such deformations are always subject to the curvature problem. We use certain categories of unbounded twisted complexes to describe partially wrapped Fukaya categories of surfaces and all of their deformations in terms of weak generators. This bypasses the curvature problem that naturally appears for the description via classical generators. Finally, we show that all categories in this semi-universal family of deformations do in fact have a natural geometric interpretation --- they are partially wrapped Fukaya categories of surfaces with a finite number of orbifold points in the sense of \cite{chokim,amiotplamondon,barmeierschrollwang}. In other words, our results show that all abstract A$_\infty$ deformations controlled by the Hochschild complex of the partially wrapped Fukaya category are geometric and correspond to partial orbifold compactifications of the underlying surface.

\subsection{Main results}

We give a more detailed summary of our main results. Let $\mathbf S = (S, \Sigma, \eta)$ be a {\it graded orbifold surface with stops}, i.e.\ an orbifold surface $S$ with grading structure $\eta$ given by a line field on $S$, together with a closed subset $\Sigma \subset \partial S$ of stop data. We denote by $\mathcal W (\mathbf S)$ the partially wrapped Fukaya category of $\mathbf S$ (see Section~\ref{section:fukayacategories}).

\addtocontents{toc}{\SkipTocEntry}
\subsection*{Hochschild cohomology}

The deformation theory of $\mathcal W (\mathbf S)$ as an A$_\infty$ category is controlled by its second Hochschild cohomology. We show that the dimension of this Hochschild cohomology group can be read directly from the surface --- as the number of certain boundary components with winding number $1$ or $2$.

\begin{theorem}[Theorem \ref{theorem:hochschild}]
\label{theorem:mainhochschild}
Let $\mathbf S = (S, \Sigma, \eta)$ be a graded orbifold surface with stops on the boundary and let $\mathcal W (\mathbf S)$ be its partially wrapped Fukaya category. Then the dimension of $\HH^2 (\mathcal W (\mathbf S), \mathcal W (\mathbf S))$ is given by the total number of the following types of boundary components:
\begin{enumerate}
\item boundary components with a single boundary stop and winding number $1$
\item boundary components with a full boundary stop and winding number $1$ or $2$
\item boundary components without stops and winding number $1$ or $2$.
\end{enumerate}
\end{theorem}

To prove Theorem \ref{theorem:mainhochschild}, we choose a suitable generator of $\mathcal W (\mathbf S)$ whose endomorphism algebra $A$ is what we call a {\it DG gentle algebra} (see Section~\ref{section:dggentle}). This allows us to identify $\mathcal W (\mathbf S) \simeq \tw (\add A)$ as the category of twisted complexes over the DG category $\add A$ formed by the indecomposable direct summands of $A$ (viewed as a right DG $A$-module). This identification yields an isomorphism $\HH^2 (\mathcal W (\mathbf S), \mathcal W (\mathbf S)) \simeq \HH^2 (A, A)$. We compute the latter by generalizing Bardzell's resolution for monomial algebras to a natural class of DG monomial algebras (Section~\ref{subsection:bardzellresolution}). In the case of a DG gentle algebra, this resolution can be viewed as a perturbation of the Koszul resolution of the underlying graded gentle algebra.

\addtocontents{toc}{\SkipTocEntry}
\subsection*{Deformations of Fukaya categories of surfaces}

For every graded orbifold surface $\mathbf S$ with at least one partially wrapped boundary component we construct a semi-universal family of deformations of $\mathcal W (\mathbf S)$.

\begin{theorem}[Theorem \ref{theorem:solution} and Theorem \ref{theorem:deformationwrapped}]
\label{theorem:maincompactification}
Let $\mathbf S = (S, \Sigma, \eta)$ be a graded orbifold surface with stops in the boundary and let $d = \dim \HH^2 (\mathcal W (\mathbf S), \mathcal W (\mathbf S))$. Then, up to localization, the partially wrapped Fukaya category $\mathcal W (\mathbf S)$ admits a family $\{ \mathcal W_\lambda \}_{\lambda \in \mathbb A^d}$ of A$_\infty$ deformations, with central fiber $\mathcal W_0 = \mathcal W (\mathbf S)$.

This family of deformations has the following properties:
\begin{enumerate}
\item For each $\lambda \in \mathbb A^d \smallsetminus \{ 0 \}$, there is an equivalence $\mathcal W_\lambda \simeq \mathcal W (\mathbf S_\lambda)$, where $\mathbf S_\lambda$ is a graded orbifold surface with stops obtained from $\mathbf S = \mathbf S_0$ by partial compactification.
\item For each $\lambda \in \mathbb A^d$, the category $\mathcal W_\lambda \simeq \mathcal W (\mathbf S_\lambda)$ admits a formal generator which is intrinsically formal if $\lambda$ is generic.
\item The family $\{ \mathcal W_\lambda \}_{\lambda \in \mathbb A^d}$ of pretriangulated A$_\infty$ categories is induced by a family $\{ B_\lambda \}_{\lambda \in \mathbb A^d}$ of DG algebras which is versal at all $\lambda \in \mathbb A^d$ and semi-universal at $\lambda = 0$. 
\item If $\mathbf S$ has a boundary component that does not contribute to $\HH^2 (\mathcal W (\mathbf S), \mathcal W (\mathbf S))$, then $\mathcal W_\lambda$ is generically rigid, i.e.\ $\HH^2 (\mathcal W_\lambda, \mathcal W_\lambda) = 0$ for generic $\lambda$.
\end{enumerate}
\end{theorem}

To prove this result in the stated generality, we use the observation that $\mathcal W (\mathbf S)$ for general $\mathbf S$ can be obtained by localization from $\mathcal W (\bar{\mathbf S})$, where $\bar{\mathbf S}$ is the graded orbifold surface obtained from $\mathbf S$ by adding a single boundary stop to every boundary component of winding number $0$ (in case such boundary components exist). This ensures that $\mathcal W (\bar{\mathbf S})$ is {\it locally proper}, i.e.\ the space of morphisms in the homotopy category is a graded vector space which is finite-dimensional in every degree. We show that any deformation of $\mathcal W (\mathbf S)$ is induced by a deformation of $\mathcal W (\bar{\mathbf S})$ (Theorem \ref{theorem:localization}).

\addtocontents{toc}{\SkipTocEntry}
\subsection*{Weak generators, weak duality, and a solution to the curvature problem}

The identification $\mathcal W (\mathbf S) \simeq \tw (\add A)$ shows that the deformation theory of $\mathcal W (\mathbf S)$ is equivalent to the deformation theory of $A$. However, as soon as $\mathbf S$ has boundary components of winding number $1$ or $2$ without stops, i.e.\ when these boundary components are fully wrapped, then any choice of classical generator gives rise to {\it curved} deformations of $\mathcal W (\mathbf S)$ (Proposition \ref{proposition:curveddeformations}). In order to construct the semi-universal family of Theorem \ref{theorem:maincompactification} without any restrictions on $\mathbf S$ (such as nonexistence of such boundary components), we describe deformations of $\mathcal W (\mathbf S)$ via deformations of a category of certain unbounded twisted complexes over an algebra $B$, where $B$ is the endomorphism algebra of a weak generator of the derived Fukaya category $\mathrm D \mathcal W (\bar{\mathbf S})$. This allows us to circumvent the curvature problem of $\mathcal W (\mathbf S)$ in the following sense.

\begin{theorem}[Proposition \ref{proposition:localization}, Theorem \ref{theorem:localization} and Theorem \ref{theorem:solution}]
\label{theorem:maindual}
Let $\mathbf S$ be any graded orbifold surface with stops in the boundary. Then deformations of $\mathcal W (\mathbf S)$ can be uncurved in the following sense.
\begin{enumerate}
\item We have an equivalence $\mathcal W (\mathbf S) \simeq \mathcal W (\bar{\mathbf S}) [W^{-1}]$, induced by a localization functor given by ``stop removal'', where $\bar{\mathbf S}$ is the surface obtained from $\mathbf S$ by adding a single stop to each fully wrapped boundary component of winding number $0$.
\item There exists a weak generator of $\mathrm D \mathcal W (\bar{\mathbf S})$ with finite-dimensional DG endomorphism algebra $B$ such that $\mathcal W (\bar{\mathbf S})$ is equivalent to the category $\Tw^{-_J, \b} (\add B)$ formed by certain unbounded twisted complexes over the DG category $\add B$.
\item $B$ and $\Tw^{-_J, \b} (\add B) \simeq \mathcal W (\bar{\mathbf S})$ have equivalent deformation theories, with $B$ having no curved deformations.
\item Any deformation of $\mathcal W (\mathbf S)$ is induced by a deformation of $B$.
\end{enumerate}
\end{theorem}

The algebra $B$ is itself a DG gentle algebra and $A$ and $B$ can be chosen in a way that they satisfy a type of duality that we call {\it weak duality}: For certain choices of classical and weak generator with endomorphism algebras $A$ and $B$ we have $B = A^{\vee_J}$ and we call $B$ the {\it weak dual} of $A$. Here $J$ is a subset of indecomposable objects of $\add A$ or $\add B$, or equivalently a subset of the vertices of the quivers for $A$ and $B$. It is a duality since we have $B^{\vee_J} = (A^{\vee_J})^{\vee_J} \simeq A$ (see Section~\ref{subsection:weakduality}). Weak duality generalizes Koszul duality for graded gentle or skew-gentle algebras in the following sense. If $A$ is homologically smooth, then the algebra $B = A^{\vee_J}$ is derived equivalent to the Koszul dual $A^!$ and $\H^0 \Tw^{-_J, \b} (\add B) \simeq \mathrm D^\b (B) \simeq \mathrm D^\b (A^!) \simeq \per (A)$. When $A$ is proper, then $A$ has no curved deformations and we may take $J = \varnothing$ in which case $A = B = A^{\vee_J}$. In other words, weak duality may be used to turn an arbitrary DG gentle algebra $A$ into a proper DG gentle algebra $A^{\vee_J}$, in a way that the perfect derived category of $A$ can still be recovered from $A^{\vee_J}$, as long as $A$ is locally proper (equivalently, $A$ is finite-dimensional in degree $0$).

\addtocontents{toc}{\SkipTocEntry}
\subsection*{Fukaya categories of pillowcases}

In Theorem \ref{theorem:maincompactification}, the surface $\mathbf S_\lambda$ is obtained from $\mathbf S$ by compactifying certain boundary components of winding number $1$ or $2$ to orbifold points or to smooth points, respectively. In particular, it is possible that all boundary components contribute to $\HH^2$ and are all compactified for generic values of the deformation parameters. In this case, the generic deformation is the Fukaya category of a {\it pillowcase}, a sphere with four orbifold points of order $2$.

\begin{theorem}[Theorem \ref{theorem:pillowcase}]
\label{theorem:mainpillowcase}
Let $\mathbf S$ be a graded orbifold surface such that all boundary components contribute to $\HH^2 (\mathcal W (\mathbf S), \mathcal W (\mathbf S))$. Then the generic deformations of $\mathcal W (\mathbf S)$ are Fukaya categories of compact pillowcases.
\end{theorem}

The case of Theorem \ref{theorem:mainpillowcase} is the only case, where we have a true $1$-parameter moduli. In this case, the family contains no rigid categories, as the Fukaya category of a pillowcase can always be further deformed to that of a ``nearby'' pillowcase, such deformations corresponding to deformations of the (non-exact) symplectic form.

\section{Partially wrapped Fukaya categories of surfaces}
\label{section:fukayacategories}

Following \cite{kontsevich1,seidel2} we view Fukaya categories as pretriangulated A$_\infty$ categories of twisted complexes over a certain A$_\infty$ algebra. In the case of graded surfaces with stops, this A$_\infty$ algebra can be taken to be the endomorphism algebra of a collection of curves (Lagrangian submanifolds) equipped with higher structure capturing the partially wrapped Floer theory of these curves in the sense of \cite{auroux1,auroux2,sylvan,ganatrapardonshende1}. Morphisms are then given by internal intersections or boundary flows (Reeb chords) and the differential, multiplication and higher multiplications count pseudo-holomorphic disks. 

We work with the partially wrapped Fukaya categories of surfaces studied in \cite{haidenkatzarkovkontsevich,lekilipolishchuk2} and their generalization to orbifold surfaces given from several different (but equivalent) perspectives in \cite{chokim,amiotplamondon,barmeierschrollwang}. More precisely, a {\it graded (orbifold) surface with stops in the boundary} $\mathbf S = (S, \Sigma, \eta)$ consists of
\begin{itemize}
\item a smooth compact real oriented (orbifold) surface $S$ with nonempty smooth boundary $\partial S$
\item a nonempty closed subset $\Sigma \subsetneq \partial S$ of {\it stop data}
\item a {\it grading structure} given by a smooth line field $\eta$, i.e.\ a global section of the projectivized tangent (orbi)bundle.
\end{itemize}
We call a contractible connected component of $\Sigma$ a {\it boundary stop} and a connected component homeomorphic to $\mathrm S^1$ a {\it full boundary stop} and we refer to both collectively as {\it stops}. (See \cite{haidenkatzarkovkontsevich,lekilipolishchuk2,barmeierschrollwang} for further details.)

Throughout we assume that $\Sigma$ contains at least one boundary stop which is not a full boundary stop. (In other words, we assume at least one of the boundary components to be partially wrapped.)

Throughout we work with $\Bbbk$-linear categories and $\Bbbk$-algebras, where $\Bbbk$ is a field of characteristic $0$. We take the partially wrapped Fukaya category $\mathcal W (\mathbf S)$ of $\mathbf S$ to be the ($\Bbbk$-linear) pretriangulated A$_\infty$ category
\[
\mathcal W (\mathbf S) = \tw (\add \End (\Gamma))^\natural
\]
where $\Gamma = \bigoplus_i \gamma_i$ is the direct sum of a finite collection of arcs in the surface. Here, $\End (\Gamma)$ denotes its A$_\infty$ endomorphism algebra and $\tw (\add \End (\Gamma))$ denotes the category of twisted complexes \cite{bondalkapranov,kontsevich1,seidel2} over the A$_\infty$ category formed by the indecomposable direct summands of $\End (\Gamma)$ (viewed as right A$_\infty$-module over itself), and $-^\natural$ denotes idempotent completion. Concrete descriptions of $\Gamma$ can be found in \cite{haidenkatzarkovkontsevich} in the smooth case and in \cite{chokim,amiotplamondon,barmeierschrollwang,kim} in the orbifold case. In the case of smooth surfaces, the curves $\gamma_i$ should have no intersections in the interior of the surface and dissect the surface into ``polygons'', each containing at most one stop given either by a boundary stop or by a puncture, the latter being modelled by a boundary component with a full boundary stop. In the case of orbifold surfaces, there is a more general notion of dissection, where curves may also intersect at orbifold points. We work with the dissections considered in \cite[Section~5]{barmeierschrollwang}. An equivalent but slightly different perspective can be found in \cite[Section~3]{chokim}.

\begin{remark}[Proper and/or homologically smooth]
\label{remark:propersmooth}
The stop data $\Sigma$ of the graded surface $\mathbf S = (S, \Sigma, \eta)$ determines whether the partially wrapped Fukaya category is homologically smooth and/or proper: one has that $\mathcal W (\mathbf S)$ is {\it proper} if and only if every boundary component contains either a full boundary stop or at least one boundary stop \cite[Proposition 3.5]{haidenkatzarkovkontsevich} and {\it homologically smooth} if and only if $\Sigma$ contains no full boundary stops (cf.\ \cite{kalck}). (See Sections \ref{subsection:weakduality} and \ref{subsection:localization} for a similar characterization for when $\mathcal W (\mathbf S)$ is {\it locally proper}.)
\end{remark}

\begin{remark}[From pretriangulated to triangulated categories]
\label{remark:triangulated}
The category $\mathcal W (\mathbf S)$ can be viewed as a DG or A$_\infty$ enhancement of the (idempotent complete and triangulated) {\it derived Fukaya category} $\mathrm D \mathcal W (\mathbf S) = \H^0 \tw (\End (\Gamma))^\natural$, where the homotopy category $\H^0 \mathcal C$ of a DG category $\mathcal C$ has the same objects as $\mathcal C$ and morphisms are given by the zeroth cohomology groups of the morphism complexes in $\mathcal C$.

Since Hochschild cohomology and deformation theory are better behaved at the enhanced level (see e.g.\ \cite{keller21}), we will almost exclusively work at the DG level. We use unbounded twisted complexes to circumvent the need to consider curved deformations (see Section~\ref{section:curved}). This implies that we are able to pass to the triangulated level at any point.
\end{remark}

A central result about the partially wrapped Fukaya category $\mathcal W (\mathbf S)$ is the following.

\begin{theorem}
\label{theorem:cosheaf}
Let $\mathbf S$ be a graded orbifold surface with stops in the boundary. Then $\mathcal W (\mathbf S)$ is the category of global sections of a cosheaf of pretriangulated A$_\infty$ categories on a Lagrangian core of the surface.
\end{theorem}

A cosheaf-theoretical description of wrapped Fukaya categories was conjectured by Kontsevich \cite{kontsevich3}, originally for the fully wrapped Fukaya categories of Stein manifolds. Theorem \ref{theorem:cosheaf} establishes this conjecture for surfaces and is proved in \cite{dyckerhoffkapranov,haidenkatzarkovkontsevich} for smooth surfaces (see also \cite{nadlerzaslow,sibillatreumannzaslow,dyckerhoff,pascaleffsibilla} for related results) and for orbifold surfaces in \cite{barmeierschrollwang}. See also \cite{ganatrapardonshende1,ganatrapardonshende2} for analogous results in higher dimensions.

For a graded (orbifold) surface $\mathbf S$ with stops, the Lagrangian core appearing in Theorem \ref{theorem:cosheaf} is a ribbon graph $\mathrm G$ endowed with extra structure encoding stops, grading structure and orbifold points. This ribbon graph is dual to a dissection $\Delta$ consisting of arcs in the surface.

\subsection{Generators}
\label{subsection:generators}

\subsubsection{Formal generators}

A striking result of \cite{haidenkatzarkovkontsevich} is the observation that when $\mathbf S$ is a smooth surface with stops, its partially wrapped Fukaya category $\mathcal W (\mathbf S)$ admits formal (classical) generators $\Gamma$ whose endomorphism algebras $A = \End (\Gamma)$ are graded {\it gentle algebras}. At the triangulated level this yields an equivalence
\begin{equation}
\label{eq:per}
\mathrm D \mathcal W (\mathbf S) \simeq \per (A)
\end{equation}
between the derived Fukaya category of $\mathbf S$ (cf.\ Remark \ref{remark:triangulated}) and the perfect derived category of $A$. Conversely, for any graded gentle algebra $A$, there exists a graded smooth surface $\mathbf S = (S, \Sigma, \eta)$ such that \eqref{eq:per} holds \cite{lekilipolishchuk2}. Here the presentation of $A$ by a quiver with relations determines the underlying surface $S$ as well as the set $\Sigma$ of stops via a ribbon graph that can be constructed directly from the quiver \cite{schroll}. For any grading on the quiver, there exists a line field $\eta$ (unique up to homotopy) inducing this grading \cite{lekilipolishchuk2}.

Gentle algebras were originally introduced in the representation theory of finite-dimensional algebras for their tractable representation theory \cite{assemhappel,assemskowronski}. They are quadratic monomial associative algebras whose presentation via quivers with relations has at each vertex at most two incoming and two outgoing arrows
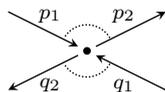
\[
\begin{tikzpicture}[baseline=-2.75pt,x=1.5em,y=1.5em]
\node[shape=circle, scale=.7] (M) at (0,0) {};
\node[shape=circle, scale=.7] (UL) at (-2,1) {};
\node[shape=circle, scale=.7] (UR) at (2,1) {};
\node[shape=circle, scale=.7] (LL) at (-2,-1) {};
\node[shape=circle, scale=.7] (LR) at (2,-1) {};
\draw[line width=1pt, fill=black] (0,0) circle(0.2ex);
\path[->, line width=.5pt, font=\scriptsize] (UL) edge node[above,pos=.6] {$p_1$} (M);
\path[->, line width=.5pt, font=\scriptsize] (M) edge node[above,pos=.4] {$p_2$} (UR);
\path[<-, line width=.5pt, font=\scriptsize] (LL) edge node[below=-.1ex,pos=.6] {$q_2$} (M);
\path[<-, line width=.5pt, font=\scriptsize] (M) edge node[below,pos=.4] {$q_1$} (LR);
\draw[dash pattern=on 0pt off 1.2pt, line width=.6pt, line cap=round] ($(0,0)+(145:.9em)$) arc[start angle=145, end angle=30, radius=.9em];
\draw[dash pattern=on 0pt off 1.2pt, line width=.6pt, line cap=round] ($(0,0)+(-145:.9em)$) arc[start angle=-145, end angle=-30, radius=.9em];
\end{tikzpicture}
\]
where the dotted lines indicate that the compositions $p_2 p_1 = 0$ and $q_2 q_1 = 0$, the other two compositions being nonzero. In the surface, this configuration of a single object in the Fukaya category and two incoming and two outgoing morphisms is given by a curve connecting two (possibly equal) boundary components
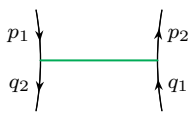
\[
\begin{tikzpicture}[decoration={markings,mark=at position 0.63 with {\arrow[black]{Stealth[length=4.2pt]}}}]
\draw[line width=0pt,postaction={decorate}] (-2.15em,1.74em) arc[start angle=10, end angle=0, radius=10em];
\draw[line width=0pt,postaction={decorate}] (-2em,0) arc[start angle=0, end angle=-10, radius=10em];
\draw[line width=0pt,postaction={decorate}] (2.15em,-1.74em) arc[start angle=190, end angle=180, radius=10em];
\draw[line width=0pt,postaction={decorate}] (2em,0) arc[start angle=180, end angle=170, radius=10em];
\draw[line width=.5pt, line cap=round] (-2em,0) arc[start angle=0, end angle=10, radius=10em];
\draw[line width=.5pt, line cap=round] (-2em,0) arc[start angle=0, end angle=-10, radius=10em];
\draw[line width=.5pt, line cap=round] (2em,0) arc[start angle=180, end angle=170, radius=10em];
\draw[line width=.5pt, line cap=round] (2em,0) arc[start angle=180, end angle=190, radius=10em];
\draw[line width=.75pt, line cap=round, color=arccolour] (-2em,0) -- (2em,0);
\node[left,font=\scriptsize] at (-2em,.85em) {$p_1$};
\node[left,font=\scriptsize] at (-2em,-.85em) {$q_2$};
\node[right,font=\scriptsize] at (2em,.85em) {$p_2$};
\node[right,font=\scriptsize] at (2em,-.85em) {$q_1$};
\end{tikzpicture}
\]
where $p_1, p_2, q_1, q_2$ correspond to Reeb chords along the boundary to and from other curves in the surface. If there is a stop on the boundary, the corresponding arrow in the quiver does not exist (and vice versa), for example
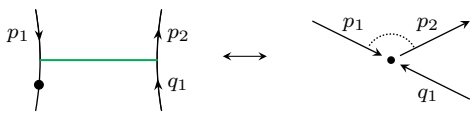
\[
\begin{tikzpicture}[x=1.5em,y=1.5em,decoration={markings,mark=at position 0.63 with {\arrow[black]{Stealth[length=4.2pt]}}}]
\begin{scope}
\draw[line width=0pt,postaction={decorate}] (-2.15em,1.74em) arc[start angle=10, end angle=0, radius=10em];
\draw[line width=0pt,postaction={decorate}] (2.15em,-1.74em) arc[start angle=190, end angle=180, radius=10em];
\draw[fill=black] (-2.04em,-.85em) circle(.15em);
\draw[line width=0pt,postaction={decorate}] (2em,0) arc[start angle=180, end angle=170, radius=10em];
\draw[line width=.5pt, line cap=round] (-2em,0) arc[start angle=0, end angle=10, radius=10em];
\draw[line width=.5pt, line cap=round] (-2em,0) arc[start angle=0, end angle=-10, radius=10em];
\draw[line width=.5pt, line cap=round] (2em,0) arc[start angle=180, end angle=170, radius=10em];
\draw[line width=.5pt, line cap=round] (2em,0) arc[start angle=180, end angle=190, radius=10em];
\draw[line width=.75pt, line cap=round, color=arccolour] (-2em,0) -- (2em,0);
\node[left,font=\scriptsize] at (-2em,.85em) {$p_1$};
\node[right,font=\scriptsize] at (2em,.85em) {$p_2$};
\node[right,font=\scriptsize] at (2em,-.85em) {$q_1$};
\node at (5em,0) {$\leftrightarrow$};
\end{scope}
\begin{scope}[xshift=10em]
\node[shape=circle, scale=.7] (M) at (0,0) {};
\node[shape=circle, scale=.7] (UL) at (-2,1) {};
\node[shape=circle, scale=.7] (UR) at (2,1) {};
\node[shape=circle, scale=.7] (LR) at (2,-1) {};
\draw[line width=1pt, fill=black] (0,0) circle(0.2ex);
\path[->, line width=.5pt, font=\scriptsize] (UL) edge node[above,pos=.6] {$p_1$} (M);
\path[->, line width=.5pt, font=\scriptsize] (M) edge node[above,pos=.4] {$p_2$} (UR);
\path[<-, line width=.5pt, font=\scriptsize] (M) edge node[below,pos=.4] {$q_1$} (LR);
\draw[dash pattern=on 0pt off 1.2pt, line width=.6pt, line cap=round] ($(0,0)+(145:.9em)$) arc[start angle=145, end angle=30, radius=.9em];
\end{scope}
\end{tikzpicture}
\]

For orbifold surfaces $\mathbf S$, the existence of a large class of formal generators for $\mathcal W (\mathbf S)$ is shown in \cite{barmeierschrollwang,kim} where a  classification of so-called {\it formal dissections} is given. (See also \cite[Section~5]{amiotplamondon} where tagged dissections are used to study tilting objects in $\mathcal W (\mathbf S)$.) We conjecture in \cite[Conjecture 8.9]{barmeierschrollwang} that the endomorphism algebras of all formal generators arise from formal dissections. This would give a geometric criterion to determine all graded associative algebras derived equivalent to a given graded gentle or skew-gentle algebra. The full classification of formal generators and the description of the resulting class of graded associative algebras is the content of ongoing work \cite{barmeierchokimrhoschrollwang}.

\begin{figure}
\begin{tikzpicture}[x=1em,y=1em,decoration={markings,mark=at position 0.7 with {\arrow[black]{Stealth[length=4.2pt]}}}]
\draw[line width=.5pt, line cap=round] (64:11.5) arc[start angle=64, end angle=123, radius=11.5];
\draw[line width=.5pt, line cap=round] (127:11.5) arc[start angle=127, end angle=190, radius=11.5];
\draw[line width=.5pt, line cap=round] (194:11.5) arc[start angle=194, end angle=246, radius=11.5];
\draw[line width=.5pt, line cap=round] (250:11.5) arc[start angle=250, end angle=271, radius=11.5];
\draw[line width=.5pt, line cap=round] (275:11.5) arc[start angle=275, end angle=324, radius=11.5];
\draw[line width=.5pt, line cap=round] (328:11.5) arc[start angle=328, end angle=371.5, radius=11.5];
\draw[line width=.5pt, line cap=round] (375.5:11.5) arc[start angle=375.5, end angle=420, radius=11.5];
\node[font=\scriptsize] at (124:11.5) {.};
\node[font=\scriptsize] at (125:11.5) {.};
\node[font=\scriptsize] at (126:11.5) {.};
\node[font=\scriptsize] at (191:11.5) {.};
\node[font=\scriptsize] at (192:11.5) {.};
\node[font=\scriptsize] at (193:11.5) {.};
\node[font=\scriptsize] at (247:11.5) {.};
\node[font=\scriptsize] at (248:11.5) {.};
\node[font=\scriptsize] at (249:11.5) {.};
\node[font=\scriptsize] at (272:11.5) {.};
\node[font=\scriptsize] at (273:11.5) {.};
\node[font=\scriptsize] at (274:11.5) {.};
\node[font=\scriptsize] at (325:11.5) {.};
\node[font=\scriptsize] at (326:11.5) {.};
\node[font=\scriptsize] at (327:11.5) {.};
\node[font=\scriptsize] at (13.5:11.5) {.};
\node[font=\scriptsize] at (12.5:11.5) {.};
\node[font=\scriptsize] at (14.5:11.5) {.};
\node[font=\scriptsize] at (61:11.5) {.};
\node[font=\scriptsize] at (62:11.5) {.};
\node[font=\scriptsize] at (63:11.5) {.};
\draw[line width=0pt,postaction={decorate}] (119:11.5) -- (120:11.5);
\node[font=\scriptsize] at (118:12.25) {$u^\I_1$};
\draw[line width=0pt,postaction={decorate}] (294:11.5) -- (295:11.5);
\node[font=\scriptsize] at (294:12.3) {$u^\III_l$};
\draw[line width=0pt,postaction={decorate}] (320:11.5) -- (321:11.5);
\node[font=\scriptsize] at (319:12.25) {$u^\IV_1$};
\draw[line width=0pt,postaction={decorate}] (353:11.5) -- (354:11.5);
\node[font=\scriptsize] at (353:12.35) {$u^\IV_m$};
\draw[line width=0pt,postaction={decorate}] (8:11.5) -- (9:11.5);
\node[font=\scriptsize] at (7:12.35) {$u^\V_1$};
\draw[line width=0pt,postaction={decorate}] (20:11.5) -- (21:11.5);
\node[font=\scriptsize] at (18:12.85) {$u^\V_{n-1}$};
\draw[line width=0pt,postaction={decorate}] (28:11.5) -- (29:11.5);
\node[font=\scriptsize] at (28:12.35) {$p^\VI_1$};
\draw[line width=0pt,postaction={decorate}] (43:11.5) -- (44:11.5);
\node[font=\scriptsize] at (42:12.35) {$p^\VI_2$};
\draw[line width=0pt,postaction={decorate}] (57:11.5) -- (58:11.5);
\node[font=\scriptsize] at (56:12.35) {$p^\VI_3$};
\draw[line width=0pt,postaction={decorate}] (67:11.5) -- (68:11.5);
\node[font=\scriptsize] at (66:12.35) {$p^\VI_{s_0-1}$};
\draw[fill=black] (90:11.5) circle(.15em);
\begin{scope}[rotate=165]
\draw[fill=black] (270:11.5) circle(.15em);
\draw[-, line width=.75pt, draw=arccolour, line cap=round] (266:11.5) to[out=85, in=95, looseness=2] (274:11.5);
\end{scope}
\begin{scope}[rotate=140]
\draw[fill=black] (270:11.5) circle(.15em);
\draw[-, line width=.75pt, draw=arccolour, line cap=round] (266:11.5) to[out=85, in=95, looseness=2] (274:11.5);
\end{scope}
\begin{scope}[rotate=125]
\draw[fill=black] (270:11.5) circle(.15em);
\draw[-, line width=.75pt, draw=arccolour, line cap=round] (266:11.5) to[out=85, in=95, looseness=2] (274:11.5);
\end{scope}
\begin{scope}[rotate=-145]
\begin{scope}[yshift=.85em]
\end{scope}
\end{scope}
\begin{scope}[rotate=-71]
\begin{scope}[yshift=.1em]
\end{scope}
\end{scope}
\begin{scope}[rotate=0]
\begin{scope}[yshift=.5em]
\end{scope}
\end{scope}
\begin{scope}[rotate=0]
\begin{scope}[yshift=.5em]
\end{scope}
\end{scope}
\begin{scope}[rotate=84]
\begin{scope}[yshift=.5em]
\end{scope}
\end{scope}
\begin{scope}[rotate=-109]
\begin{scope}[yshift=-8.5em]
\draw[line width=0pt,postaction={decorate}] (270:3) -- (275:3);
\node[font=\scriptsize] at (0,-3.75) {$u^\I_g$};
\end{scope}
\end{scope}
\begin{scope}[rotate=-165]
\begin{scope}[yshift=-7em]
\draw[line width=0pt,postaction={decorate}] (0,-4.5) -- (.4,-4.5);
\draw[line width=0pt,postaction={decorate}] (-.8,-4.48) -- (-.4,-4.49);
\draw[line width=0pt,postaction={decorate}] (.8,-4.48) -- (1.2,-4.46);
\draw[line width=.5pt, line cap=round] (-1.75,-3) to[out=90, in=180, looseness=1.1] (0,2.2) to[out=0, in=90, looseness=1.1] (1.75,-3);
\draw[line width=.5pt, line cap=round] (-.05,.7) to[bend left=20] (-.05,-1.5);
\draw[line width=.5pt, line cap=round] (0,.5) to[bend right=20] (0,-1.3);
\draw[-, line width=.75pt, draw=arccolour, line cap=round] (-1.2,-4.42) to[out=85, in=180, looseness=.8] (0,1.4) to[out=0, in=85, looseness=.9] (.4,-4.48);
\draw[-, line width=.75pt, draw=arccolour, line cap=round] (-.4,-4.45) to[out=90, in=270, looseness=.8] (.12,-1) (1.8,-1.2) to[out=270, in=85, looseness=.8] (1.2,-4.42);
\draw[-, line width=1pt, dash pattern=on 0pt off 2pt, draw=arccolour, line cap=round] (.14,-.85) to[bend left=88, looseness=1.6] (1.8,-1.2);
\node[font=\scriptsize] at (-.9,-5.2) {$p^\I_1$\strut};
\node[font=\scriptsize] at (0,-5.25) {$q^\I_1$\strut};
\node[font=\scriptsize] at (.9,-5.2) {$r^\I_1$\strut};
\end{scope}
\end{scope}
\begin{scope}[rotate=-125]
\begin{scope}[yshift=-7em]
\draw[line width=0pt,postaction={decorate}] (0,-4.5) -- (.4,-4.5);
\draw[line width=0pt,postaction={decorate}] (-.8,-4.48) -- (-.4,-4.49);
\draw[line width=0pt,postaction={decorate}] (.8,-4.48) -- (1.2,-4.46);
\draw[line width=.5pt, line cap=round] (-1.75,-3) to[out=90, in=180, looseness=1.1] (0,2.2) to[out=0, in=90, looseness=1.1] (1.75,-3);
\draw[line width=.5pt, line cap=round] (-.05,.7) to[bend left=20] (-.05,-1.5);
\draw[line width=.5pt, line cap=round] (0,.5) to[bend right=20] (0,-1.3);
\draw[-, line width=.75pt, draw=arccolour, line cap=round] (-1.2,-4.42) to[out=85, in=180, looseness=.8] (0,1.4) to[out=0, in=85, looseness=.9] (.4,-4.48);
\draw[-, line width=.75pt, draw=arccolour, line cap=round] (-.4,-4.45) to[out=90, in=270, looseness=.8] (.12,-1) (1.8,-1.2) to[out=270, in=85, looseness=.8] (1.2,-4.42);
\draw[-, line width=1pt, dash pattern=on 0pt off 2pt, draw=arccolour, line cap=round] (.14,-.85) to[bend left=88, looseness=1.6] (1.8,-1.2);
\node[font=\scriptsize] at (-1,-5.2) {$p^\I_g$\strut};
\node[font=\scriptsize] at (0,-5.25) {$q^\I_g$\strut};
\node[font=\scriptsize] at (1,-5.2) {$r^\I_g$\strut};
\end{scope}
\end{scope}
\begin{scope}[rotate=-94]
\begin{scope}[yshift=-7.5em]
\node[font=\scriptsize] at (0,0) {$\times$};
\draw[line width=.5pt, line cap=round] (225:.5em) arc[start angle=225, end angle=-30, radius=.5em];
\draw[fill=black] (0,-.6em) circle(.1em);
\draw[line width=0pt,postaction={decorate}] (0,-4) -- (.4,-4);
\draw[line width=0pt,postaction={decorate}] (-49:.5em) -- ++(237:.001em);
\draw[line width=.75pt, line cap=round, color=arccolour] (240:.4em) edge (-1,-3.93);
\draw[line width=.75pt, line cap=round, color=arccolour] (300:.4em) edge (1,-3.93);
\node[font=\scriptsize] at (0,-4.75) {$p^\II_1$};
\node[font=\scriptsize] at (0,1.1) {$q^\II_1$};
\end{scope}
\end{scope}
\begin{scope}[rotate=-85]
\begin{scope}[yshift=-8.5em]
\draw[line width=0pt,postaction={decorate}] (270:3) -- (275:3);
\node[font=\scriptsize] at (0,-3.75) {$u^\II_1$};
\end{scope}
\end{scope}
\begin{scope}[rotate=-68]
\begin{scope}[yshift=-7.5em]
\node[font=\scriptsize] at (0,0) {$\times$};
\draw[line width=.5pt, line cap=round] (225:.5em) arc[start angle=225, end angle=-30, radius=.5em];
\draw[fill=black] (0,-.6em) circle(.1em);
\draw[line width=0pt,postaction={decorate}] (0,-4) -- (.4,-4);
\draw[line width=0pt,postaction={decorate}] (-49:.5em) -- ++(237:.001em);
\draw[line width=.75pt, line cap=round, color=arccolour] (240:.4em) edge (-1,-3.93);
\draw[line width=.75pt, line cap=round, color=arccolour] (300:.4em) edge (1,-3.93);
\node[font=\scriptsize] at (0,-4.75) {$p^\II_k$};
\node[font=\scriptsize] at (0,1.1) {$q^\II_k$};
\end{scope}
\end{scope}
\begin{scope}[rotate=-55]
\begin{scope}[yshift=-8.5em]
\draw[line width=0pt,postaction={decorate}] (270:3) -- (275:3);
\node[font=\scriptsize] at (0,-3.75) {$u^\II_k$};
\end{scope}
\end{scope}
\begin{scope}[rotate=-40]
\begin{scope}[yshift=-7em]
\draw[line width=0pt,postaction={decorate}] (85:0.99) -- (75:1.02);
\draw[line width=.5pt] circle(1em);
\draw[fill=black] (0,-1) circle(.15em);
\draw[line width=.75pt, line cap=round, color=arccolour] (225:1em) to[out=245, in=85, looseness=.5] (-1.3,-4.42);
\draw[line width=.75pt, line cap=round, color=arccolour] (315:1em) to[out=295, in=95, looseness=.5] (1.3,-4.42);
\draw[line width=0pt,postaction={decorate}] (0,-4.5) -- (.4,-4.5);
\node[font=\scriptsize] at (90:1.75) {$q^\III_1$};
\node[font=\scriptsize] at (0,-5.25) {$p^\III_1$};
\end{scope}
\end{scope}
\begin{scope}[rotate=-29]
\begin{scope}[yshift=-8.5em]
\draw[line width=0pt,postaction={decorate}] (270:3) -- (275:3);
\node[font=\scriptsize] at (0,-3.75) {$u^\III_1$};
\end{scope}
\end{scope}
\begin{scope}[rotate=0]
\begin{scope}[yshift=-7em]
\draw[line width=0pt,postaction={decorate}] (85:0.99) -- (75:1.02);
\draw[line width=.5pt, line cap=round] (285:1) arc[start angle=285, end angle=-35, radius=1];
\node[font=\scriptsize] at (295:1) {.};
\node[font=\scriptsize] at (305:1) {.};
\node[font=\scriptsize] at (315:1) {.};
\draw[fill=black] (160:1) circle(.15em);
\draw[fill=black] (20:1) circle(.15em);
\draw[fill=black] (230:1) circle(.15em);
\draw[line width=0pt,postaction={decorate}] (-2.7,-4.18) arc[start angle=256, end angle=261, radius=11.5];
\draw[line width=0pt,postaction={decorate}] (-1.65,-4.39) arc[start angle=262, end angle=267, radius=11.5];
\draw[line width=0pt,postaction={decorate}] (1.8,-4.37) arc[start angle=280, end angle=285, radius=11.5];
\draw[line width=.75pt, line cap=round, color=arccolour] (125:1em) to[out=145, in=85, looseness=.9] (-2.7,-4.2);
\draw[line width=.75pt, line cap=round, color=arccolour] (195:1em) to[out=215, in=86, looseness=.5] (-1.7,-4.36);
\draw[line width=.75pt, line cap=round, color=arccolour] (265:1em) to[out=245, in=88, looseness=.9] (-.7,-4.42);
\draw[line width=.75pt, line cap=round, color=arccolour] (55:1em) to[out=35, in=95, looseness=.9] (2.7,-4.2);
\draw[line width=.75pt, line cap=round, color=arccolour] (-15:1em) to[out=-35, in=95, looseness=.5] (1.7,-4.36);
\node[font=\scriptsize] at (90:1.75) {$q^\III_l$};
\node[font=\scriptsize] at (-2.5,-5.2) {$p^\III_{l,1}$};
\node[font=\scriptsize] at (-1,-5.4) {$p^\III_{l,2}$};
\node[font=\scriptsize] at (2.4,-5.25) {$p^\III_{l,s_l}$};
\end{scope}
\end{scope}
\begin{scope}[rotate=39]
\begin{scope}[yshift=-7.5em]
\draw[line width=.2em] circle(.8em);
\draw[line width=.75pt, line cap=round, color=arccolour] (-1,-3.92) to[out=85, in=240, looseness=.8] (150:1.5) arc[start angle=150, end angle=30, radius=1.5] to[out=300, in=95, looseness=.8] (1.05,-3.92);
\draw[line width=0pt,postaction={decorate}] (270:4) -- (275:4);
\node[font=\scriptsize] at (0,-4.75) {$p^\IV_1$};
\end{scope}
\end{scope}
\begin{scope}[rotate=68]
\begin{scope}[yshift=-7.5em]
\draw[line width=.2em] circle(.8em);
\draw[line width=.75pt, line cap=round, color=arccolour] (-1,-3.92) to[out=85, in=240, looseness=.8] (150:1.5) arc[start angle=150, end angle=30, radius=1.5] to[out=300, in=95, looseness=.8] (1.05,-3.92);
\draw[line width=0pt,postaction={decorate}] (270:4) -- (275:4);
\node[font=\scriptsize] at (0,-4.9) {$p^\IV_m$};
\end{scope}
\end{scope}
%
\begin{scope}[rotate=93]
\begin{scope}[yshift=-7.5em]
\draw[line width=0pt,postaction={decorate}] (85:0.89) -- (75:.92);
\draw[line width=.5pt] circle(.9em);
\draw[line width=.75pt, line cap=round, color=arccolour] (0,-.9) to (0,-4);
\node[font=\scriptsize] at (90:1.55) {$q^\V_1$};
\end{scope}
\end{scope}
\begin{scope}[rotate=114]
\begin{scope}[yshift=-7.5em]
\draw[line width=0pt,postaction={decorate}] (85:0.89) -- (75:.92);
\draw[line width=.5pt] circle(.9em);
\draw[line width=.75pt, line cap=round, color=arccolour] (0,-.9) to (0,-4);
\node[font=\scriptsize] at (90:1.55) {$q^\V_n$};
\end{scope}
\end{scope}
\end{tikzpicture}
\caption{The standard dissection of a graded orbifold surface $\mathbf S$ with stops in each boundary component given by curves whose endomorphism algebra generates $\mathcal W (\mathbf S)$}
\label{fig:dissection}
\end{figure}

In order to study the deformation theory of $\mathcal W (\mathbf S)$, we may choose any suitable generator, since \cite{keller03} shows that the Hochschild complex of $\mathcal W (\mathbf S)$ is isomorphic to that of $\End (\Gamma)$ for any (classical) generator $\Gamma$ of $\mathcal W (\mathbf S)$ (see Section~\ref{subsection:keller}). A suitable classical generator of $\mathcal W (\mathbf S)$ can be obtained from a certain ``standard'' dissection of $\mathbf S$ illustrated in Fig.~\ref{fig:dissection}. (We refer to \cite{barmeierschrollwang} for the details.) The corresponding endomorphism algebra is described via a DG quiver with relations in Fig.~\ref{fig:quiver}, the grading and differential being given in Proposition \ref{proposition:generator}. Whereas the statement of Proposition \ref{proposition:generator} is immediate from the results of \cite{barmeierschrollwang}, we record it here to fix a systematic labelling of the arrows in the quiver for later use.

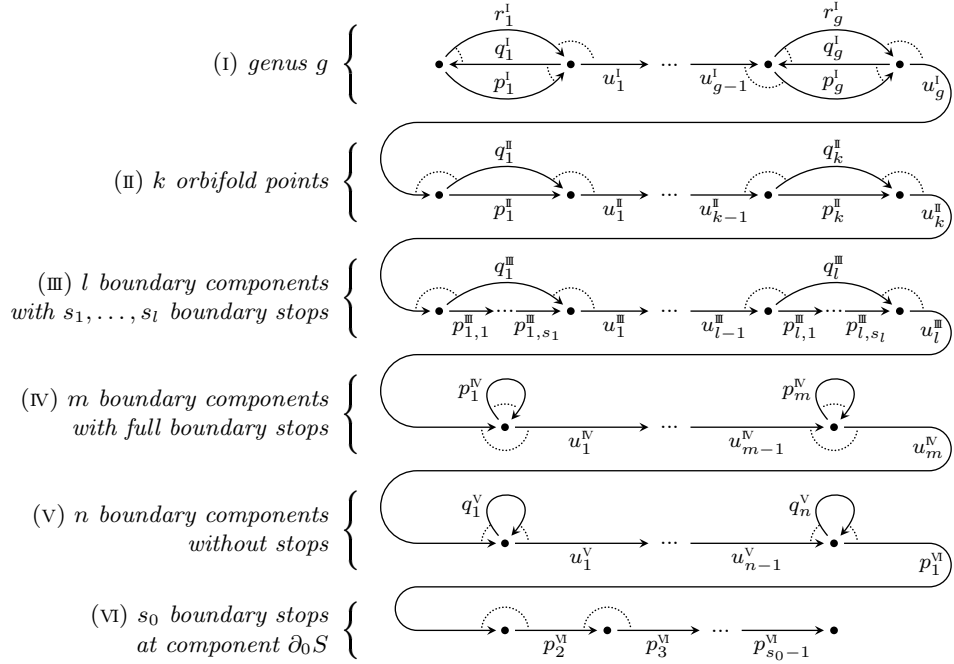
\begin{figure}
\begin{tikzpicture}[baseline=-2.75pt,x=1.5em,y=1.5em]
\begin{scope}
\node[font=\scriptsize] at (-2,0) {\rotatebox{-90}{$\underbrace{\hspace{3.5em}}$}};
\node[font=\footnotesize, left=1ex, align=right] at (-2,0) {(\I) \it genus $g$};
\foreach \label/\x in {0/0,3/3,5/5,6/5.5,8/7.5,11/10.5} {
\node[shape=circle, scale=.7] (\label) at (\x,0) {};
}
\foreach \x in {0,3,7.5,10.5} {
\draw[line width=1pt, fill=black] (\x,0) circle(0.2ex);
}
\path[->, line width=.5pt, font=\scriptsize, bend right=50] (0) edge node[above=-.5ex] {$p^\I_1$} (3);
\path[->, line width=.5pt, font=\scriptsize] (3) edge node[above=-.5ex] {$q^\I_1$} (0);
\path[->, line width=.5pt, font=\scriptsize, bend left=50] (0) edge node[above=-.5ex] {$r^\I_1$} (3);
\path[->, line width=.5pt, font=\scriptsize] (3) edge node[below=-.5ex] {$u^\I_1$} (5);
\node[font=\scriptsize] at (5.25,0) {...};
\path[->, line width=.5pt, font=\scriptsize] (6) edge node[below=-.5ex] {$u^\I_{g-1}$} (8);
\path[->, line width=.5pt, font=\scriptsize, bend right=50] (8) edge node[above=-.5ex] {$p^\I_g$} (11);
\path[->, line width=.5pt, font=\scriptsize] (11) edge node[above=-.5ex] {$q^\I_g$} (8);
\path[->, line width=.5pt, font=\scriptsize, bend left=50] (8) edge node[above=-.5ex] {$r^\I_g$} (11);
\draw[->, line width=.5pt] (11) -- ++(.5,0) arc[start angle=90, end angle=-90, radius=1em] -- ++(-11.5,0) arc[start angle=90, end angle=270, radius=1.25em] -- ++(.3,0);
\node[font=\scriptsize,left] at (11.8,-.8em) {$u^\I_g$};
\draw[dash pattern=on 0pt off 1.2pt, line width=.6pt, line cap=round] ($(0,0)+(8:.8em)$) arc[start angle=8, end angle=45, radius=.8em];
\draw[dash pattern=on 0pt off 1.2pt, line width=.6pt, line cap=round] ($(3,0)+(120:.8em)$) arc[start angle=120, end angle=10, radius=.8 em];
\draw[dash pattern=on 0pt off 1.2pt, line width=.6pt, line cap=round] ($(3,0)+(188:.8em)$) arc[start angle=188, end angle=225, radius=.8em];
\draw[dash pattern=on 0pt off 1.2pt, line width=.6pt, line cap=round] ($(7.5,0)+(8:.8em)$) arc[start angle=8, end angle=45, radius=.8em];
\draw[dash pattern=on 0pt off 1.2pt, line width=.6pt, line cap=round] ($(10.5,0)+(120:.8em)$) arc[start angle=120, end angle=10, radius=.8em];
\draw[dash pattern=on 0pt off 1.2pt, line width=.6pt, line cap=round] ($(10.5,0)+(188:.8em)$) arc[start angle=188, end angle=225, radius=.8em];
\draw[dash pattern=on 0pt off 1.2pt, line width=.6pt, line cap=round] ($(7.5,0)+(188:.8em)$) arc[start angle=188, end angle=310, radius=.8em];
\end{scope}
\begin{scope}[yshift=-4.5em]
\node[font=\scriptsize] at (-2,.3) {\rotatebox{-90}{$\underbrace{\hspace{3.5em}}$}};
\node[font=\footnotesize, left=1ex, align=right] at (-2,.3) {(\II) \it $k$ orbifold points};
\foreach \label/\x in {0/0,3/3,5/5,6/5.5,8/7.5,11/10.5} {
\node[shape=circle, scale=.7] (\label) at (\x,0) {};
}
\foreach \x in {0,3,7.5,10.5} {
\draw[line width=1pt, fill=black] (\x,0) circle(0.2ex);
}
\path[->, line width=.5pt, font=\scriptsize] (0) edge node[below=-.5ex] {$p^\II_1$} (3);
\path[->, line width=.5pt, font=\scriptsize, bend left=40] (0) edge node[above=-.5ex] {$q^\II_1$} (3);
\path[->, line width=.5pt, font=\scriptsize] (3) edge node[below=-.5ex] {$u^\II_1$} (5);
\node[font=\scriptsize] at (5.25,0) {...};
\path[->, line width=.5pt, font=\scriptsize] (6) edge node[below=-.5ex] {$u^\II_{k-1}$} (8);
\path[->, line width=.5pt, font=\scriptsize] (8) edge node[below=-.5ex] {$p^\II_k$} (11);
\path[->, line width=.5pt, font=\scriptsize, bend left=40] (8) edge node[above=-.5ex] {$q^\II_k$} (11);
\draw[->, line width=.5pt] (11) -- ++(.67,0) arc[start angle=90, end angle=-90, radius=.75em] -- ++(-11.67,0) arc[start angle=90, end angle=270, radius=1.25em] -- ++(.3,0);
\node[font=\scriptsize,left] at (11.8,-.7em) {$u^\II_k$};
\draw[dash pattern=on 0pt off 1.2pt, line width=.6pt, line cap=round] ($(0,0)+(172:.8em)$) arc[start angle=172, end angle=40, radius=.8em];
\draw[dash pattern=on 0pt off 1.2pt, line width=.6pt, line cap=round] ($(3,0)+(8:.8em)$) arc[start angle=8, end angle=140, radius=.8em];
\draw[dash pattern=on 0pt off 1.2pt, line width=.6pt, line cap=round] ($(7.5,0)+(172:.8em)$) arc[start angle=172, end angle=40, radius=.8em];
\draw[dash pattern=on 0pt off 1.2pt, line width=.6pt, line cap=round] ($(10.5,0)+(8:.8em)$) arc[start angle=8, end angle=140, radius=.8em];
\end{scope}
\begin{scope}[yshift=-8.5em]
\node[font=\scriptsize] at (-2,.3) {\rotatebox{-90}{$\underbrace{\hspace{3.5em}}$}};
\node[font=\footnotesize, left=1ex, align=right] at (-2,.3) {(\III) \it $l$ boundary components \\ \it with $s_1, \dotsc, s_l$ boundary stops};
\foreach \label/\x in {0/0,3/3,5/5,6/5.5,8/7.5,11/10.5} {
\node[shape=circle, scale=.7] (\label) at (\x,0) {};
}
\foreach \x in {0,3,7.5,10.5} {
\draw[line width=1pt, fill=black] (\x,0) circle(0.2ex);
}
\path[->, line width=.5pt, font=\scriptsize] (0) edge node[below=-.5ex] {$p^\III_{1,1}$} (1.25,0);
\node[font=\scriptsize] at (1.5,0) {...};
\path[->, line width=.5pt, font=\scriptsize] (1.75,0) edge node[below=-.5ex] {$p^\III_{1,s_1}$} (3);
\path[->, line width=.5pt, font=\scriptsize, bend left=40] (0) edge node[above=-.5ex] {$q^\III_1$} (3);
\path[->, line width=.5pt, font=\scriptsize] (3) edge node[below=-.5ex] {$u^\III_1$} (5);
\node[font=\scriptsize] at (5.25,0) {...};
\path[->, line width=.5pt, font=\scriptsize] (6) edge node[below=-.5ex] {$u^\III_{l-1}$} (8);
\path[->, line width=.5pt, font=\scriptsize] (8) edge node[below=-.5ex] {$p^\III_{l,1}$} (8.75,0);
\node[font=\scriptsize] at (9,0) {...};
\path[->, line width=.5pt, font=\scriptsize] (9.25,0) edge node[below=-.5ex] {$p^\III_{l,s_l}$} (11);
\path[->, line width=.5pt, font=\scriptsize, bend left=40] (8) edge node[above=-.5ex] {$q^\III_l$} (11);
\draw[->, line width=.5pt] (11) -- ++(.67,0) arc[start angle=90, end angle=-90, radius=.75em] -- ++(-11.67,0) arc[start angle=90, end angle=270, radius=1.25em] -- ++(1.8,0);
\node[font=\scriptsize,left] at (11.75,-.7em) {$u^\III_l$};
\draw[dash pattern=on 0pt off 1.2pt, line width=.6pt, line cap=round] ($(0,0)+(172:.8em)$) arc[start angle=172, end angle=40, radius=.8em];
\draw[dash pattern=on 0pt off 1.2pt, line width=.6pt, line cap=round] ($(3,0)+(8:.8em)$) arc[start angle=8, end angle=140, radius=.8em];
\draw[dash pattern=on 0pt off 1.2pt, line width=.6pt, line cap=round] ($(7.5,0)+(172:.8em)$) arc[start angle=172, end angle=40, radius=.8em];
\draw[dash pattern=on 0pt off 1.2pt, line width=.6pt, line cap=round] ($(10.5,0)+(8:.8em)$) arc[start angle=8, end angle=140, radius=.8em];
\end{scope}
\begin{scope}[yshift=-12.5em]
\node[font=\scriptsize] at (-2,.3) {\rotatebox{-90}{$\underbrace{\hspace{3.5em}}$}};
\node[font=\footnotesize, left=1ex, align=right] at (-2,.3) {(\IV) \it $m$ boundary components \\ \it with full boundary stops};
\foreach \label/\x in {2/1.5,5/5,6/5.5,9/9} {
\node[shape=circle, scale=.7] (\label) at (\x,0) {};
}
\foreach \x in {1.5,9} {
\draw[line width=1pt, fill=black] (\x,0) circle(0.2ex);
}
\path[->, line width=.5pt, font=\scriptsize] (2) edge[out=130, in=50, looseness=14] node[pos=.25, left=-.5ex] {$p^\IV_1$} (2);
\path[->, line width=.5pt, font=\scriptsize] (2) edge node[below=-.5ex] {$u^\IV_1$} (5);
\node[font=\scriptsize] at (5.25,0) {...};
\path[->, line width=.5pt, font=\scriptsize] (6) edge node[below=-.5ex] {$u^\IV_{m-1}$} (9);
\path[->, line width=.5pt, font=\scriptsize] (9) edge[out=130, in=50, looseness=14] node[pos=.25, left=-.5ex] {$p^\IV_m$} (9);
\draw[->, line width=.5pt] (9) -- ++(2.17,0) arc[start angle=90, end angle=-90, radius=.75em] -- ++(-11.67,0) arc[start angle=90, end angle=270, radius=1.25em] -- ++(1.8,0);
\node[font=\scriptsize,left] at (11.75,-.7em) {$u_m^\IV$};
\draw[dash pattern=on 0pt off 1.2pt, line width=.6pt, line cap=round] ($(1.5,0)+(188:.8em)$) arc[start angle=188, end angle=354, radius=.8em];
\draw[dash pattern=on 0pt off 1.2pt, line width=.6pt, line cap=round] ($(1.5,0)+(116:.8em)$) arc[start angle=116, end angle=58, radius=.8em];
\draw[dash pattern=on 0pt off 1.2pt, line width=.6pt, line cap=round] ($(9,0)+(188:.8em)$) arc[start angle=188, end angle=354, radius=.8em];
\draw[dash pattern=on 0pt off 1.2pt, line width=.6pt, line cap=round] ($(9,0)+(116:.8em)$) arc[start angle=116, end angle=58, radius=.8em];
\end{scope}
\begin{scope}[yshift=-16.5em]
\node[font=\scriptsize] at (-2,.3) {\rotatebox{-90}{$\underbrace{\hspace{3.5em}}$}};
\node[font=\footnotesize, left=1ex, align=right] at (-2,.3) {(\V) \it $n$ boundary components \\ \it without stops};
\foreach \label/\x in {2/1.5,5/5,6/5.5,9/9} {
\node[shape=circle, scale=.7] (\label) at (\x,0) {};
}
\foreach \x in {1.5,9} {
\draw[line width=1pt, fill=black] (\x,0) circle(0.2ex);
}
\path[->, line width=.5pt, font=\scriptsize] (2) edge[out=130, in=50, looseness=14] node[pos=.25, left=-.5ex] {$q^\V_1$} (2);
\path[->, line width=.5pt, font=\scriptsize] (2) edge node[below=-.5ex] {$u^\V_1$} (5);
\node[font=\scriptsize] at (5.25,0) {...};
\path[->, line width=.5pt, font=\scriptsize] (6) edge node[below=-.5ex] {$u^\V_{n-1}$} (9);
\path[->, line width=.5pt, font=\scriptsize] (9) edge[out=130, in=50, looseness=14] node[pos=.25, left=-.5ex] {$q^\V_n$} (9);
\draw[->, line width=.5pt] (9) -- ++(2.17,0) arc[start angle=90, end angle=-90, radius=.75em] -- ++(-11.67,0) arc[start angle=90, end angle=270, radius=.75em] -- ++(1.8,0);
\node[font=\scriptsize,left] at (11.75,-.7em) {$p^\VI_1$};
\draw[dash pattern=on 0pt off 1.2pt, line width=.6pt, line cap=round] ($(1.5,0)+(173:.8em)$) arc[start angle=173, end angle=130, radius=.8em];
\draw[dash pattern=on 0pt off 1.2pt, line width=.6pt, line cap=round] ($(1.5,0)+(45:.8em)$) arc[start angle=45, end angle=5, radius=.8em];
\draw[dash pattern=on 0pt off 1.2pt, line width=.6pt, line cap=round] ($(9,0)+(173:.8em)$) arc[start angle=173, end angle=130, radius=.8em];
\draw[dash pattern=on 0pt off 1.2pt, line width=.6pt, line cap=round] ($(9,0)+(45:.8em)$) arc[start angle=45, end angle=5, radius=.8em];
\end{scope}
\begin{scope}[yshift=-19.5em]
\node[font=\scriptsize] at (-2,0) {\rotatebox{-90}{$\underbrace{\hspace{2.5em}}$}};
\node[font=\footnotesize, left=1ex, align=right] at (-2,0) {(\VI) \it $s_0$ boundary stops \\ \it at component $\partial_0 S$};
\foreach \label/\x in {2/1.5,4/3.833,6/6.167,7/6.667,9/9} {
\node[shape=circle, scale=.7] (\label) at (\x,0) {};
}
\foreach \x in {1.5,3.833,9} {
\draw[line width=1pt, fill=black] (\x,0) circle(0.2ex);
}
\path[->, line width=.5pt, font=\scriptsize] (2) edge node[below=-.5ex] {$p^\VI_2$} (4);
\path[->, line width=.5pt, font=\scriptsize] (4) edge node[below=-.5ex] {$p^\VI_3$} (6);
\node[font=\scriptsize] at (6.417,0) {...};
\path[->, line width=.5pt, font=\scriptsize] (7) edge node[below=-.5ex] {$p^\VI_{s_0-1}$} (9);
\draw[dash pattern=on 0pt off 1.2pt, line width=.6pt, line cap=round] ($(1.5,0)+(172:.8em)$) arc[start angle=172, end angle=6, radius=.8em];
\draw[dash pattern=on 0pt off 1.2pt, line width=.6pt, line cap=round] ($(3.833,0)+(172:.8em)$) arc[start angle=172, end angle=6, radius=.8em];
\end{scope}
\end{tikzpicture}
\caption{The quiver for the DG endomorphism algebra for the generator of $\mathcal W (\mathbf S)$ corresponding to the dissection of Fig.~\ref{fig:dissection}. The dotted lines indicate the quadratic monomial relations. The differential and grading are given in Proposition \ref{proposition:generator}.}
\label{fig:quiver}
\end{figure}

\begin{proposition}
\label{proposition:generator}
Let $\mathbf S = (S, \Sigma, \eta)$ be a graded orbifold surface with stops with at least one boundary component $\partial_0 S$ containing at least one boundary stop. We denote by
\begin{itemize}
\item[\textup{(\I)}] $g$ the genus of $S$
\item[\textup{(\II)}] $k$ the number of orbifold points
\item[\textup{(\III)}] $l$ the number of additional boundary components with boundary stops (not including $\partial_0 S$), denoted each by $\partial^\III_1 S, \dotsc, \partial^\III_l S$, and for each $1 \leq j \leq l$, the boundary component $\partial^\III_j S$ contains $s_j \geq 1$ boundary stops
\item[\textup{(\IV)}] $m$ the number of boundary components with full boundary stops, denoted by $\partial^\IV_1 S, \dotsc, \partial^\IV_m S$
\item[\textup{(\V)}] $n$ the number of boundary components without stops, denoted by $\partial^\V_1 S, \dotsc, \partial^\V_n S$
\item[\textup{(\VI)}] $s_0 \geq 1$ the number of boundary stops of $\partial_0 S$.
\end{itemize}

Then $\mathcal W (\mathbf S)$ admits a formal (classical) generator $\Gamma$ whose DG endomorphism algebra $\End (\Gamma)$ is the DG gentle algebra given by the graded quiver of Fig.~\ref{fig:quiver} with quadratic monomial relations
\begin{flalign*}
&&&& q^\I_i p^\I_i, \; r^\I_i q^\I_i, \; u^\I_i r^\I_i &\quad (1 \leq i \leq g) & p^\I_{i+1} u^\I_i &\quad (1 \leq i < g) && \\
&& q^\II_1 u^\I_g && \eqmakebox[pqr][r]{$u^\II_i q^\II_i$} &\quad (1 \leq i \leq k) & q^\II_{i+1} u^\II_i &\quad (1 \leq i < k) && \\
&& q^\III_1 u^\II_k && \eqmakebox[pqr][r]{$u^\III_i q^\III_i$} &\quad (1 \leq i \leq l) & q^\III_{i+1} u^\III_i &\quad (1 \leq i < l) && \\
&& u^\IV_1 u^\III_l && \eqmakebox[pqr][r]{${p^\IV_i}^2$} &\quad (1 \leq i \leq m) & u^\IV_{i+1} u^\IV_i &\quad (1 \leq i < m) && \\
&& q^\V_1 u^\IV_m && \eqmakebox[pqr][r]{$u^\V_i q^\V_i$} &\quad (1 \leq i \leq n) & q^\V_{i+1} u^\V_i &\quad (1 \leq i < n) && \\
&& p^\VI_1 q^\V_n && \eqmakebox[pqr][r]{$p^\VI_{i+1} p^\VI_i$} &\quad (1 \leq i \leq s_0 - 2) &&&&
\end{flalign*}
and differential
\begin{align*}
d p^\II_i = q^\II_i \quad (1 \leq i \leq k)
\end{align*}
such that $d$ is zero on all other arrows. The grading of the arrows is given by
\begin{flalign*}
&& \lvert q_i^\II \rvert  &= \eqmakebox[gr1][l]{$1$}                              \quad (1 \leq i \leq k) &
   \lvert q_i^\III \rvert &= \eqmakebox[gr2][l]{$\w_\eta (\partial_i^\III S)$}    \quad (1 \leq i \leq l) && \\
&& \lvert p_i^\IV \rvert  &= \eqmakebox[gr1][l]{$1 - \w_\eta (\partial_i^\IV S)$} \quad (1 \leq i \leq m) &
   \lvert q_i^\V \rvert   &= \eqmakebox[gr2][l]{$\w_\eta (\partial_i^\V S)$}      \quad (1 \leq i \leq n) &&
\end{flalign*}
all other arrows being graded in degree $0$.
\end{proposition}

\begin{proof}
The curves in Fig.~\ref{fig:dissection} define a formal DG dissection in the sense of \cite[Section~8]{barmeierschrollwang}. Hence $\Gamma$ generates $\mathcal W (\mathbf S)$ and by \cite[Proposition 8.4]{barmeierschrollwang} its endomorphism algebra $\End (\Gamma)$ is a DG algebra which by \cite[Theorem 8.6]{barmeierschrollwang} is formal, i.e.\ A$_\infty$-quasi-isomorphic to its (associative) cohomology algebra. The description of $\End (\Gamma)$ as a DG quiver with relations is given in op.\ cit.\ as well.
\end{proof}

We will see in Section~\ref{subsection:hochschildwrapped} that the cohomology $\HH^2 (\End (\Gamma), \End (\Gamma))$ is spanned by $2$-cocycles giving rise to deformations of $\End (\Gamma)$ as a curved DG algebra. This is due to the particular choice of generator. For other generators, the $2$-cocycles of its endomorphism algebra naturally parametrize deformations as a curved A$_\infty$ algebra which generally feature nontrivial higher multiplications. Such deformations are described in \cite{barmeierschrollwang} (see also Remark \ref{remark:general} for A$_\infty$ deformations and Section~\ref{section:curved} for how to deal with curved deformations).

\subsubsection{Classical and weak generators}
\label{subsubsection:weak}

Recall that a {\it weak generator} of a triangulated category $\mathcal C$ is an object $E$ such that for any other object $X$, there exists a nontrivial morphism from $E$ to some shift of $X$, i.e.\ $\Hom (E, X [n]) \neq 0$ for some $n \in \mathbb Z$. If the smallest triangulated subcategory of $\mathcal C$ containing $E$ is all of $\mathcal C$, then $E$ is called a {\it classical generator}.

Using the classification of indecomposable objects and morphisms in partially wrapped Fukaya categories of surfaces \cite{haidenkatzarkovkontsevich,opperplamondonschroll}, it is possible to characterize geometrically when a partial dissection constitutes a weak generator. This difference only occurs when the category has boundary components without stops and of winding number $\neq 0$. In other words, to see the difference between classical and weak generators, one needs to consider gentle or skew-gentle algebras which are non-proper and not concentrated in degree $0$.

It turns out that this difference allows us to replace a classical generator by a weak generator and pass to certain unbounded twisted complexes to recover $\mathcal W (\mathbf S)$ (see Section~\ref{subsection:unbounded}).

\begin{theorem}
\label{theorem:weak}
Let $\mathbf S$ be a graded orbifold surface and let $\Gamma$ be a collection of curves in $\mathbf S$ that do not intersect in the interior of $\mathbf S$. Then $\Gamma$ is a weak generator of the derived partially wrapped Fukaya category $\mathrm D \mathcal W (\mathbf S)$ if and only if the curves in $\Gamma$ cut the surface into connected components such that each connected component is of one of the following types:
\begin{enumerate}
\item a smooth disk containing at most one stop
\item an orbifold disk containing no boundary stops 
\item a smooth annulus containing a full boundary stop and no further boundary stops
\item a smooth annulus containing no stops and a boundary component of winding number $\neq 0$.
\end{enumerate}
If there are no connected components of the last type, then $\Gamma$ is a classical generator.
\end{theorem}

\begin{proof}
Let $\Gamma = \{ \gamma_1, \dotsc, \gamma_n \}$ and write $S \smallsetminus \Gamma = \bigsqcup_i C_i$. For the ``only if'' direction one observes that as soon as a connected component $C_i$ contains anything else, it is possible to construct an object $X$ represented by a curve lying entirely inside $C_i$ such that there is no morphism from $\Gamma$ to $X [n]$, violating the assumption of weak generation on $\Gamma$. For the ``if'' direction, observe that the conditions on the connected components $\Gamma$ imply that $\Gamma$ is a classical generator, given by a dissection of $\mathbf S$ as considered in \cite{haidenkatzarkovkontsevich,barmeierschrollwang}, unless we have components of the fourth type, in which case the ``missing'' objects which are not classically generated by $\Gamma$, are curves connecting to the boundary component without stops. But these curves must connect to a different boundary component inside the same connected component, or cross one of the arcs cutting out the component. In both these cases, there exists a nontrivial morphism of some degree, showing that $\Gamma$ satisfies the condition of being a weak generator.
\end{proof}

\begin{remark}[Classification of weak generators]
We expect Theorem \ref{theorem:weak} to hold in general, namely we expect that the conditions give a full characterization of weak generators, i.e.\ that any weak generator of $\mathrm D\mathcal W (\mathbf S)$ is given by an object whose corresponding curves dissect the surface into connected components satisfying the conditions {\itemi}--{\itemiv}, without the assumption on intersections in the interior. Note that this general result would imply that for any graded gentle or skew-gentle algebra $A$ that is either proper or (possibly nonproper and) concentrated in degree $0$, any weak generator of $\per (A)$ is a classical generator. 
\end{remark}

\section{Hochschild cohomology, deformations and families of DG algebras}

We now recall some basic notions on the Hochschild cohomology of DG algebras needed to compute the Hochschild cohomology of $\mathcal W (\mathbf S)$. We first fix notations and sign conventions as follows.

\begin{notation}
\label{notation:gvs}
Let $(V, d_V)$ be a cochain complex of $\Bbbk$-vector spaces, i.e.\ $V$ is a $\mathbb Z$-graded $\Bbbk$-vector space and $d_V \colon V \to V$ is a differential of degree $1$. We denote the degree $i$ component of $V$ by $V^i$ and by $\s V$ the $1$-shifted complex with components $(\s V)^i = V^{i+1}$ and differential $d_{\s V} (\s v) = -\s d_V (v)$. For a homogeneous element $v \in V^i$ we denote by $\lvert v \rvert = i$ its degree. The degree of the corresponding element $\s v \in \s V$ is given by $\lvert \s v \rvert = \lvert v \rvert - 1$. (One may also view $\s V$ as the tensor product $\Bbbk [1] \otimes V$ where $\Bbbk [1]$ is the $1$-dimensional complex concentrated in degree $-1$ so that $\s v$ may be identified with $1 \otimes v$, where $1 \in \Bbbk$ is the unit of $\Bbbk$ but placed in degree $-1$.)

We often suppress the differential from the notation and simply write $V$ for a cochain complex with the differential $d_V$ being implicit.

For two cochain complexes $U$ and $V$, let $\Hom (U, V)$ be the cochain complex whose degree $i$ component $\Hom (U, V)^i$ is the $\Bbbk$-vector space of $\Bbbk$-linear maps of degree $i$. Its differential is given by the graded commutator
\begin{equation}
\label{eq:dhom}
d_{\Hom (U, V)} (f) = d_V \circ f - (-1)^{\lvert f \rvert} f \circ d_U
\end{equation}
where the sign is given by the Koszul sign rule.

Similarly, the tensor product of two cochain complexes $U$ and $V$ is the cochain complex $U \otimes V$ where
\[
(U \otimes V)^i = \bigoplus_{j \in \mathbb Z} U^{i-j} \otimes V^j
\]
and differential
\begin{equation}
\label{eq:tensordifferential}
d_{U \otimes V} (u \otimes v) = d_U (u) \otimes v + (-1)^{\lvert u \rvert} u \otimes d_V (v). 
\end{equation}
(This also explains the sign for the differential of $\s V = \Bbbk [1] \otimes V$.)
\end{notation}

\subsection{Bar resolution}

Let $A$ be a DG algebra and denote by $A^\e = A \otimes A^{\mathrm{op}}$ its enveloping algebra, where the opposite DG algebra $A^{\mathrm{op}}$ has the same underlying graded cochain complex as $A$ with multiplication $b \cdot^{\mathrm{op}} a = (-1)^{\lvert a \rvert \lvert b \rvert} a b$. The {\it bar resolution} of $A$ is the left DG $A^\e$-module 
\begin{equation}
\label{eq:bar}
\Bar(A) = \bigoplus_{n \geq 0} A\otimes (\s A)^{\otimes n} \otimes A.
\end{equation}
For $1 \leq i \leq j \leq n$ we use the shorthand
\begin{equation}
\label{eq:shorthand}
\s a_{i \dotsc j} := \s a_i \otimes \s a_{i+1} \otimes \cdots \otimes \s a_j
\end{equation}
for elements in $j - i + 1$ tensor factors of $(\s A)^{\otimes n}$. Since $\s a_{i \dotsc j}$ contains $j - i + 1$ instances of $\s$, its degree is
\[
\lvert \s a_{i \dotsc j} \rvert = \lvert a_i \rvert + \lvert a_{i+1} \rvert + \dotsb + \lvert a_j \rvert - (j - i + 1).
\]
The differential of $\Bar (A)$ can then be expressed as the sum $d_{\Bar (A)} = d^{\,\mathrm{int}} + d^{\,\mathrm{ext}}$ of the ``internal'' differential induced from $A$ via \eqref{eq:tensordifferential}
\begin{equation*}
\begin{split}
d^{\,\mathrm{int}} (a_0 \otimes \s a_{1 \dotsc n} \otimes a_{n+1}) &= d_A (a_0) \otimes \s a_{1 \dotsc n} \otimes a_{n+1} \\ &\quad + \sum_{1 \leq i \leq n} (-1)^{\lvert a_0 \rvert + \lvert \s a_{1 \dotsc i-1} \rvert} a_0 \otimes \s a_{1 \dotsc i-1} \otimes d_{\s A} (\s a_i) \otimes \s a_{i+1 \dotsc n} \otimes a_{n+1} \\ &\quad + (-1)^{\lvert a_0 \rvert + \lvert \s a_{1 \dotsc n} \rvert} a_0 \otimes \s a_{1 \dotsc n} \otimes d_A (a_{n+1})
\end{split}
\end{equation*}
and the ``external'' differential of the bar construction
\begin{equation*}
\begin{split}
d^{\,\mathrm{ext}} (a_0 \otimes \s a_{1 \dotsc n} \otimes a_{n+1}) &= (-1)^{\lvert a_0 \rvert} a_0 a_1 \otimes \s a_{2 \dotsc n} \otimes a_{n+1} \\
 &\quad + \sum_{1 \leq i < n} (-1)^{\lvert a_0 \rvert + \lvert \s a_{1 \dotsc i} \rvert} a \otimes \s a_{1 \dotsc i-1} \otimes \s a_i a_{i+1} \otimes \s a_{i+1 \dotsc n} \otimes a_{n+1} \\
&\quad - (-1)^{\lvert a_0 \rvert + \lvert \s a_{1 \dotsc n-1} \rvert} a_0 \otimes \s a_{1\dotsc n-1} \otimes a_n a_{n+1}.
\end{split}
\end{equation*}

The graded $A^\e$-module structure on $A\otimes (\s A)^{\otimes n} \otimes A$ is given by 
\[
(a \otimes b) (a_0 \otimes \s a_{1 \dotsc n} \otimes a_{n+1}) := (-1)^{\lvert b \rvert (\lvert a_0 \rvert + \lvert \s a_{1 \dotsc n} \rvert + \lvert a_{n+1} \rvert)} a a_0 \otimes \s a_{1 \dotsc n} \otimes a_{n+1} b.
\]
There is a natural morphism of DG $A^\e$-modules $\varepsilon \colon  \Bar(A)\rightarrow A$ given by the composition
\begin{align}\label{equ:bar}
\Bar(A) \stackrel{\mathrm{pr}_0} \longrightarrow A \otimes A \stackrel{m_2}\longrightarrow A
\end{align}
where $\mathrm{pr}_0$ is the canonical projection and $m_2$ is the multiplication of $A$. It is well known that $\varepsilon$ is a quasi-isomorphism.

Since $\Bar(A)$ is a DG projective resolution of the left DG $A^\e$-module $A$, the cochain complex $\Hom_{A^\e} (\Bar(A), A)$ computes the Hochschild cohomology of $A$. Its differential is given by $d (f):= d_A \circ f - (-1)^{\lvert f \rvert} f \circ d_{\Bar (A)}$ as in \eqref{eq:dhom}. The natural isomorphism 
\begin{align}
\label{identification-bimodule}
\Hom((\s A)^{\otimes n}, A) \stackrel{\sim}\longrightarrow \Hom_{A^\e}(A\otimes (\s A)^{\otimes n}\otimes A, A)
\end{align}
which sends $f$ to the map $\widetilde f$ defined by 
\[
a_0 \otimes \s a_{1\dotsc n} \otimes a_{n+1} \mapsto (-1)^{\lvert a_0 \rvert \lvert f \rvert} a_0 f(\s a_{1\dotsc n}) a_{n+1}.
\]
induces an isomorphism of cochain complexes between $\Hom_{A^\e} (\Bar (A), A)$ and the classical Hochschild cochain complex $\mathrm C^\bullet (A, A)$ of $A$.
 
\subsection{Bardzell resolution for DG monomial algebras}
\label{subsection:bardzellresolution}

The bar resolution \eqref{eq:bar} is often too large to be of use for explicit computations. For finitely generated algebras, a workable replacement is the resolution associated to a reduction system, see \cite{bardzell,chouhysolotar} and \cite{barmeierwang1,barmeierwang2,barmeierwang3} for its role in deformation theory. We now explain how Bardzell's minimal resolution \cite{bardzell} for (trivially graded) monomial algebras can be generalized to the DG setting.

Let $Q$ be a (finite) quiver. A {\it grading} on $Q$ is given by a function $\lvert - \rvert \colon Q_1 \to \mathbb Z$ assigning any arrow $x \in Q_1$ its degree $\lvert x \rvert$. A path $p = p_n \dotsb p_2 p_1$ of length $n$ is given by a sequence of $n$ composable arrows $p_1, \dotsc, p_n \in Q_1$ and we denote by $Q_n$ the set of paths of length $n$. (In particular, $Q_0$ corresponds to the vertices of $Q$.) The grading on $Q_1$ extends to $Q_n$ by setting
\[
\lvert p \rvert = \lvert p_1 \rvert + \dotsb + \lvert p_n \rvert
\]
which turns the path algebra $\Bbbk Q$ into a graded algebra. Note that $\Bbbk Q$ is naturally bigraded since each path has both a (path) length and a degree.

Let $I$ be a homogeneous ideal of $\Bbbk Q$. Since we will only need the quadratic monomial case, we shall assume that $I$ is generated by a subset $R \subset Q_2$. (The general monomial case is similar and the general non-monomial case follows for example by homological perturbation, see \cite{barmeierwang1}.) Let $A = \Bbbk Q / I$ be the graded quotient algebra. 

Set $W_0 = Q_0$ and $W_1 = Q_1$ and for $n \geq 2$ let 
\[
W_n = \{ w_1 w_2 \dotsb w_n \mid w_1, \dotsc, w_n \in Q_1, \; w_i w_{i+1} \in R \text{ for each $1 \leq i < n$} \} \subset Q_n.
\]
In particular, we have $W_2 = R$. An element of $W_3$ is called an overlap ambiguity and an element of $W_n$ for $n \geq 4$ a higher overlap ambiguity. We call $w \in W_n$ for $n \geq 1$ a {\it maximal overlap ambiguity} or simply {\it maximal overlap} if $w$ is not a subpath of any $w' \in W_m$ for $m > n$.

Now consider the following left DG $A^\e$-module 
\[
\mathrm P_\bullet (Q, R) = \bigoplus_{n \geq 0} A \otimes \s^n \Bbbk W_n \otimes A
\]
where $\otimes = \otimes_{\Bbbk Q_0}$ and $\s^n \Bbbk W_n$ is the $n$th shift of the graded vector space spanned by $W_n$. Analogous to \eqref{eq:shorthand} we shall use the shorthand
\[
\s w_{i \dotsc j} := \s^{j-i+1} w_i w_{i+1} \dotsb w_j
\]
for $i \leq j$ and $w_i w_{i+1} \dotsb w_j \in W_{j-i+1}$ and write elements in $A \otimes \s^n \Bbbk W_n \otimes A$ as
\[
a \otimes \s w_{1 \dotsc n} \otimes b 
\] 
for $a, b \in A$ and $w = w_1 w_2 \dotsb w_n \in W_n$. The differential $\partial^{\,\mathrm{ext}} $ is given as 
\begin{equation}
\begin{split}
\partial^{\,\mathrm{ext}}_n (a_0 \otimes \s w_{1 \dotsc n} \otimes a_{n+1}) &= (-1)^{\lvert a_0 \rvert} a_0 w_1 \otimes \s w_{2 \dotsc n} \otimes a_{n+1} \\&\quad - (-1)^{\lvert a_0 \rvert + \lvert \s w_{1 \dotsc n-1} \rvert} a \otimes \s w_{1 \dotsc n-1} \otimes w_n a_{n+1}.
\end{split}
\end{equation}
Note that we have 
\[
\partial^{\,\mathrm{ext}}_1 (a \otimes \s x \otimes  b) = (-1)^{\lvert a \rvert} (a x \otimes b - a \otimes x b)
\]
for any arrow $x \in Q_1$ and elements $a, b \in A$. Then $\mathrm P_\bullet (Q, R)$ together with $\partial^{\,\mathrm{ext}}$ is a DG projective resolution of $A$. 

If $(A, d)$ is a DG algebra it is possible to perturb the differential $\partial^{\,\mathrm{ext}}$ of $\mathrm P_\bullet (Q, R)$ to obtain a DG projective resolution of $(A, d)$. The general formula for $\partial^{\,\mathrm{ext}}$ is rather complicated. However, it simplifies if we assume that $d(w) =0$ whenever there is a relation through $w$, namely if either $w w' \in I$ or $w'' w\in I$ for some $w', w''\in Q_1$. This simplifying assumption suffices for our purposes, as it holds if $(A, d)$ is a {\it DG gentle algebra} (see Definition \ref{definition:dggentle}). In this case, the internal differential $\partial^{\,\mathrm{int}}$ is simply given as   
\begin{multline*}
\partial^{\,\mathrm{int}}_n (a_0 \otimes \s w_{1 \dotsc n} \otimes a_{n+1}) \\
  {} = d(a_0) \otimes \s w_{1 \dotsc n} \otimes a_{n+1} + (-1)^{|a_0|+|\s w_{1 \dotsc n}|} a_0 \otimes \s w_{1 \dotsc n} \otimes d(a_{n+1})
\end{multline*} 
for any $n \geq 0$. Denote $\partial = \partial^{\,\mathrm{int}} + \partial^{\,\mathrm{ext}}$. Then we have the following result.

\begin{proposition}
Let $(A, d)$ be a quadratic monomial DG algebra such that $d$ vanishes on arrows appearing in the quadratic relations, i.e.\ writing $A = \Bbbk Q / I$ we have $d (w) = 0$ and $d (w') = 0$ for any $ww' \in I$ with $w, w'\in Q_1$.

Then $\mathrm P_\bullet (Q, R)$ together with the differential $\partial$ is a DG projective resolution of $A$ as a left DG $A^\e$-module.
\end{proposition}

\begin{proof}
First we note that $\partial^2 = (\partial^{\,\mathrm{ext}} + \partial^{\,\mathrm{int}})^2 = 0$. That is, $(\mathrm P_\bullet (Q, R), \partial)$ is a left DG $A^\e$-module. We may view $\mathrm P^\bullet (Q, R)$ as the total complex of a double complex whose vertical complexes are $A \otimes \s^n \Bbbk W_n \otimes A$ with the internal differential. It is known that each horizontal complex with the external (Hochschild) differential $\partial^{\,\mathrm{ext}}$ is exact. Then it follows from the Acyclic Assembly Lemma (see e.g.\ \cite[Section~2.7]{weibel}) that the total complex $(\mathrm P_\bullet (Q, R), \partial)$ is also exact. 
\end{proof}

We may use the projective resolution $\mathrm P_\bullet (Q, R)$ to compute the Hochschild cohomology of graded gentle algebras $A$. Using the observation that
\[
\Hom_{A^\e} (A \otimes \s^i \Bbbk W_i \otimes A, A) \simeq \Hom_{\Bbbk Q_0^\e} (\s^i \Bbbk W_i, A)
\]
one has that $\HH^\bullet (A, A)$ is isomorphic to the cohomology of the following complex 
\begin{align}\label{algin:smallcomplex}
\mathrm P^\bullet (Q, R) = \Biggl( \prod_{i \geq 0} \Hom_{\Bbbk Q_0^\e} (\s^i \Bbbk W_i, A), \delta \Biggr)
\end{align}
where the differential $\delta$ is given by 
\begin{align}\label{differentialsmaller}
\delta(f) (\s^{i+1} w) = -(-1)^{\lvert \s w_1 \rvert \lvert f \rvert} w_1 f (\s w_{2 \dotsc i+1}) + (-1)^{\lvert \s w_{1 \dotsc i} \rvert -  \lvert f \rvert} f (\s w_{1 \dotsc i}) w_{i+1}
\end{align}
for any $f \in \Hom_{\Bbbk Q_0^\e} (\s^i \Bbbk W_i, A)$ and any $w = w_1 w_2 \dotsb w_{i+1} \in W_{i+1}$.

\subsection{Deformations of pretriangulated DG and \texorpdfstring{A$_\infty$}{A-infinity} categories}
\label{subsection:keller}

Deformations of algebras and categories are controlled by a suitable notion of the Hochschild complex \cite{kontsevichsoibelman}, with $\HH^2$ being identified with the set of first-order deformations up to equivalence.

A result of Keller \cite{keller03} shows that the Hochschild cohomology of a DG category is isomorphic to the Hochschild cohomology of its DG category of twisted complexes. This isomorphism can be seen as a shadow of a much stronger general result: for any (small) DG category $\mathcal A$, the Hochschild complexes of $\mathcal A$ and $\tw (\mathcal A)$ are isomorphic in the homotopy category of B$_\infty$ algebras \cite{keller03} and hence in the homotopy category of DG Lie or L$_\infty$ algebras. (The B$_\infty$ structure includes the Gerstenhaber bracket which controls the deformation theory.) One may thus study the deformation theory of $\tw (\mathcal A)$ via the deformation theory of $\mathcal A$ itself. (See \cite{lowenvandenbergh1,lowenvandenbergh2} for analogous results for deformations of Abelian categories.)

For $\mathcal W (\mathbf S)$, we take $\mathcal A = \add \End (\Gamma)$ for $\Gamma$ the standard generator of $\mathrm D \mathcal W (\mathbf S)$ in Fig.~\ref{fig:dissection}. The results of \cite{keller03} then yield the following isomorphism
\begin{equation}
\label{equation:isomorphismhh}
\HH^\bullet (\mathcal W (\mathbf S), \mathcal W (\mathbf S)) \simeq \HH^\bullet (\End (\Gamma), \End (\Gamma))
\end{equation}
(See \cite[Section~2.4]{perutzsheridan} for the case when $\mathcal A$ is an A$_\infty$ category.) Moreover, the full deformation theory of $\mathcal W (\mathbf S)$ may be described via the deformation theory of $\End (\Gamma)$. When $\Gamma$ is a formal generator, $\End (\Gamma)$ is a DG algebra and its Hochschild complex generally controls deformations of $\End (\Gamma)$ as a curved A$_\infty$ algebra. In Section~\ref{subsection:hochschildwrapped} we will show that for the standard generator of Proposition \ref{proposition:generator}, there is a standard basis of $\HH^2 (\End (\Gamma), \End (\Gamma))$ consisting of $2$-cocycles which parametrize deformations of $\End (\Gamma)$ as a curved DG algebra. We show that it even suffices to choose $\Gamma$ to be weak generator of $\mathcal W (\mathbf S)$ (see Section~\ref{section:curved}) in which case it suffices to deform $\End (\Gamma)$ as (uncurved) DG algebra.

\subsection{Families of DG algebras}
\label{subsection:families}

Whereas the Hochschild complex of a DG algebra $A$ can be seen to parametrize formal or infinitesimal deformations of $A$, we shall consider a slightly stronger notion of deformation of DG algebras, namely that of families of DG algebras ``parametrized'' by some affine scheme $\Spec R$. Given such a family, we can consider the completion at a closed point of $\Spec R$ which gives a formal deformation of the fiber of the family at this point. Conversely, such a family can be viewed as an algebraization of the formal family.\footnote{This is an unfortunate clash of standard terminology. A {\it formal generator} $\Gamma$ is a generator whose endomorphism DG algebra is A$_\infty$-quasi-isomorphic to its (associative) $\Ext$-algebra $A = \Ext^\bullet (\Gamma, \Gamma)$ without higher structures. A {\it formal deformation} of $A$ is a deformation over a complete local Noetherian $\Bbbk$-algebra $R$, i.e.\ here ``formal'' refers to the completeness of $R$, and not to a statement about the vanishing of higher structures.}

The passage from a formal deformation to an algebraic family is generally highly nontrivial. However, for DG or A$_\infty$ algebras there are additional hurdles for obtaining a good formal theory of deformations in the first place: the {\it curvature problem} presents an obstacle for defining a suitable notion of equivalence of deformations. Fortunately, for deformations of $\mathcal W (\mathbf S)$ we are able to surmount all of these difficulties. Even though $\mathcal W (\mathbf S)$ naturally admits curved deformations when described via a classical generator, we bypass the curvature problem by giving an equivalent description in terms of a weak generator (see Section~\ref{section:curved}) and then construct an algebraic family.

We shall give a detailed exposition of our notion of family, including the associated Kodaira--Spencer map to Hochschild cohomology which is used to introduce the notions of a versal or semi-universal family.

For our applications, deformations of $\mathcal W (\mathbf S)$ turn out to be unobstructed and parametrized by a finite number of parameters, whence it suffices to work over the base algebra $R = \Bbbk [x_1, \dotsc, x_d]$, although the same definitions make sense for any commutative Noetherian $\Bbbk$-algebra $R$. We shall assume that $\Bbbk$ is algebraically closed and of characterstic $0$, so that the maximal ideals of $R = \Bbbk [x_1, \dotsc, x_d]$ are of the form $\mathfrak m_\lambda = (x_1 - \lambda_1, \dotsc, x_d - \lambda_d)$ for some $d$-tuple $\lambda = (\lambda_1, \dotsc, \lambda_d) \in \Bbbk^d$. Of course $\Spec R$ is affine $d$-space $\mathbb A^d$ and we identify maximal ideals of $R$, closed points of $\Spec R$ and elements $\lambda \in \Bbbk^d$ as usual.

\subsubsection{DG algebras}
\label{subsubsection:dgalgebras}

Let $V$ be a $\mathbb Z$-graded $\Bbbk$-vector space whose degree $i$ component we denote by $V^i$ as in Notation \ref{notation:gvs}. For sign reasons, it is convenient to define DG or A$_\infty$ structures on the shifted vector space $\s V$ whose degree $i$ component is $V^{i+1}$ and whose elements we denote by $\s v$ for $v \in V$.  This shift ensures that all signs arise from the Koszul sign rule. A DG algebra structure on $V$ is then given by a pair of degree $1$ maps
\begin{equation}
\label{eq:m1m2}
\mu^1 \in \Hom (\s V, \s V) \qquad \text{and} \qquad \mu^2 \in \Hom ((\s V)^{\otimes 2}, \s V)
\end{equation}
satisfying the (shifted) A$_\infty$ equations
\begin{equation}
\label{eq:ainfinityequations}
\begin{aligned}
\mu^1 (\mu^1 (\s v)) &= 0 \\
\mu^2 (\s v, \mu^1 (\s u)) + \mu^1 (\mu^2 (\s v, \s u)) + (-1)^{\lvert \s u \rvert} \mu^2 (\mu^1 (\s v), \s u) &= 0 \\
\mu^2 (\s w, \mu^2 (\s v, \s u)) + (-1)^{\lvert \s u \rvert} \mu^2 (\mu^2 (\s w, \s v), \s u) &= 0
\end{aligned}
\end{equation}
for any $u, v, w \in V$, where $\lvert \s u \rvert = \lvert u \rvert - 1$ is the shifted degree of $u$. Of course, these three equations encode that $\mu^1$ is a differential which satisfies the graded Leibniz rule with respect to the graded associative multiplication given by $\mu^2$. (See \cite[Section~2.1]{barmeierschrollwang} for more details, in particular \cite[Remark 2.4]{barmeierschrollwang}, for recovering the usual sign and degree conventions of a DG algebra.) We usually write $A = (V, \mu^1, \mu^2)$ for a DG algebra and without further qualification we always mean a DG $\Bbbk$-algebra.

\subsubsection{Families of DG algebras}

Now let $R$ be a commutative Noetherian $\Bbbk$-algebra. A {\it family of DG algebras over $R$} is a DG $R$-algebra $\widetilde A = (V \otimes R, \widetilde \mu^1, \widetilde \mu^2)$, defined as in Section~\ref{subsubsection:dgalgebras}, but where everything is not $\Bbbk$-linear, but $R$-linear, in particular $\Hom$'s and $\otimes$'s are over $R$. More precisely, $V$ is a graded $\Bbbk$-vector space so that $V \otimes R$ is a free graded $R$-module and $\widetilde \mu^1 \in \Hom_R (\s V \otimes R, \s V \otimes R)$ and $\widetilde \mu^2 \in \Hom_R ((\s V \otimes R)^{\otimes_R 2}, \s V \otimes R)$ are required to satisfy the analogue of \eqref{eq:ainfinityequations}.

Note that $R$-linearity implies that $\widetilde \mu^1$ and $\widetilde \mu^2$ are determined by their values on $\s V$ or $(\s V)^{\otimes 2}$, respectively, i.e.\ we may view $\widetilde \mu^1$ and $\widetilde \mu^2$ as maps
\[
\widetilde \mu^1 \in \Hom (\s V, \s V \otimes R), \qquad \widetilde \mu^2 \in \Hom ((\s V)^{\otimes 2}, \s V \otimes R)
\]
i.e.\ as maps $\s V \to \s V$ or $(\s V)^{\otimes 2} \to \s V$ but with coefficients given by elements of $R$ which for us will simply be polynomials in $x_1, \dotsc, x_d$.

A family of DG algebras is indeed a family of DG $\Bbbk$-algebras in the naive sense whence we also write a family over $R$ as $\{ A_\lambda \}_{\lambda \in \Spec R}$. Namely, for $R = \Bbbk [x_1, \dotsc, x_d]$ and every $\lambda = (\lambda_1, \dotsc, \lambda_d) \in \Bbbk^d$, we may write $\Bbbk_\lambda = R / \mathfrak m_\lambda \simeq \Bbbk$ for the $1$-dimensional $R$-module on which each $x_i$ acts by multiplication by $\lambda_i$. For each $\lambda \in \Bbbk^d$, applying the functor $- \otimes_R \Bbbk_\lambda$ turns a graded $R$-module into a graded $\Bbbk$-vector space and a DG $R$-algebra $\widetilde A$ into the DG algebra $\widetilde A \otimes_R \Bbbk_\lambda \simeq A_\lambda = (V, \mu^1_\lambda, \mu^2_\lambda)$. Here, $\mu^1_\lambda = \widetilde \mu^1 \otimes_R \Bbbk_\lambda$ and $\mu^2_\lambda = \widetilde \mu^2 \otimes_R \Bbbk_\lambda$ are maps as in \eqref{eq:m1m2} obtained by evaluating $x_i \mapsto \lambda_i$ for all coefficient polynomials of $\widetilde \mu^1$ and $\widetilde \mu^2$.

Since $\widetilde \mu^1$ and $\widetilde \mu^2$ satisfy \eqref{eq:ainfinityequations}, so do $\mu^1_\lambda$ and $\mu^2_\lambda$, so that $A_\lambda = (V, \mu^1_\lambda, \mu^2_\lambda)$ is indeed a DG $\Bbbk$-algebra.

Note that $\Hom_R ((\s V \otimes R)^{\otimes_R n}, \s V \otimes R)$ is a right $R$-module via
\[
(\phi \cdot r) (x_1 \otimes \dotsb \otimes x_n) = \phi (x_1 \otimes \dotsb \otimes x_n) \cdot r
\]
so that we may identify $\mathrm C^\bullet_R (\widetilde A, \widetilde A) \otimes_R \Bbbk_\lambda \simeq \mathrm C^\bullet (A_\lambda, A_\lambda)$. That is, applying the functor $- \otimes_R \Bbbk_\lambda$ turns the $R$-linear Hochschild complex of the DG $R$-algebra $\widetilde A$ into the usual $\Bbbk$-linear Hochschild complex of the DG algebra $A_\lambda$. (Note that the $R$-linear Hochschild complex uses the $R$-linear DG structure $\widetilde \mu^1, \widetilde \mu^2$ which reduces to the $\Bbbk$-linear structure $\mu^1_\lambda, \mu^2_\lambda$ after tensoring by $\Bbbk_\lambda$.)

\subsubsection{Kodaira--Spencer map}

Given a family of DG algebras $\widetilde A = (V \otimes R, \widetilde \mu^1, \widetilde \mu^2)$ over $R = \Bbbk [x_1, \dotsc, x_d]$, its tangent space at a closed point $\lambda$ can be viewed as a $d$-parameter first-order deformation of $A_\lambda$. The Kodaira--Spencer map is the natural map that assigns to a tangent vector at $\lambda$ the cohomology class corresponding to this first-order deformation. (See for example \cite[Section~3.4]{iwanari} for the Kodaira--Spencer map in the context of deformations of stable $\infty$-categories.)

Concretely, under the $\Bbbk$-linear identification $R / \mathfrak m_\lambda^2 \simeq \Bbbk_\lambda \oplus (\mathfrak m_\lambda / \mathfrak m_\lambda^2)$ we may view the images of $\widetilde \mu^1, \widetilde \mu^2$ under $- \otimes_R R / \mathfrak m_\lambda^2$ as first-order deformations of $\mu^1_\lambda, \mu^2_\lambda$.

For example, if
\[
\widetilde \mu^1 (\s v) = \sum_{1 \leq i \leq k} \s w_i \otimes r_i
\]
for some $w_i \in V$ and $r_i \in R$, the Nullstellensatz allows us to write
\[
r_i = \underbrace{r_i (\lambda)}_{\in \Bbbk_\lambda} \oplus \underbrace{r_i - r_i (\lambda)}_{\in \mathfrak m_\lambda}.
\]
We may thus write the images of $\widetilde \mu^1, \widetilde \mu^2$ under $- \otimes_R(R / \mathfrak m_\lambda^2)$ as $\mu^1_\lambda + \widehat \mu^1_\lambda$ and $\mu^2_\lambda + \widehat \mu^2_\lambda$, respectively, where now
\begin{equation}
\label{eq:cocycles}
\widehat \mu^i_\lambda \in \Hom ((\s V)^{\otimes i}, \s V) \otimes (\mathfrak m_\lambda / \mathfrak m_\lambda^2), \qquad i = 1, 2.
\end{equation}
Viewing the first tensor factor as coefficients of $\mathfrak m_\lambda / \mathfrak m_\lambda^2 \simeq \Bbbk^d$, let us denote by $(\widehat \mu^i_\lambda)_j$ the coefficient of the $j$th basis vector of $\Bbbk^d$.

\begin{lemma}
The elements $((\widehat \mu^1_\lambda)_j, (\widehat \mu^2_\lambda)_j)$ for $1 \leq j \leq d$ define $2$-cocycles in $\mathrm C^\bullet (A_\lambda, A_\lambda) \otimes (\mathfrak m_\lambda / \mathfrak m_\lambda^2)$.
\end{lemma}

\begin{proof}
Since $\widehat \mu^1_\lambda$ and $\widehat \mu^2_\lambda$ are the infinitesimals of a $d$-parameter first-order deformation of $A_\lambda$, their components satisfy the cocycle condition in $\mathrm C^2 (A_\lambda, A_\lambda)$.
\end{proof}

Since $\mathfrak m_\lambda / \mathfrak m_\lambda^2$ is finite dimensional we have a natural isomorphism 
\begin{align}\label{align:identification}
\Hom ((\s V)^{\otimes i}, \s V) \otimes (\mathfrak m_\lambda / \mathfrak m_\lambda^2) \simeq \Hom( (\mathfrak m_\lambda / \mathfrak m_\lambda^2)^*,(\s V)^{\otimes i}, \s V)).
\end{align}  Note that we may identify the tangent space $\mathrm T_\lambda \mathbb A^d $ at $\lambda$ with the dual space $(\mathfrak m_\lambda / \mathfrak m_\lambda^2)^*$.  For each $\lambda$ we thus obtain a well-defined {\it Kodaira--Spencer map}
\begin{equation}
\label{eq:kodairaspencer}
\mathrm{KS}_\lambda \colon \mathrm T_\lambda \mathbb A^d \to \HH^2 (A_\lambda, A_\lambda), 
\end{equation}
which is the $\Bbbk$-linear map induced by the cohomology class of the cocycle $((\widehat \mu^1_\lambda)_j, (\widehat \mu^2_\lambda)_j)$. This shows that the Kodaira--Spencer map has a natural interpretation --- it is simply the natural map induced by the evaluation map $\mathfrak m_\lambda / \mathfrak m_\lambda^2 \otimes (\mathfrak m_\lambda / \mathfrak m_\lambda^2)^* \to \Bbbk$ where $\partial_{x_j-\lambda_j} = (x_j - \lambda_j)^*$ acts by evaluation on the right tensor factor of \eqref{eq:cocycles}.

\begin{definition}
We call a family $\widetilde A$ of DG algebras over $R$ {\it versal} resp.\ {\it semi-universal} at a closed point $\lambda \in \Spec R$, if the Kodaira--Spencer map \eqref{eq:kodairaspencer} is surjective resp.\ bijective.
\end{definition}

The following proposition shows that the notion of being versal or semi-universal at $\lambda$ is invariant under quasi-isomorphisms.

\begin{proposition}
Let $\widetilde A = (V\otimes R, \widetilde \mu^1, \widetilde \mu^2)$ and $\widetilde A' = (V'\otimes R, \widetilde{\mu}'{}^1, \widetilde{\mu}'{}^2)$ be two families of DG algebras over $R$. Then any $R$-linear quasi-isomorphism $\Phi\colon \widetilde A \to \widetilde A'$ of DG algebras induces a commutative diagram
\begin{align}\label{align:KSmaps}
\begin{tikzpicture}[baseline=-2.6pt]
\matrix (m) [matrix of math nodes, row sep=.25em, text height=1.5ex, column sep=3em, text depth=0.25ex, ampersand replacement=\&, inner sep=3pt]
{
\& \HH^2 (A_\lambda, A_\lambda) \\
\mathrm T_\lambda \mathbb A^d \& \\
\& \HH^2 (A'_\lambda, A'_\lambda)\mathrlap{.} \\
};
\path[->,line width=.4pt]
(m-2-1) edge node[font=\scriptsize, above] {$\mathrm{KS}_\lambda$} (m-1-2.west)
(m-2-1) edge node[font=\scriptsize, below=-.2ex] {$\mathrm{KS}_\lambda$} (m-3-2.west)
(m-1-2) edge node[font=\scriptsize, right=-.2ex] {$\simeq$} (m-3-2)
;
\end{tikzpicture}
\end{align}
As a result, the property of being versal or semi-universal at a closed point $\lambda \in \Spec R$ is invariant under quasi-isomorphisms of the family.
\end{proposition}

\begin{proof}
The morphism of DG $R$-algebras $\Phi\colon \widetilde A \to \widetilde A'$ induces a morphism of DG $\Bbbk$-algebras $\Phi_\lambda \colon A_\lambda \to A'_\lambda$. The latter induces two maps of complexes 
\[
\mathrm C^\bullet  (A_\lambda, A_\lambda)   \toarg{f} \mathrm C^\bullet   (A_\lambda, A'_\lambda) \leftarrowarg{g}   \mathrm C^\bullet  (A'_\lambda, A'_\lambda).
\]
which are quasi-isomorphisms.  
Then the vertical morphism in \eqref{align:KSmaps} is given by the composition 
\[
\HH^2  (A_\lambda, A_\lambda)  \toarg{f}  \HH^2   (A_\lambda, A'_\lambda)  \toarg{g^{-1}}  \HH^2 (A'_\lambda, A'_\lambda).
\]
It remains to show that the following diagram commutes 
\[
\begin{tikzpicture}[baseline=-2.6pt]
\matrix (m) [matrix of math nodes, row sep=0, text height=1.5ex, column sep=3em, text depth=0.25ex, ampersand replacement=\&, inner sep=3pt]
{
\& \HH^2 (A_\lambda, A_\lambda) \& \\
\mathrm T_\lambda \mathbb A^d \& \& \HH^2 (A_\lambda, A'_\lambda)\mathrlap{.} \\
\& \HH^2 (A'_\lambda, A'_\lambda) \& \\
};
\path[->,line width=.4pt]
(m-2-1) edge (m-1-2.west)
(m-2-1) edge (m-3-2.west)
(m-1-2.east) edge node[font=\scriptsize, above=-.1em] {$f$} (m-2-3.172)
(m-3-2.east) edge node[font=\scriptsize, below=-.1em] {$g$} (m-2-3.188)
;
\end{tikzpicture}
\]
Using the isomorphism in \eqref{align:identification} this follows from the following commutative diagram 
\[
\begin{tikzpicture}[baseline=-2.6pt]
\matrix (m) [matrix of math nodes, row sep=2em, text height=1.5ex, column sep=3em, text depth=0.25ex, ampersand replacement=\&, inner sep=3pt]
{
(\widetilde A  \otimes \widetilde A)  \otimes_R R / \mathfrak m_\lambda^2 \& \widetilde A  \otimes_R R / \mathfrak m_\lambda^2 \\
(\widetilde A' \otimes \widetilde A') \otimes_R R / \mathfrak m_\lambda^2 \& \widetilde A' \otimes_R R / \mathfrak m_\lambda^2\mathrlap{.} \\
};
\path[->,line width=.4pt]
(m-1-1) edge node[font=\scriptsize, right=-.1em] {$(\Phi \otimes \Phi) \otimes_R \id$} (m-2-1)
(m-1-1) edge (m-1-2)
(m-1-2) edge node[font=\scriptsize, right=-.1em] {$\Phi \otimes_R \id$} (m-2-2)
(m-2-1) edge (m-2-2)
;
\end{tikzpicture}
\]
The commutative diagram \eqref{align:KSmaps} shows that if one is surjective then so is the other. 
\end{proof} 

\section{DG gentle algebras}
\label{section:dggentle}

In this section we introduce the notion of a {\it DG gentle algebra} which can be used to model the partially wrapped Fukaya category $\mathcal W (\mathbf S)$ for any graded orbifold surface $\mathbf S$. It turns out that there is very little choice for turning a graded gentle algebra into a DG algebra. Firstly, the Leibniz rule implies that a differential is determined by its value on arrows. In combination with the particular shape of the quiver with relations for gentle algebras, we show that, up to coboundary, the differential can only be nonvanishing on the following three types of arrows:
\begin{equation}
\label{eq:d1}
\begin{tikzpicture}[x=3em,y=1.5em,baseline=.5em]
\node at (-1.7,0) {};
\node at (8,0) {};
\begin{scope}[yshift=1.4em]
\node[shape=circle, scale=.7] (LL) at (-.5,0) {};
\node[shape=circle, scale=.7] (RR) at (3.5,0) {};
\node[shape=circle, scale=.7] (M) at (1.5,0) {};
\draw[line width=1pt, fill=black] (1.5,0) circle(0.2ex);
\path[->, line width=.5pt, font=\scriptsize] (M) edge[out=230, in=310, looseness=25] node[right=-.1ex,pos=.8] {$p$} (M);
\path[->, line width=.5pt, font=\scriptsize] (LL) edge node[below=-.3ex] {$u$} (M);
\path[->, line width=.5pt, font=\scriptsize] (M) edge node[below=-.3ex] {$v$} (RR);
\draw[dash pattern=on 0pt off 1.2pt, line width=.6pt, line cap=round] ($(1.5,0)+(170:.8em)$) arc[start angle=170, end angle=8, radius=.8em];
\draw[dash pattern=on 0pt off 1.2pt, line width=.6pt, line cap=round] ($(1.5,0)+(244:.8em)$) arc[start angle=244, end angle=302, radius=.8em];
\end{scope}
\node[right] at (5.25,.5) {$d (p) = e_{\mathrm s (p)}$};
\end{tikzpicture}
\end{equation}
\begin{equation}
\label{eq:d2}
\begin{tikzpicture}[x=3em,y=1.5em,baseline=.5em]
\node at (-1.7,0) {};
\node at (8,0) {};
\node[shape=circle, scale=.7] (LL) at (-1.5,0) {};
\node[shape=circle, scale=.7] (L) at (0,0) {};
\node[shape=circle, scale=.7] (R) at (3,0) {};
\node[shape=circle, scale=.7] (RR) at (4.5,0) {};
\node[shape=circle, scale=.7] (TL) at (.95,1.08) {};
\node[shape=circle, scale=.7] (TM) at (1.9,1.14) {};
\node[shape=circle, scale=.7] (TR) at (2.05,1.07) {};
\draw[line width=1pt, fill=black] (0,0) circle(0.2ex);
\draw[line width=1pt, fill=black] (3,0) circle(0.2ex);
\draw[line width=1pt, fill=black] (.95,1.06) circle(0.2ex);
\node[font=\scriptsize] at (1.98,1.112) {.};
\node[font=\scriptsize] at (1.92,1.135) {.};
\node[font=\scriptsize] at (2.04,1.082) {.};
\path[->, line width=.5pt, font=\scriptsize] (LL) edge node[below=-.3ex] {$u$} (L);
\path[->, line width=.5pt, font=\scriptsize] (L) edge node[below=-.3ex] {$p$} (R);
\path[->, line width=.5pt, font=\scriptsize] (R) edge node[below=-.3ex] {$v$} (RR);
\path[->, line width=.5pt, font=\scriptsize, bend left=9] (L) edge node[above=-.2ex] {$q_1$} (TL);
\path[->, line width=.5pt, font=\scriptsize, bend left=8, overlay] (TL) edge node[above=-.4ex] {$q_2$} (TM.170);
\path[->, line width=.5pt, font=\scriptsize, bend left=9] (TR.-13) edge node[above=-.1ex] {$q_n$} (R);
\draw[dash pattern=on 0pt off 1.2pt, line width=.6pt, line cap=round] ($(0,0)+(172:.8em)$) arc[start angle=172, end angle=40, radius=.8em];
\draw[dash pattern=on 0pt off 1.2pt, line width=.6pt, line cap=round] ($(3,0)+(8:.8em)$) arc[start angle=8, end angle=140, radius=.8em];
\node[right] at (5.25,.5) {$d (p) = q_n \dotsb q_2 q_1$};
\end{tikzpicture}
\end{equation}
\begin{equation}
\label{eq:d3}
\begin{tikzpicture}[x=3em,y=1.5em,baseline=.5em]
\node at (-1.7,0) {};
\node at (8,0) {};
\begin{scope}[yshift=.6em]
\node[shape=circle, scale=.7] (CL1) at (-.65,-.5) {};
\draw[line width=1pt, fill=black] (-.65,-.5) circle(0.2ex);
\node[shape=circle, scale=.7] (CL2) at (-.8,.5) {};
\node[shape=circle, scale=.7] (CL3) at (-.65,.52) {};
\node[shape=circle, scale=.7] (CR1) at (3.65,.5) {};
\draw[line width=1pt, fill=black] (3.65,.5) circle(0.2ex);
\node[shape=circle, scale=.7] (CR2) at (3.8,-.5) {};
\node[shape=circle, scale=.7] (CR3) at (3.65,-.52) {};
\node[shape=circle, scale=.7] (L) at (0,0) {};
\node[shape=circle, scale=.7] (R) at (3,0) {};
\draw[line width=1pt, fill=black] (0,0) circle(0.2ex);
\draw[line width=1pt, fill=black] (3,0) circle(0.2ex);
\node[font=\scriptsize] at (-.67,.53) {.};
\node[font=\scriptsize] at (-.73,.52) {.};
\node[font=\scriptsize] at (-.79,.5) {.};
\node[font=\scriptsize] at (3.67,-.53) {.};
\node[font=\scriptsize] at (3.73,-.52) {.};
\node[font=\scriptsize] at (3.79,-.5) {.};
\path[->, line width=.5pt, font=\scriptsize] (L) edge node[below=-.3ex] {$p$} (R);
\path[->, line width=.5pt, font=\scriptsize] (L) edge[in=10, out=210] node[below, pos=.2] {$u_1$} (CL1.8);
\path[->, line width=.5pt, font=\scriptsize] (CL1) edge[out=175, in=220, looseness=1.5] node[left=-.3ex] {$u_2$} (CL2.195);
\path[->, line width=.5pt, font=\scriptsize] (CL3.-5) edge[out=-10, in=150] node[above=-.3ex] {$u_k$} (L.140);
\path[->, line width=.5pt, font=\scriptsize] (R) edge[in=192, out=30] node[above, pos=.2] {$v_1$} (CR1.188);
\path[->, line width=.5pt, font=\scriptsize] (CR1) edge[out=-5, in=40, looseness=1.5] node[right=-.3ex] {$v_2$} (CR2.15);
\path[->, line width=.5pt, font=\scriptsize] (CR3.175) edge[out=170, in=-30] node[below=-.3ex] {$v_l$} (R.-40);
\draw[dash pattern=on 0pt off 1.2pt, line width=.6pt, line cap=round] ($(0,0)+(201:.8em)$) arc[start angle=201, end angle=150, radius=.8em];
\draw[dash pattern=on 0pt off 1.2pt, line width=.6pt, line cap=round] ($(3,0)+(21:.8em)$) arc[start angle=21, end angle=-30, radius=.8em];
\end{scope}
\node[right] at (5.25,.5) {$d (p) = v p u$.};
\end{tikzpicture}
\end{equation}
In \eqref{eq:d1} and \eqref{eq:d2} $u$ or $v$ may or may not be present in $Q$, but there must not be any more incoming or outgoing arrows at $\mathrm s (p)$ and $\mathrm t (p)$. In \eqref{eq:d3}, the cycles $u = u_k \dotsb u_2 u_1$ and $v = v_l \dotsb v_2 v_1$ must be present (i.e.\ $k, l \geq 1$), but likewise, there must not be more incoming or outgoing arrows at $\mathrm s (p)$ and $\mathrm t (p)$.

The Leibniz rule implies moreover, than on longer paths, the differential can only be nonvanishing on paths starting and/or ending with a loop as in \eqref{eq:d1}.

These observations are proved in the following lemma. (See also the proof of Theorem \ref{theorem:hochschild} for a more general discussion on the nontrivial cocycles of $\HH^2 (A, A)$.)

\begin{lemma}
\label{lemma:differential}
Let $A = \Bbbk Q / I$ be a graded gentle algebra and let $\widetilde d$ be a differential on $A$ satisfying the Leibniz rule with respect to the multiplication on $A$. Then $\widetilde d$ defines a Hochschild $2$-cocycle for $A$ which, up to isomorphism, is cohomologous to a differential $d$ of the following standard form:
\begin{enumerate}
\item $d (e_i) = 0$ for any vertex $i \in Q_0$
\item if $p \in Q_1$ belongs to a cycle without relations, i.e.\ $p = p_1$ for some cycle $p_n \dotsb p_2 p_1$ with $p_2 p_1, p_3 p_2, \dotsc, p_1 p_n \not\in I$, then $p_1 p_n p_{n-1} \dotsb p_2 p_1$ does not appear as a term with nonzero coefficient in $d (p)$
\item if $p \in Q_1$ belongs to a zero-relation involving another arrow, i.e.\ if either $p p' \in I$ or $p'' p \in I$ for some $p', p'' \in Q_1 \smallsetminus \{ p \}$, then $d (p) = 0$
\item if $p \in Q_1$ does not belong to a cycle without relations and not belong to a zero-relation involving another arrow, then $d (p) \neq 0$ only if 
\begin{itemize}
\item either $\mathrm s (p) \neq \mathrm t (p)$ in which case $d (p) = q$ where $q$ is the unique path parallel to $p$ which must have degree $\lvert q \rvert = \lvert p \rvert + 1$
\item or $\mathrm s (p) = \mathrm t (p)$ in which case we must have $\lvert p \rvert = -1$ and $d (p) = e_{\mathrm s (p)}$
\end{itemize}
\item if $r = r_n \dotsb r_2 r_1$ is a path of length $n \geq 2$, then $d (r) \neq 0$ only if none of the $r_i$'s belongs to a cycle without relations and we are moreover in one of the following three situations:
\begin{itemize}
\item $d (r_n) = 0$ and $r_1$ is a loop of degree $-1$ with $d (r_1) = e_{\mathrm s (r)} \neq 0$ in which case $d (r) =   (-1)^{|r_n\dotsb r_2|} r_n \dotsb r_3 r_2$
\item $d (r_1) = 0$ and $r_n$ is a loop of degree $-1$ with $d (r_n) = e_{\mathrm t (r)}$ in which case $d (r) = r_{n-1} \dotsb r_2 r_1$
\item $d (r_1) = e_{\mathrm s (s)} \neq 0$ and $d (r_n) = e_{\mathrm t (r)} \neq 0$ in which case
\[
d (r) = (-1)^{|r_n\dotsb r_2|} r_n \dotsb r_3 r_2 + r_{n-1} \dotsb r_2 r_1.
\]
\end{itemize}
\end{enumerate}
\end{lemma}

\begin{proof}
For {\itemi} observe that for any idempotent $e_i$ with $i \in Q_0$ we have for {\it any} differential $d$
\[
d (e_i) = d (e_i^2) = e_i d(e_i) + d(e_i)e_i = 2 d (e_i).
\]
Since we work over a field of characteristic $0$ we conclude that $d (e_i) = 0$.

Let us verify {\itemii}. First observe that if $\widetilde d (p) \neq 0$, then there can be at most two paths $q, r$ parallel to $p$ and of degree $\lvert p \rvert + 1$. Since there are at most two outgoing arrows at $\mathrm s (p)$, $p$ itself must belong to at least one of these two paths, say $r$. Writing $p = p_0$, we thus have $r = p_0 p_n \dotsb p_2 p_1 p_0$, where $p_n \dotsb p_2 p_1 p_0$ is a cycle without relations. (Even though the paths $p_0 (p_n \dotsb p_2 p_1 p_0)^k$ are parallel to $p = p_0$ for all $k \geq 0$, only for $k = 1$ can such a path have degree one more than $p_0$.) We now show that the term of $\widetilde d (p_0)$ corresponding to $r$ can be killed by a coboundary. In other words, we show that
\[
\widetilde d (p_0) = \lambda q + \lambda' p_0 p_n \dotsb p_2 p_1 p_0
\]
is cohomologous to
\[
d (p_0) = \lambda q.
\]
Consider the element $\varphi \in \Hom(\Bbbk Q_0, A)$ given by 
\begin{align*}
\varphi(e_i) = \lambda_i p_{i-1}\dotsb p_2 p_1 p_0 p_n  \dotsb p_{i} 
\end{align*}
for $0\leq i\leq n$ where we write $p = p_0$. Then using the formula \eqref{differentialsmaller} we obtain for $0 \leq i \leq n$
\begin{align*}
\delta(\varphi) (\s p_i) &=  -(-1)^{|p_i|-1}  p_i\varphi(e_{i}) +(-1)^{-|\varphi|} \varphi(e_{i+1}) p_{i}\\
&= (-1)^{|p_i|}  \lambda_{i} p_i  \dotsb p_1p_0p_n  \dotsb p_{i} - \lambda_{i+1}  p_i  \dotsb p_1p_0p_n  \dotsb p_{i}\\
&= ((-1)^{|p_i|}  \lambda_{i}  - \lambda_{i+1}) p_i  \dotsb p_1p_0p_n  \dotsb p_{i}\end{align*}
where we denote $\lambda_{n+1} = \lambda_0$. Note that $\delta(\varphi)(\s p') = 0$ for any arrow $p'$ which does not belong to the cycle since $p'p_i=0=p_ip'$ in $A$.  In order to obtain $\delta (\varphi) = r$ we need to have 
\begin{align*}
(-1)^{|p_0|} \lambda_0 -  \lambda_1 &=1\\
(-1)^{|p_i|}  \lambda_i - \lambda_{i+1}& =0 \quad  1\leq i \leq n-1\\
(-1)^{|p_n|} \lambda_n -  \lambda_0 &=0.
\end{align*}
This is equivalent to that $(-1)^{|p_0|} \lambda_0 -  \lambda_1 =1$ and 
\[
\lambda_{i} = (-1)^{|p_i|+\dotsb +|p_n|} \lambda_0 \quad \text{for $1\leq i \leq n$}.
\]
Thus, we obtain a (unique) solution $\lambda_0 = (-1)^{|p_0|}$ and 
$$\lambda_i = (-1)^{|p_0|+|p_i|+\dotsb +|p_n|} \lambda_0 \quad \text{for $1\leq i \leq n$}.$$
This shows that $\widetilde d$ is cohomologous to $d$.

To show {\itemiii}, let $p \in Q_1$ with $d (p) \neq 0$ and let $p' \in Q_1 \smallsetminus \{ p \}$ such that $p p' \neq 0$ in $\Bbbk Q$, but $p p' \in I$. By {\itemii} and the Leibniz rule we may assume that $d (p)$ has no terms which belong to a cycle without relations. In this case there is at most one nonzero path $q$ parallel to $p$. Up to rescaling $p$ we may thus assume $d (p) = q$. Let us first assume that $q$ contains $p$ as a subpath, i.e.\ $q = v p u$ for some cycles $u = u_k \dotsb u_2 u_1$ at $\mathrm s (p)$ and $v = v_l \dotsb v_2 v_1$ at $\mathrm t (p)$ of length $k, l \geq 0$. Note that since $\Bbbk Q / I$ is gentle we must have $u_1 u_k, v_1 v_l \in I$ since $p u_k, v_1 p \not\in I$, i.e.\ the cycles $u$ and $v$ have a single relation at $\mathrm s (p)$ and $\mathrm t (p)$, respectively. 
\begin{equation*}
\begin{tikzpicture}[x=3em,y=1.5em,baseline=.5em]
\begin{scope}[yshift=.6em]
\node[shape=circle, scale=.7] (CL1) at (-.65,-.5) {};
\draw[line width=1pt, fill=black] (-.65,-.5) circle(0.2ex);
\node[shape=circle, scale=.7] (CL2) at (-.8,.5) {};
\node[shape=circle, scale=.7] (CL3) at (-.65,.52) {};
\node[shape=circle, scale=.7] (CR1) at (3.65,.5) {};
\draw[line width=1pt, fill=black] (3.65,.5) circle(0.2ex);
\node[shape=circle, scale=.7] (CR2) at (3.8,-.5) {};
\node[shape=circle, scale=.7] (CR3) at (3.65,-.52) {};
\node[shape=circle, scale=.7] (L) at (0,0) {};
\node[shape=circle, scale=.7] (R) at (3,0) {};
\node[shape=circle, scale=.7] (M) at (1,-1.4) {};
\draw[line width=1pt, fill=black] (0,0) circle(0.2ex);
\draw[line width=1pt, fill=black] (3,0) circle(0.2ex);
\node[font=\scriptsize] at (-.67,.53) {.};
\node[font=\scriptsize] at (-.73,.52) {.};
\node[font=\scriptsize] at (-.79,.5) {.};
\node[font=\scriptsize] at (3.67,-.53) {.};
\node[font=\scriptsize] at (3.73,-.52) {.};
\node[font=\scriptsize] at (3.79,-.5) {.};
\path[->, line width=.5pt, font=\scriptsize] (L) edge node[below=-.3ex] {$p$} (R);
\path[->, line width=.5pt, font=\scriptsize] (M) edge[bend left=18] node[below=.1ex, pos=.7] {$p'$} (L);
\path[->, line width=.5pt, font=\scriptsize] (L) edge[in=10, out=210] node[below, pos=.2] {$u_1$} (CL1.8);
\path[->, line width=.5pt, font=\scriptsize] (CL1) edge[out=175, in=220, looseness=1.5] node[left=-.3ex] {$u_2$} (CL2.195);
\path[->, line width=.5pt, font=\scriptsize] (CL3.-5) edge[out=-10, in=150] node[above=-.3ex] {$u_k$} (L.140);
\path[->, line width=.5pt, font=\scriptsize] (R) edge[in=192, out=30] node[above, pos=.2] {$v_1$} (CR1.188);
\path[->, line width=.5pt, font=\scriptsize] (CR1) edge[out=-5, in=40, looseness=1.5] node[right=-.3ex] {$v_2$} (CR2.15);
\path[->, line width=.5pt, font=\scriptsize] (CR3.175) edge[out=170, in=-30] node[below=-.3ex] {$v_l$} (R.-40);
\draw[dash pattern=on 0pt off 1.2pt, line width=.6pt, line cap=round] ($(0,0)+(201:.8em)$) arc[start angle=201, end angle=150, radius=.8em];
\draw[dash pattern=on 0pt off 1.2pt, line width=.6pt, line cap=round] ($(0,0)+(-7:.85em)$) arc[start angle=-7, end angle=-50, radius=.85em];
\draw[dash pattern=on 0pt off 1.2pt, line width=.6pt, line cap=round] ($(3,0)+(21:.8em)$) arc[start angle=21, end angle=-30, radius=.8em];
\end{scope}
\end{tikzpicture}
\end{equation*}
We have
\[
0 = d (0) = d (p p') = d (p) p' + (-1)^{\lvert p \rvert} p d (p') = v p u p' + (-1)^{\lvert p \rvert} p d (p').
\]
This equality holds only if $d(p')= (-1)^{|p|+1} up'$ and $v= \mathrm t (p)$ is a vertex (i.e.\ $l = 0$). In this case, $d$ is actually a coboundary since if we consider $\psi(e_{\mathrm s(p)}) = (-1)^{|p|-1} u$ then  $\delta(\psi) =d$.

If $d (p)$ does not contain $p$ as a subpath, i.e.\ $q = q_m \dotsb q_2 q_1$ with $q_i \neq p$ for all $1 \leq i \leq m$, then $d (p) p' = q p' \neq 0$ since both $q = q_m \dotsb q_2 q_1 \neq 0$ and $q_1 p' \neq 0$. This is because we already have $p p' \in I$ whence $q_1 p' \not\in I$ since $\Bbbk Q / I$ is gentle. If $d (p') = 0$, this already yields a contradiction. If $d (p') = q' \neq 0$ we have
\[
0 = d (p p') = q p'+(-1)^{|p|} p q'
\]
but since $q_m \neq p$, the right-hand side is nonzero, yielding a contradiction to the well-definedness of $d$ on $\Bbbk Q / I$.

Using a similar argument one verifies {\itemiv} and {\itemv}. 
\end{proof}

Lemma \ref{lemma:differential} motivates the following definition.

\begin{definition}
\label{definition:dggentle}
A {\it DG gentle algebra} is a DG algebra $(A, d)$ such that $A = \Bbbk Q / I$ is a graded gentle algebra and for any arrow $p \in Q_1$, the paths appearing with nonzero coefficient in $d (p)$ do not start or end with $p$.
\end{definition}

\begin{remark}
Lemma \ref{lemma:differential} shows that the condition on the differential in a DG gentle algebra is always satisfied up to a coboundary. The three possible situations where $d (p) \neq 0$ are given in \eqref{eq:d1}--\eqref{eq:d3} and the local pictures are given in Fig.~\ref{fig:differential}. On paths of length $\geq 2$, the differential can only be nonvanishing if the path starts and/or ends with a loop that is mapped to the idempotent by the differential.
\end{remark}

\subsection{Surface model and formality for DG gentle algebras}

In \cite{barmeierschrollwang} we introduce several notions of dissections of graded orbifold surfaces and associate to each dissection $\Delta$ an A$_\infty$ category $\mathbf A_\Delta$ such that $\mathcal W (\mathbf S) \simeq \tw (\mathbf A_\Delta)^\natural$. In \cite[Section~8]{barmeierschrollwang} we define the notion of a {\it DG dissection} whose associated A$_\infty$ category $\mathbf A_\Delta$ is a DG category. This notion turns out to coincide with the notion of a DG gentle algebra in the following sense.

\begin{proposition}
\label{proposition:dgdissection}
Let $A$ be a DG gentle algebra viewed as a DG category with objects $Q_0$. Then there exists a graded orbifold surface $\mathbf S$ with stops and a DG dissection $\Delta$ of $\mathbf S$ such that $A$ is quasi-isomorphic to the DG category $\mathbf A_\Delta$.
\end{proposition}

\begin{figure}
\begin{tikzpicture}[x=1em,y=1em]
\begin{scope}[xshift=-7em]
\node[font=\small] at (0,2) {$A$};
\end{scope}
\begin{scope} 
\draw[line width=.2em] (0,3) circle(.8em);
\draw[line width=.75pt, line cap=round, color=arccolour] (-1,0) to[out=85, in=240, looseness=.8] ($(0,3)+(150:1.5)$) arc[start angle=150, end angle=30, radius=1.5] to[out=300, in=95, looseness=.8] (1,0);
\begin{scope}[decoration={markings,mark=at position 0.55 with {\arrow[black]{Stealth[length=4.2pt]}}}]
\draw[line width=.5pt, line cap=round, postaction={decorate}] (-3,0) -- (3,0);
\draw[line width=0pt, line cap=round, postaction={decorate}] (-3,0) -- (-1.25,0);
\draw[line width=0pt, line cap=round, postaction={decorate}] (1.75,0) -- (3,0);
\end{scope}
\node[font=\scriptsize] at (0,-.75) {$p$};
\node[font=\scriptsize] at (-2.25,-.75) {$u$};
\node[font=\scriptsize] at (2.25,-.75) {$v$};
\draw[dash pattern=on 3.9pt off 2pt, line width=.5pt, line cap=round, looseness=2.9] (-3,0) to[bend left=90] (3,0);
\end{scope}
\begin{scope}[xshift=9em] 
\draw[line width=.5pt] (0,3) circle(.8em);
\draw[line width=.5pt, fill=black] (0,2.2) circle(.15em);
\draw[line width=.75pt, line cap=round, color=arccolour] (-1,0) to ($(0,3)+(240:.8)$) (1,0) to ($(0,3)+(300:.8)$) ($(0,3)+(180:.8)$) to ++(180:1.65) ($(0,3)+(110:.8)$) to ++(110:1.25) ($(0,3)+(0:.8)$) to ++(0:1.65);
\begin{scope}[yshift=3em,decoration={markings,mark=at position 0.55 with {\arrow[black]{Stealth[length=4.2pt]}}}]
\draw[line width=0pt, line cap=round, postaction={decorate}] (205:.79) -- (198:.806);
\begin{scope}[rotate=-70]
\draw[line width=0pt, line cap=round, postaction={decorate}] (205:.79) -- (198:.806);
\end{scope}
\begin{scope}[rotate=115]
\draw[line width=0pt, line cap=round, postaction={decorate}] (205:.79) -- (198:.806);
\end{scope}
\node[font=\scriptsize] at (210:1.35) {$q_1$};
\node[font=\scriptsize] at (150:1.4) {$q_2$};
\node[font=\scriptsize] at (-30:1.45) {$q_n$};
\end{scope}
\draw[line width=.5pt, line cap=round, postaction={decorate}] (-3,0) -- (3,0);
\draw[line width=0pt, line cap=round, postaction={decorate}] (-3,0) -- (-1.25,0);
\draw[line width=0pt, line cap=round, postaction={decorate}] (1.75,0) -- (3,0);
\node[font=\scriptsize] at (0,-.75) {$p$};
\node[font=\scriptsize] at (-2.25,-.75) {$u$};
\node[font=\scriptsize] at (2.25,-.75) {$v$};
\draw[dash pattern=on 3.9pt off 2pt, line width=.5pt, line cap=round, looseness=2.9] (-3,0) to[bend left=90] (3,0);
\end{scope}
\begin{scope}[xshift=18em, yshift=2.5em, rotate=180, decoration={markings,mark=at position 0.55 with {\arrow[black]{Stealth[length=4.2pt]}}}] 
\draw[dash pattern=on 3.9pt off 2pt, line width=.5pt, line cap=round] (0,0) circle(3em);
\draw[dash pattern=on 3.9pt off 2pt, line width=.5pt, line cap=round] (1.5,0) ellipse[x radius=.5, y radius=1];
\begin{scope}[xshift=-1em]
\draw[line width=.5pt, line cap=round] (0,0) circle(1em);
\draw[line width=.5pt, fill=black] (0,1) circle(.15em);
\begin{scope}[rotate=-15]
\draw[line width=0pt, postaction={decorate}] (140:1) -- +(140-80:.1);
\draw[line width=0pt, postaction={decorate}] (225:1) -- +(225-80:.1);
\draw[line width=0pt, postaction={decorate}] (277:1) -- +(277-80:.1);
\draw[line width=0pt, postaction={decorate}] (335:1) -- +(335-80:.1);
\draw[line width=0pt, postaction={decorate}] (55:1) -- +(55-80:.1);
\end{scope}
\node[font=\scriptsize] at (135:1.55) {$v_l$};
\node[font=\scriptsize] at (215:1.5) {$v_1$};
\node[font=\scriptsize] at (270:1.5) {$p$};
\node[font=\scriptsize] at (330:1.6) {$u_k$};
\node[font=\scriptsize] at (45:1.5) {$u_1$};
\draw[line width=.75pt, line cap=round, color=arccolour] (155:1) -- +(155:1);
\draw[line width=.75pt, line cap=round, color=arccolour] (25:1) -- +(15:1.3);
\draw[line width=.75pt, line cap=round, color=arccolour] (190:1) -- +(190:1);
\draw[line width=.75pt, line cap=round, color=arccolour] (-10:1) -- +(-10:1);
\end{scope}
\draw[line width=.75pt, line cap=round, color=arccolour] ($(-1,0)+(110:1)$) to[out=110, in=180, looseness=1.2] (90:2.66) arc[start angle=90, end angle=-90, radius=2.66] to[out=180, in=240, looseness=1.2] ($(-1,0)+(240:1)$);
\draw[line width=.75pt, line cap=round, color=arccolour] ($(-1,0)+(70:1)$) to[out=70, in=150, looseness=1.2] (60:2.33) arc[start angle=60, end angle=-60, radius=2.33] to[out=210, in=300, looseness=1.2] ($(-1,0)+(300:1)$);
\end{scope}
\begin{scope}[yshift=-8em] 
\begin{scope}[xshift=-7em]
\node[font=\small] at (0,2) {$(A, d)$};
\end{scope}
\begin{scope}[decoration={markings,mark=at position 0.55 with {\arrow[black]{Stealth[length=4.2pt]}}}] 
\draw[line width=.5pt, line cap=round, postaction={decorate}] (-3,0) -- (3,0);
\node[font=\scriptsize] at (0,-.75) {$v p u$};
\draw[dash pattern=on 3.9pt off 2pt, line width=.5pt, line cap=round, looseness=2.9] (-3,0) to[bend left=90] (3,0);
\end{scope}
\begin{scope}[xshift=9em] 
\node[font=\scriptsize] at (0,3) {$\times$};
\draw[line width=.5pt, fill=black] (0,2.3) circle(.1em);
\draw[line width=.75pt, line cap=round, color=arccolour] (-1,0) to ($(0,3)+(240:.3)$) (1,0) to ($(0,3)+(300:.3)$) ($(0,3)+(180:.3)$) to ++(180:2.15) ($(0,3)+(110:.3)$) to ++(110:1.75) ($(0,3)+(0:.3)$) to ++(0:2.15);
\begin{scope}[yshift=3em, decoration={markings,mark=at position 1 with {\arrow[black]{Stealth[length=4.2pt]}}}]
\draw[line width=.5pt, line cap=round] (240:.8) arc[start angle=240, end angle=190, radius=.8];
\draw[line width=.5pt, line cap=round] (175:.8) arc[start angle=175, end angle=120, radius=.8];
\draw[line width=.6pt, line cap=round, dash pattern=on 0pt off 1.2pt] (60:.8) arc[start angle=64, end angle=46, radius=.8];
\draw[line width=.5pt, line cap=round] (-5:.8) arc[start angle=-5, end angle=-55, radius=.8];
\begin{scope}[rotate=-15]
\draw[line width=0pt, line cap=round, postaction={decorate}] (205:.79) -- (198:.806);
\begin{scope}[rotate=-70]
\draw[line width=0pt, line cap=round, postaction={decorate}] (205:.79) -- (198:.806);
\end{scope}
\begin{scope}[rotate=115]
\draw[line width=0pt, line cap=round, postaction={decorate}] (205:.79) -- (198:.806);
\end{scope}
\end{scope}
\node[font=\scriptsize] at (210:1.35) {$q_1$};
\node[font=\scriptsize] at (150:1.4) {$q_2$};
\node[font=\scriptsize] at (-30:1.45) {$q_n$};
\end{scope}
\begin{scope}[decoration={markings,mark=at position 0.55 with {\arrow[black]{Stealth[length=4.2pt]}}}]
\draw[line width=.5pt, line cap=round, postaction={decorate}] (-3,0) -- (3,0);
\draw[line width=0pt, line cap=round, postaction={decorate}] (-3,0) -- (-1.25,0);
\draw[line width=0pt, line cap=round, postaction={decorate}] (1.75,0) -- (3,0);
\node[font=\scriptsize] at (0,-.75) {$p$};
\node[font=\scriptsize] at (-2.25,-.75) {$u$};
\node[font=\scriptsize] at (2.25,-.75) {$v$};
\draw[dash pattern=on 3.9pt off 2pt, line width=.5pt, line cap=round, looseness=2.9] (-3,0) to[bend left=90] (3,0);
\end{scope}
\end{scope}
\begin{scope}[xshift=18em, yshift=2.5em, rotate=180, decoration={markings,mark=at position 1 with {\arrow[black]{Stealth[length=4.2pt]}}}] 
\draw[dash pattern=on 3.9pt off 2pt, line width=.5pt, line cap=round] (0,0) circle(3em);
\draw[dash pattern=on 3.9pt off 2pt, line width=.5pt, line cap=round] (1.5,0) ellipse[x radius=.5, y radius=1];
\begin{scope}[xshift=-1em]
\node[font=\scriptsize] at (0,0) {$\times$};
\draw[line width=.5pt, fill=black] (0,.8) circle(.1em);
\draw[line width=.5pt, line cap=round] (296:1) arc[start angle=296, end angle=250, radius=1];
\draw[line width=.5pt, line cap=round] (236:1) arc[start angle=236, end angle=200, radius=1];
\draw[dash pattern=on 0pt off 1.2pt, line width=.6pt, line cap=round] (180:1) arc[start angle=180, end angle=165, radius=1];
\draw[line width=.5pt, line cap=round] (151:1) arc[start angle=151, end angle=120, radius=1];
\draw[line width=.5pt, line cap=round] (66:1) arc[start angle=66, end angle=35, radius=1];
\draw[dash pattern=on 0pt off 1.2pt, line width=.6pt, line cap=round] (12:1) arc[start angle=12, end angle=-3, radius=1];
\draw[line width=.5pt, line cap=round] (-14:1) arc[start angle=-14, end angle=-50, radius=1];
\draw[line width=0pt, postaction={decorate}] (113:1) -- +(113-80:.01);
\node[font=\scriptsize] at (135:1.55) {$v_l$};
\draw[line width=0pt, postaction={decorate}] (193:1) -- +(193-80:.01);
\node[font=\scriptsize] at (215:1.5) {$v_1$};
\draw[line width=0pt, postaction={decorate}] (243:1) -- +(243-80:.01);
\node[font=\scriptsize] at (270:1.45) {$p$};
\draw[line width=0pt, postaction={decorate}] (303:1) -- +(303-80:.01);
\node[font=\scriptsize] at (330:1.6) {$u_k$};
\draw[line width=0pt, postaction={decorate}] (23:1) -- +(23-80:.01);
\node[font=\scriptsize] at (45:1.5) {$u_1$};
\draw[line width=.75pt, line cap=round, color=arccolour] (155:.3) -- +(155:1.7);
\draw[line width=.75pt, line cap=round, color=arccolour] (20:.3) -- +(20:2);
\draw[line width=.75pt, line cap=round, color=arccolour] (190:.3) -- +(190:1.65);
\draw[line width=.75pt, line cap=round, color=arccolour] (-10:.3) -- +(-10:1.7);
\end{scope}
\draw[line width=.75pt, line cap=round, color=arccolour] ($(-1,0)+(110:.3)$) to ($(-1,0)+(110:1)$) to[out=110, in=180, looseness=1.2] (90:2.66) arc[start angle=90, end angle=-90, radius=2.66] to[out=180, in=240, looseness=1.2] ($(-1,0)+(240:1)$) to ($(-1,0)+(240:.3)$);
\draw[line width=.75pt, line cap=round, color=arccolour] ($(-1,0)+(70:.3)$) to ($(-1,0)+(70:1)$) to[out=70, in=150, looseness=1.2] (60:2.33) arc[start angle=60, end angle=-60, radius=2.33] to[out=210, in=300, looseness=1.2] ($(-1,0)+(300:1)$) to ($(-1,0)+(300:.3)$);
\end{scope}
\end{scope} 
\end{tikzpicture}
\caption{Local pictures of the three types of (smooth) surfaces for the underlying gentle algebra $A$ of a DG gentle algebra $(A, d)$, and the corresponding orbifold surface modelled by $(A, d)$ illustrating the geometric interpretation of the differentials \eqref{eq:d1}--\eqref{eq:d3}.}
\label{fig:differential}
\end{figure}
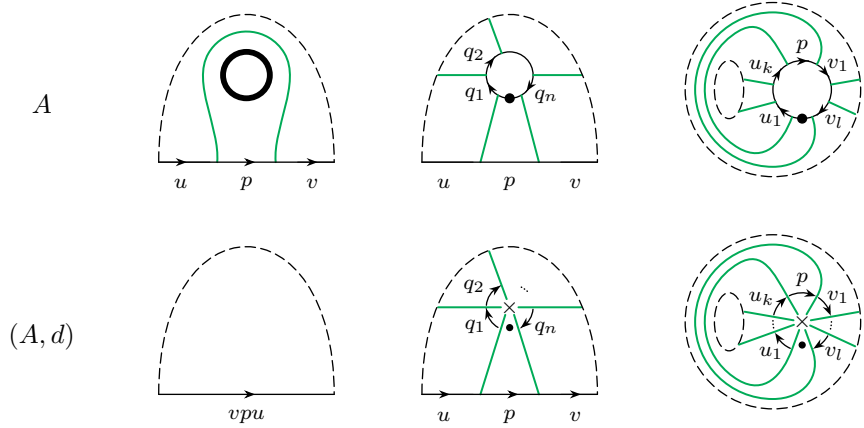

Combining Proposition \ref{proposition:dgdissection} with the main idea of \cite[Theorem 8.6]{barmeierschrollwang} we can also characterize when a DG gentle algebra is formal. (In particular, a DG gentle algebra whose differential has a nonzero component of type \eqref{eq:d3} is never formal.)

\begin{theorem}
\label{theorem:formality}
Let $A$ be a DG gentle algebra such that the associated orbifold surface is not a disk with one stop in the boundary and two orbifold points. Then $A$ is formal if and only if for every arrow $p$ such that $d (p) \neq 0$ we either have that $d (p)$ is a path of length $0$ or $1$ or that $p$ cannot be composed nontrivially with any other arrow.
\end{theorem}

\begin{proof}
The statement is proven in \cite[Theorem 8.6]{barmeierschrollwang} for the case when the differential is nonvanishing on an arrow of type \eqref{eq:d2}. It remains to prove that $A$ is not formal whenever its differential $d$ is nonvanishing on an arrow of type \eqref{eq:d3}, and that nonvanishing on a loop of type \eqref{eq:d1} does not hinder it being formal.

The latter follows from the observation that there is a quasi-isomorphism $e A e \to A$ where $e A e$ is the subalgebra determined by the idempotent $e$ given by the sum of all idempotents excluding the idempotents which lie in the image of $d$. (See also \eqref{eq:obs3} in the proof of Theorem \ref{theorem:deformationwrapped} below.)

The former can be proved as follows. Let $p$ be an arrow such that $d (p) = v_l \dotsb v_2 v_1 p u_k \dotsb u_2 u_1$ as in \eqref{eq:d3} and let $e = e_{\mathrm s (p)} + e_{\mathrm t (p)}$ be the idempotent corresponding to the starting and ending vertices of $p$. We claim that the subalgebra $e A e \subset A$ is not formal. Writing $v = v_l \dotsb v_1$ and $u = u_k \dotsb u_1$, we have
\begin{equation*}
\begin{tikzpicture}[x=1em,y=1em,baseline=.5em]
\begin{scope}
\node[left] at (-1.1,0) {$e A e \simeq \biggl( \Bbbk \Bigl($};
\node[shape=circle, scale=.7, inner sep=1pt] (1) at (0,0) {};
\node[shape=circle, scale=.7, inner sep=1pt] (2) at (2.75,0) {};
\path[->, line width=.5pt, font=\scriptsize] (1) edge node[below=-.3ex] {$p$} (2);
\path[->, line width=.5pt, font=\scriptsize] (1) edge[out=220, in=140, looseness=40] node[pos=.2, below=-.3ex] {$u$} (1);
\path[<-, line width=.5pt, font=\scriptsize] (2) edge[out=-40, in=40, looseness=40] node[pos=.2, below=-.3ex] {$v$} (2);
\node[right] at (3.85,0) {$\Bigr), d (p) = v p u \biggr)$};
\draw[dash pattern=on 0pt off 1.2pt, line width=.6pt, line cap=round] ($(0,0)+(205:.6em)$) arc[start angle=205, end angle=150, radius=.6em];
\draw[dash pattern=on 0pt off 1.2pt, line width=.6pt, line cap=round] ($(2.75,0)+(25:.6em)$) arc[start angle=25, end angle=-30, radius=.6em];
\end{scope}
\begin{scope}[yshift=-3em]
\node[left] at (-1.1,0) {$\H^\bullet (e A e) \simeq \phantom{\biggl(} \Bbbk \Bigl($};
\node[shape=circle, scale=.7, inner sep=1pt] (1) at (0,0) {};
\node[shape=circle, scale=.7, inner sep=1pt] (2) at (2.75,0) {};
\path[->, line width=.5pt, font=\scriptsize] (0.1,.175) edge node[above=-.4ex, pos=.55] {$[pu]$} (2.65,.175);
\path[->, line width=.5pt, font=\scriptsize] (.1,-.175) edge node[below=-.4ex, pos=.4] {$[vp]$} (2.65,-.175);
\path[->, line width=.5pt, font=\scriptsize] (1) edge[out=220, in=140, looseness=40] node[pos=.2, below=-.3ex] {$u$} (1);
\path[->, line width=.5pt, font=\scriptsize] (2) edge[out=-40, in=40, looseness=40] node[pos=.2, below=-.3ex] {$v$} (2);
\node[right] at (3.85,0) {$\Bigr) \bigm / I$};
\draw[dash pattern=on 0pt off 1.2pt, line width=.6pt, line cap=round] ($(0,0)+(205:.6em)$) arc[start angle=205, end angle=150, radius=.6em];
\draw[dash pattern=on 0pt off 1.2pt, line width=.6pt, line cap=round] ($(2.75,0)+(25:.6em)$) arc[start angle=25, end angle=-30, radius=.6em];
\draw[dash pattern=on 0pt off 1.2pt, line width=.6pt, line cap=round] ($(0,0)+(130:.6em)$) to[bend left=45] (.5,.3);
\draw[dash pattern=on 0pt off 1.2pt, line width=.6pt, line cap=round] ($(2.75,0)+(-45:.6em)$) to[bend left=45] (2.15,-.3);
\end{scope}
\end{tikzpicture}
\end{equation*}
where $[pu]$ and $[vp]$ denote the cohomology classes of (the cocycles) $pu$ and $vp$ and $I$ is the ideal generated by the monomial relations indicated by dots together with the element $[vp] u - v [pu]$. We claim that $\H^\bullet (e A e)$ cannot be a minimal model for the DG algebra $e A e$ unless $\H^\bullet (e A e)$ carries nontrivial higher structure. The underlying algebra of $e A e$ is a gentle algebra. It follows similar to the proof of Theorem \ref{theorem:hochschild} below that $\dim \HH^2 (e A e, e A e) \leq 1$. However, one sees that $\H^\bullet (e A e)$, viewed as graded associative algebra without higher structure, has at least one additional nontrivial $2$-cocycle whence
\[
\dim \HH^2 (\H^\bullet (e A e), \H^\bullet (e A e)) \geq \dim \HH^2 (e A e, e A e) + 1.
\]
As in Theorem \ref{theorem:hochschild}, we may use the Chouhy--Solotar resolution for $\H^\bullet (e A e)$ associated to the reduction system $\{ (u^2, 0), ([pu] u, 0), ([vp] u, v [pu]), (v [vp], 0), (v^2, 0) \}$ (see \cite{chouhysolotar,barmeierwang1}) to compute $\HH^\bullet$. We have a cochain $\phi$ determined by
\begin{align*}
\phi \colon v [vp] u &\mapsto (-1)^{|vp|} [vp] \\
[vp] u^2 &\mapsto  [pu].
\end{align*}
Since $v [vp] u$ and $[vp] u^2$ are overlap ambiguities ({\it $1$-ambiguities} in the notation of \cite[Chapter 3]{barmeierwang1}), $\phi$ defines a $3$-cochain in the underlying cochain complex. Since $\lvert u \rvert + \lvert v \rvert = 1$, this cochain is a degree $-1$ map and hence defines a cochain of total degree $2$, namely $|\varphi|=2$. We claim that $\phi$ is a cocycle. Indeed, we may check that $\partial^\bullet \phi = 0$ as follows. There are two $2$-ambiguities involving the two overlap ambiguities on which $\phi$ is nonvanishing, namely $v [vp] u^2$ and $v^2 [vp] u$. One checks that $\partial^\bullet \phi$ vanishes on both:
\begin{align*}
(\partial^\bullet \phi) (v [vp] u^2) &=  -(-1)^{|\s v||\phi|}  v \cdot \phi ([vp] u^2) + (-1)^{|\s v [vp] u|-|\phi|} \phi (v [vp] u) \cdot u\\
&= -(-1)^{|vp|} [vp] u + (-1)^{|vp|} v [pu] = 0 \\
(\partial^\bullet \phi) (v^2 [vp] u) &= -(-1)^{|\s v||\phi|} v \cdot \phi (v [vp] u) = -(-1)^{|vp|} v [vp] = 0
\end{align*}
where the differential is given similar to  \eqref{differentialsmaller} and we have used $|\phi|=2$ for the sign computations. Clearly $\partial^\bullet \phi$ trivially vanishes on all other $2$-ambiguities and hence defines a $2$-cocycle. To see that $\phi$ is not a coboundary, consider a $2$-cochain $\psi$. In order for $\partial^\bullet \psi$ to have the possibility of cancelling terms of $\phi$, it is enough to consider $\psi$ to be of the general form
\begin{align*}
\psi \colon v [vp] &\mapsto \alpha_1 [pu] + \alpha_2 [vp] + \alpha_3 \hair v [pu]\\
[vp] u &\mapsto \alpha_4 [pu] + \alpha_5 [vp] + \alpha_6 \hair v [pu] \\
[pu] u &\mapsto \alpha_7 [pu] + \alpha_8 [vp] + \alpha_9 \hair v [pu] \\ 
u^2 &\mapsto \alpha_{10} e_{\mathrm s (p)} + \alpha_{11} u\\
v^2 & \mapsto \alpha_{12} e_{\mathrm t (p)} + \alpha_{13} v
\end{align*}
for some constants $\alpha_1, \dotsc, \alpha_{13} \in \Bbbk$. Then we have that
\begin{align*}
(\partial^\bullet \psi) (v [vp] u) &= \pm \alpha_2 [vp] u \pm \alpha_4 v [pu] \pm \alpha_9 [pu] \pm \alpha_{10} \hair v [pu]\\
&= (\pm \alpha_2 \pm \alpha_4 \pm \alpha_{10}) \hair v [pu] \pm \alpha_9 [pu]\\
(\partial^\bullet \psi) ([vp] u^2) &= \pm \alpha_5 v [pu] \pm \alpha_7 v [pu]  \pm \alpha_{10} [vp] \pm \alpha_{11} v[pu]\\
&=(\pm \alpha_5 \pm \alpha_7 \pm \alpha_{11}) v[pu] \pm \alpha_{10} [vp].
\end{align*}
Here the precise signs are irrelevant, as none of these terms appear in $\phi (v [vp] u)$ and $\phi ([vp] u^2)$, respectively, whence $\phi$ is not a coboundary.

Thus $\phi$ defines a nontrivial $2$-cocycle which does not exist for $e A e$ and we must have $\dim \HH^2 (\H^\bullet (e A e), \H^\bullet (e A e)) > \dim \HH^2 (e A e, e A e)$. This implies that $\H^\bullet (e A e)$ without higher structure cannot be A$_\infty$-quasi-isomorphic to $e A e$, since any A$_\infty$-quasi-isomorphism induces an isomorphism on Hochschild cohomology. 

We conclude that $e A e$ is not formal. It now follows from \cite[Proposition 2.8]{barmeierschrollwang} that also $A$ is not formal, since it contains a non-formal full subcategory (when viewing $A$ as a DG category whose objects are given by the vertices).
\end{proof}

\subsection{Hochschild cohomology of DG gentle algebras}
\label{subsection:hochschildwrapped}

The following lemma shows that up to coboundaries, the Hochschild cocycles of a DG gentle algebra $A$ are determined by their values on maximal overlaps for the monomial relations.

\begin{lemma}\label{lemma:killoverlap}
Let $A$ be a DG gentle algebra and let $\phi \in \Hom_{\Bbbk Q_0^\e} (\s^m \Bbbk W_m, A)$ be a cocycle for $m \geq 1$. In case $m = 1$, assume that $x$ does not appear as a summand of $\phi (\s x)$ for each $x \in Q_1$.

Then $\phi$ is cohomologous to a cocycle $\varphi' \in \Hom_{\Bbbk Q_0^\e} (\s^m \Bbbk W_{m}, A)$ where $\varphi' (\s^m w) =0$ for any overlap $w \in W_m$ which is a proper subpath of a maximal overlap (excluding cycles with full relations). 
\end{lemma}

\begin{proof}
Assume that $\phi (\s^m w) \neq 0$ for some overlap $w=x_1\cdots x_m$ which is a proper subpath of a maximal overlap $v$. Write  $v=x_{-k}\cdots x_0 w x_{m+1} \cdots x_n$ (all $x_i$'s are distinct). When $m =1$, we further assume that $\varphi(\s w)$ has no multiples of $w$ as a summand. 

Without loss of generality, we may assume that $\varphi (\s^m x_{j+1} \cdots x_{j+m})=0$ for all $j \leq -1$, namely $w$ is the left most overlap in $v$ satisfying $\varphi (\s^m w) \neq 0$. In particular, $\varphi (\s^m x_0 \cdots x_{m-1}) = 0$. Assume that $\lambda_p p$ is a nonzero summand of $\varphi (\s^m w)$, where $p$ is an irreducible path  parallel to $w$. Since $\delta (\varphi) = 0$ we have
\begin{align*}
\delta(\varphi) (\s^{m+1} w x_{m+1}) &= -(-1)^{\lvert \s x_1 \rvert \lvert \varphi \rvert} x_1 \varphi (\s^m x_2 \cdots x_{m+1}) + (-1)^{|w|-|\varphi|} \varphi(\s^m w) x_{m+1} \\
&= 0. 
\end{align*}

We have two cases. In the first case where the last arrow of $p$ is not $x_m$ so that $px_{m+1} \neq 0$ (since $A$ is gentle), it follows that $\varphi(\s^m w) x_{m+1}$ has a nonzero summand $px_{m+1}$. Then from the above equation we obtain that $\varphi(\s^m x_2 \cdots x_{m+1})$ has a summand $(-1)^{|x_1||\varphi|-|w|}\lambda_p p' x_{m+1}$  such that $p=x_1p'$.  Consider the element $\psi \in \Hom_{\Bbbk Q_0^\e} (\s^{m-1} \Bbbk W_{m-1}, A)$ defined as  
\begin{align*}
\psi (\s^{m-1} x_2 \cdots x_m) &= -(-1)^{|\s x_1||\varphi|} \lambda_p p'
\end{align*}
and $\psi (\s^{m-1} u) = 0$ if $u \neq x_2 \cdots x_m \in W_{m-1}$. Then 
\begin{align*}
\delta (\psi) (\s^m x_1 \cdots x_m)     &= \lambda_p x_1 p' = \lambda_p p \\
\delta (\psi) (\s^m x_2 \cdots x_{m+1}) &= -(-1)^{\epsilon-|\varphi|+|\s x_1||\varphi|} \lambda_p p' x_{m+1}
\end{align*}
and $\delta (\psi) (\s^m v) = 0$ for all the other $v \in W_m$. 

In the second case where the ending arrow of $p$ is $x_m$ (This only happens for $m >1$). Since $\delta(\varphi)(\s^{m+1} x_0w) = 0$ and $\varphi(\s^mx_0\cdots x_{m-1})=0$ it follows that
$x_0p = 0$ (since $A$ is gentle). Namely, $p=x_1p'x_m$ has the starting arrow $x_1$ and the ending arrow $x_m$. Consider the element $\psi \in \Hom_{\Bbbk Q_0^\e}(\s^{m-1}\Bbbk W_{m-1}, A)$ defined as  
\begin{align*}
 \psi(\s^{m-1} x_1\cdots x_{m-1})&=  (-1)^{\epsilon_{m-1}-|\varphi|} \lambda_px_1p'
\end{align*}  
and $\psi(\s^{m-1}u) = 0$ if $u\neq x_1\cdots x_{m-1} \in W_{m-1}$. Then 
\begin{align*}
\delta(\psi)(\s^mx_0\cdots x_{m-1}) &= \pm \lambda_p x_0x_1p'=0\\
\delta(\psi)(\s^mx_1\cdots x_m) &= \lambda_px_1 p'x_{m} =\lambda_pp
\end{align*}
and $\delta(\psi)(\s^mv)=0$ for all the other $v \in W_m$. 

In any case of the above two cases, define $\varphi_1=\varphi-\delta(\psi)$. We find that the number of nonzero summands of 
$\varphi_1(\s^m w)$ is one less than the one of $\varphi(\s^mw)$, since  $\varphi(\s^mw)$ has the summand $\lambda_p p$ but $\varphi_1(\s^nw)$ does not. Moreover, $\varphi_1(\s^mv) = \varphi(\s^mv)$ for all $v \in W_m\backslash \{ w, x_2\cdots x_{m+1}\}$. 

Apply the above process for $\varphi_1$ and iterate this,   we may get a cocycle $\varphi_2$ which is cohomologous to $\varphi$ and satisfies $\varphi_2(\s^mx_{j+1}\cdots x_{j+m}) = 0$  for $j \leq 0$. 
\end{proof}

We now prove our main theorem on the Hochschild cohomology of $\mathcal W (\mathbf S)$ --- which can be related to certain boundary components of $\mathbf S$ --- in terms of the Hochschild cohomology of the corresponding DG gentle algebra. This is a first indication about the relationship between deformations and partial compactifications which will be proved in Theorem~\ref{theorem:deformationwrapped}.

\begin{theorem}
\label{theorem:hochschild}
Let $\mathbf S = (S, \Sigma, \eta)$ be a graded orbifold surface with stops in the boundary and let $\mathcal W (\mathbf S)$ be its partially wrapped Fukaya category. Then the dimension of $\HH^2 (\mathcal W (\mathbf S), \mathcal W (\mathbf S))$ is given by the total number of the following types of boundary components:
\begin{enumerate}
\item boundary components with a single boundary stop and winding number $1$
\item boundary components with a full boundary stop and winding number $1$ or $2$
\item boundary components without stops and winding number $1$ or $2$.
\end{enumerate}
\end{theorem}

We use the isomorphism $\HH^2 (\mathcal W (\mathbf S), \mathcal W (\mathbf S)) \simeq \HH^2 (A, A)$ \eqref{equation:isomorphismhh} where $A = \End (\Gamma)$ is the DG endomorphism algebra of the generator $\Gamma$ corresponding to the dissection of Fig.~\ref{fig:dissection}. The DG quiver with relations for $A$ is given in Proposition \ref{proposition:generator}. We may then use the projective resolution $\mathrm P_\bullet (Q, R)$ given in Section~\ref{subsection:bardzellresolution}.

\begin{proof}[Proof of Theorem \ref{theorem:hochschild}]
As in Proposition \ref{proposition:generator} we choose a boundary component $\partial_0 S$ with $s_0 \geq 1$ stops which exists due to our standing assumptions on $\mathbf S$.

Since $A$ is a DG gentle algebra, it is quadratic monomial, so that we may use the Bardzell resolution of Section~\ref{subsection:bardzellresolution} to compute its Hochschild cohomology.

By Lemma \ref{lemma:killoverlap} we obtain that $\HH^2 (A, A)$ is spanned by cocycles $\phi \in \Hom (\s^n W_n, A)$ of the form $\phi (\s^n w_{1 \dotsc n}) = v$ where $w$ is either (a) a single arrow (i.e.\ $n = 1$), (b) a primitive cycle of length $n \geq 1$ with full relations, (c) a maximal overlap of length $n \geq 1$, or (d) a single vertex (i.e.\ $n = 0$), and $v$ is a path parallel to $w$ and $\phi (w') = 0$ for all other elements of $\s^n W_n$.

By the explicit description of the quiver for $A$ we have the following cocycles corresponding to the cases (a)--(d).

(a) Recall that $\partial_0 S$ denotes the distinguished boundary component of $\mathbf S$. Let us denote the remaining boundary components with boundary stops by $\partial^\III_1 S, \dotsc, \partial^\III_l S$ and let $J \subset \{ 1, \dotsc, l \}$ denote the subset of $1 \leq j \leq l$ such that $\partial^\III_j S$ contains exactly one boundary stop (i.e.\ the boundary components for which $s_j = 1$ in the notation of Proposition \ref{proposition:generator}).

For all $j \in J$ we have cocycles 
\begin{equation}
\label{eq:phiIII}
\phi^\III_j \colon p^\III_{j,1} \mapsto q^\III_j
\end{equation}
which are nontrivial $2$-cocycles whenever $\w_\eta (\partial^\III_j S) = 1$. 

We claim that any other nontrivial $2$-cocycle of form $\varphi(w) = \sum_p \lambda_p p$, where $w$ is a single arrow and $p$ is irreducible path parallel to $w$, is a linear combination of the above $2$-cocycles. Indeed, note that $w$ itself cannot appear as a summand (otherwise, $\varphi$ is in total degree $1$ which can never be a cocycle in $\HH^2(A, A)$). Since $\varphi$ is nonzero in $\HH^2(A, A)$,  by  Lemma \ref{lemma:killoverlap}, $w$ cannot be a proper subpath of a maximal overlap, namely there are no relations involving $w$. As a result, $w$ cannot appear in the parts (\I), (\II) and (\VI) of Fig.~\ref{fig:quiver}. 
Note that there are similar $2$-cocycles $\varphi^\II_j \colon p^\II_{j,1} \mapsto q^\II_j$ for the superscript $\II$, but these are $2$-coboundaries since they are in the image of the (internal) differential of $A$, namely $d (p^\II_{j,1}) = q^\II_j$. If $w=q_i^\V$ with $\lvert q_i^\V \rvert =1$ in the part $(\V)$ of Fig.~\ref{fig:quiver} then $$\varphi(w) = (q_i^\V)^2$$ is a $2$-cocycle which is actually a coboundary. Therefore, $\varphi$ can only be a linear combination of $\varphi^\III_j$. This proves the claim. 

In conclusion, the $2$-cocycles \eqref{eq:phiIII} correspond one-to-one to the boundary components with a single boundary stop and winding number $1$. 

(b) The only primitive cycles with full relations are the cycles $p^\IV_i$ for $1 \leq i \leq m$ corresponding to boundary components $\partial^\IV_1 S, \dotsc, \partial^\IV_m S$ with full boundary stops. A general cochain is of the form ${p^\IV_i}^N \mapsto \lambda e^\IV_i + \mu p^\IV_i$ where $e^\IV_i$ is the source and target of $p^\IV_i$. (Note that $e^\IV_i$ and $p^\IV_i$ are the only parallel paths since ${p^\IV_i}^2 = 0$.) Since $\lvert p^\IV_i \rvert =1- \w_\eta (\partial^\IV_i S)$, the $2$-cochains are spanned by $\phi^\IV_i$ where either 
\begin{align*}
\phi^\IV_i (\s {p^\IV_i}) = e^\IV_i \quad \text{if $\w_\eta (\s^2 \partial^\IV_i S) = 2$} \qquad \text{or} \qquad \phi^\IV_i (\s^2 {p^\IV_i}^2) = e^\IV_i \quad \text{if $\w_\eta (\partial^\IV_i S) = 1$} 
\end{align*}
and $\phi^\IV_i (\s^j w') = 0$ for all other elements of $\s^j W_j$. We claim that $\phi^\IV_i $ is a $2$-cocycle. Indeed, for the first case by the formula we have 
\begin{align*}
\delta(\phi^\IV_i) (\s^2 {p^\IV_i}^2) &= -(-1)^{(|p^\IV_i|-1)|\phi^\IV_i|} p^\IV_i \phi^\IV_i(\s p^\IV_i) + (-1)^{(|p^\IV_i|-1)-|\phi^\IV_i|} \phi^\IV_i(\s p^\IV_i) p^\IV_i\\
&=-p^\IV_i  + p^\IV_i =0
\end{align*}
where we use $|p^\IV_i| =-1$ and $|\phi^\IV_i|=2$.
For the second case we have 
\begin{align*}
\delta(\phi^\IV_i) ( \s^3 {p^\IV_i}^3) &= -(-1)^{(|p^\IV_i|-1)|\phi^\IV_i|} p^\IV_i \phi^\IV_i(\s^2 {p^\IV_i}^2) + (-1)^{2(|p^\IV_i|-1)-|\phi^\IV_i|} \phi^\IV_i(\s^2 {p^\IV}_i^2) p^\IV_i\\
&=-p^\IV_i  + p^\IV_i =0
\end{align*}
where we use $|p^\IV_i|=0$. This shows that $\phi^\IV_i$ is a cocycle. Note that it can never be a coboundary (no options for coboundaries), so it is nonzero in $\HH^2(A, A)$. 

In conclusion, the above $2$-cocycles correspond to the boundary components with a full boundary stop and winding number $1$ or $2$. 

(c) The quiver for $A$ has a maximal overlap $w = w^\VI \dotsb w^\II w^\I$ of length $>1$  where
\begin{align*}
w^\I   &= u^\I_g r^\I_g q^\I_g p^\I_g \dotsb u^\I_2 r^\I_2 q^\I_2 p^\I_2 u^\I_1 r^\I_1 q^\I_1 p^\I_1 & w^\IV &= u^\IV_m \dotsb u^\IV_2 u^\IV_1 \\
w^\II  &= u^\II_k q^\II_k \dotsb u^\II_2 q^\II_2 u^\II_1 q^\II_1                                     & w^\V  &= q^\V_n \dotsb u^\V_2 q^\V_2 u^\V_1 q^\V_1 \\
w^\III &= u^\III_l q^\III_l \dotsb u^\III_2 q^\III_2 u^\III_1 q^\III_1                               & w^\VI &= p^\VI_{s_0 - 2} \dotsb p^\VI_2 p^\VI_1.
\end{align*}
This overlap has a parallel irreducible path $v = u^\V_{n-1} \dotsb p^\I_1 q^\I_1r^\I_1$ if and only if $s_0 = 1$ so that there are no arrows of type $\VI$. Note that $s_0=1$ means that $\partial_0 S$ contains exactly one boundary stop. Consider the cochain $\varphi^\VI (\s^j w) = v$ and $\varphi^\VI (\s^j w') = 0$ for all other elements of $\s^j W_j$. Note that $\varphi^\VI$ is a cocycle automatically since $w$ is a maximal overlap so that $\delta (\varphi^\VI) = 0$. 

We claim that $\varphi^\VI \neq 0$ in $\HH^\bullet(A, A)$. Indeed, assume that there exists a cochain $\psi$ such that $\delta(\psi) (\s^n w) = \varphi(\s^n w)=v$. On the other hand, we have $$\delta(\psi) (\s^n w) = \pm \psi (\s^{n-1} w') p^\I_1 \pm p^\V_n \psi(\s^{n-1} w'')$$
which cannot be equal to $v$ since $v$ does not start with $p^\I_1$ nor end with $p^\V_n$. This proves the claim. 

It follows from the Poincaré--Hopf index theorem that $\varphi^\V$ is a degree $2$ map if and only if $\w_\eta (\partial_0 S) = 1$. By Lemma \ref{lemma:killoverlap} we know that there are no other nonzero $2$-cocycles involving any proper suboverlap of $w$. In particular, if $\mathbf S$ contains some boundary component with at least two boundary stops, we may choose $\partial_0 S$ to be this boundary component, in which case $w$ does not contribute to $\HH^2 (A, A)$.

(d) Finally, let $e$ be some idempotent at a vertex of the quiver in Fig.~\ref{fig:quiver}. Then we have the following two cases. In the first case,  the cochain $\varphi$ has the following form 
\begin{align*}
\varphi(e_{2i-1}^\I) = \lambda_i q_i^\I r_i^\I \quad \text{and} \quad  \varphi(e^\I_{2i}) = \mu_i p_i^\I q_i^\I \quad \text{for $1\leq i \leq g$}
\end{align*}
where $\lambda_i, \mu_i \in \Bbbk$. If $\varphi$ is a cocycle, i.e. $\delta(\varphi) = 0$ then $\delta(\varphi) ( p_i^\I) =  \pm \lambda_i p_i^\I q_i^\I r_i^\I =0 $ which implies that $\lambda_i = 0$ and $\delta (\varphi)(r_i^\I) = \pm \mu_i p_i^\I q_i^\I r_i^\I  = 0$ which implies that $\mu_i = 0$ for each $1\leq i \leq g$. Therefore, $\varphi = 0$, namely there does not exist such a nonzero cocycle. 

In the second case, let $e^\V_i$ denote the idempotent for the source and target of $p^\V_i$ for $1 \leq i \leq n$ and let $\partial^\V_i S$ denote the corresponding boundary components without stops. Since $\lvert p^\V_i \rvert = \w_\eta (\partial^\V_i S)$, we have $2$-cochains $\phi^\V$ given by either 
\begin{align*}
\phi^\V (e^\V_i) &= {q^\V_i}^2 \qquad \text{if $\w_\eta (\partial^\IV_i S) = 2$} 
\end{align*}
or 
\begin{align*}
\phi^\V (e^\V_i) &= q^\V_i \qquad \text{if $\w_\eta (\partial^\IV_i S) = 1$}.
\end{align*}
Similarly we may compute that $\delta(\phi^\V) = 0$.

In conclusion, the above $2$-cocycles correspond to the boundary components without stops and winding number $1$ or $2$. 
\end{proof}

\begin{remark}[Relation to symplectic cohomology]
In Theorem \ref{theorem:hochschild} we give a purely algebraic proof of the fact that Hochschild cohomology of $\mathcal W (\mathbf S)$ can be read off the boundary components with certain winding numbers which is a graded analogue of \cite{chaparroschrollsolotarsuarez}. Although we have only given the result for $\HH^2$, an analogous result can be shown in general. (See \cite{opper,bianschrollsolotarwangwen} for the case of graded gentle and skew-gentle algebras, respectively.) From a more geometric perspective, Hochschild cohomology may be related to symplectic cohomology via the closed--open map
\[
\mathcal C \mathcal O \colon \mathrm{SH}^\bullet (S, \Sigma, \eta) \to \HH^\bullet (\mathcal W (\mathbf S), \mathcal W (\mathbf S))
\]
which is an isomorphism in the fully wrapped case \cite{ganatra1,ganatra2}. 
\end{remark}

\begin{remark}[A$_\infty$ deformations from general dissections]
\label{remark:general}
In our description thus far, we have chosen a special generator of $\mathcal W (\mathbf S)$ whose endomorphism algebra is a DG gentle algebra $A$ such that all $2$-cocycles correspond to deformations of $A$ as a curved DG algebra. Other choices of generators may naturally lead to $2$-cocycles corresponding to curved A$_\infty$ deformations with nontrivial higher multiplications. Whereas the relationship of $\HH^2 (\mathcal W (\mathbf S), \mathcal W (\mathbf S))$ to certain boundary components of $\mathbf S$ as in Theorem \ref{theorem:hochschild} remains true, the $2$-cocycles will take a different shape and the corresponding deformations will generally also feature nontrivial higher multiplications. For example, a boundary component with winding number $1$ and one stop $\sigma$ will contribute to an A$_\infty$ deformation with a higher product of the form
\begin{equation}
\label{eq:higherproduct}
\mu^{n-1} (\s p_{n-1} \otimes \dotsb \otimes \s p_2 \otimes \s p_1) = \s q_k \dotsb q_2 q_1
\end{equation}
whenever the stop $\sigma$ lies inside an $n$-gon. In this case we have $\sum_{1 \leq i \leq k} \lvert q_i \rvert - \sum_{1 \leq j \leq {n-1}} \lvert p_j \rvert = 3 - n$ and the endomorphism algebra of the corresponding generator of $\mathbf S$ contains a subquiver of the form
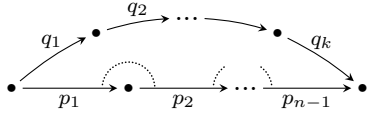
\[
\begin{tikzpicture}[x=1em,y=1em]
\node[shape=circle,inner sep=1.5pt, scale=.7] (0) at (0,0) {$\bullet$};
\node[shape=circle,inner sep=1.5pt, scale=.7] (1) at (4,0) {$\bullet$};
\node[shape=circle,inner sep=1.5pt] (d) at (8,0) {$...$};
\node[shape=circle,inner sep=1.5pt, scale=.7] (n) at (12,0) {$\bullet$};
\node[shape=circle,inner sep=1.5pt, scale=.7] (1') at (2.9,1.9) {$\bullet$};
\node[shape=circle,inner sep=1.5pt] (d') at (6,2.4) {$...$};
\node[shape=circle,inner sep=1.5pt, scale=.7] (n-1') at (9.1,1.9) {$\bullet$};
\path[->, line width=.4pt] (0) edge node[font=\scriptsize, below=-.3ex] {$p_1$} (1) (1) edge node[font=\scriptsize, below=-.3ex] {$p_2$} (d) (d) edge node[font=\scriptsize, below=-.3ex] {$p_{n-1}$} (n) (0) edge[bend left=6] node[font=\scriptsize, above=-.1ex] {$q_1$} (1'.203) (1') edge[bend left=6] node[font=\scriptsize, above=-.3ex] {$q_2$} (d'.180) (d'.0) edge[bend left=6] (n-1') (n-1') edge[bend left=6] node[font=\scriptsize, above=-.1ex] {$q_k$} (n);
\draw[dash pattern=on 0pt off 1.2pt, line width=.6pt, line cap=round] (1) ++(172:.9em) arc[start angle=172, end angle=5, radius=.9em] (7.8,0) ++(172:.9em) arc[start angle=172, end angle=120, radius=.9em] (8,0) ++(8:.9em) arc[start angle=8, end angle=60, radius=.9em];
\end{tikzpicture}
\]
and the A$_\infty$ deformation corresponds to introducing the nontrivial higher product $\mu^{n-1}$ \eqref{eq:higherproduct} on this subquiver. Considering only standard dissections, as we do in this article, ensures that this stop $\sigma$ lies in a $2$-gon (see Fig.~\ref{fig:dissection}), whence in this case the $2$-cocycle corresponds to a differential $\mu^1$ giving rise to a DG deformation (changing the differential). An explicit description of the A$_\infty$ structure for $2$-cocycles from more general dissections is given in \cite{barmeierschrollwang}.

Similarly, a boundary component with full boundary stop and winding number $1$ will give rise to a $2$-cocycle defined by
\[
\mu^{2n} (p_{i+2n} \otimes \dotsb \otimes \s p_{i + 2} \otimes \s p_{i+1}) = \s e_{\mathrm t (p_i)}, \qquad 1 \leq i \leq n
\]
(see \cite{barmeierwang4}) where the indices are to be understood cyclically and $p_1, \dotsc, p_n$ form a cycle with full quadratic relations, i.e.\ $p_{i+1} p_i = 0$ for all $i$. Again, the special case of a standard dissection ensures that $n = 1$, giving rise to a cocycle $\mu^2 (\s p \otimes \s p) = e_{\mathrm t (p)}$ parametrizing again a DG deformation (changing the underlying associative structure).
\end{remark}

\section{A solution to the curvature problem via weak duality}
\label{section:curved}

In this section, we show that the curvature problem exists for $\mathcal W (\mathbf S)$ whenever $\mathbf S$ has boundary components without stops whose winding number is $1$ or $2$. We then introduce a notion of {\it weak duality} for graded gentle and DG gentle algebras to bypass the curvature problem.

In the study of deformations of Fukaya categories, often considered in the context of {\it relative Fukaya categories}, a common strategy is to embrace the curvature and study formal curved A$_\infty$ deformations of Fukaya categories, see for example \cite{seidel1,sheridan2,perutzsheridan}. In the context of partially wrapped Fukaya categories of surfaces, it turns out that -- after a few intermediate steps -- we can completely circumvent the curvature problem. This has a number of pleasant consequences. Firstly, the absence of curved deformations allows us to deploy the usual techniques for pretriangulated A$_\infty$ categories. Notably, we are able to pass to the (triangulated) homotopy category after deformation. As we are able to construct not only formal deformations, but also a semi-universal family of ``strict'' deformations over $\mathbb A^{\dim \HH^2}$ in the sense of Section~\ref{subsection:families}, we are able to speak of individual fibers of this family, each fiber being obtained by evaluating the deformation parameters to constants. In particular, the fibers of this family are pretriangulated DG categories free of any formal deformation parameters and can be interpreted as intrinsic geometric objects: as partially wrapped Fukaya categories obtained via partial orbifold compactifications of the initial surface (see Section~\ref{section:deformation}).

\subsection{Two examples with curvature problem}

Before addressing the general case, we discuss the curvature problem and its resolution for the simplest two graded surfaces where the curvature problem arises: a cylinder with one boundary stop with two different choices of grading structure. In this case, weak duality --- which is introduced in Section \ref{subsection:weakduality} to resolve the curvature problem for an arbitrary graded (orbifold) surface --- coincides with classical Koszul duality. Koszul duality is not able to resolve the general case, but our notion of weak duality is closely related and these two examples will be useful in illustrating the general strategy for resolving the curvature problem.  

Let $S = [0, 1] \times \mathrm S^1$ be the cylinder/annulus with boundary components $\partial_0 S = \{ 0 \} \times \mathrm S^1$ and $\partial_1 S = \{ 1 \} \times \mathrm S^1$. Let $\Sigma = \{ (0, 1) \}$, consisting of a single boundary stop on $\partial_0 S$, and let $\eta_1$ and $\eta_2$ be the two line fields on $S$ such that the winding numbers along the boundary component $\partial_1 S$ without stop are $\mathrm w_{\eta_1} (\partial_1 S) = 1$ and $\mathrm w_{\eta_2} (\partial_1 S) = 2$, respectively.

Setting $\mathbf S_1 = (S, \Sigma, \eta_1)$ and $\mathbf S_2 = (S, \Sigma, \eta_2)$ we have that
\[
\mathcal W (\mathbf S_1) \simeq \tw (A_1) \qquad \text{and} \qquad \mathcal W (\mathbf S_2) \simeq \tw (A_2)
\]
where $A_1 = \Bbbk [q_1]$ and $A_2 = \Bbbk [q_2]$ are the univariate polynomial rings with a single generator $q_1$ resp.\ $q_2$ of degree $1$ resp.\ $2$. See Fig.~\ref{fig:cylinder} for an illustration of $\mathbf S_1$ and $\mathbf S_2$ with its standard dissection, where we have also drawn the two line fields $\eta_1$ and $\eta_2$.

\begin{figure}
\centering
\begin{tikzpicture}[x=1em,y=1em,decoration={markings,mark=at position 0.52 with {\arrow[black]{Stealth[length=4.2pt]}}}]
\begin{scope} 
\node[font=\scriptsize,left] at (-2.5,2.5) {$\mathbf S_1$};
\node[font=\scriptsize, right, align=center] at (4em, 0) {$A_1 = \Bbbk [q_1]$ \\[.5em] $\lvert q_1 \rvert = 1$};
\draw[line width=.5pt] (0,0) circle(3.186em);
\node[font=\scriptsize] at (0,0) {$\times$};
\draw[fill=black, color=black] (90:3.186em) circle(.15em);
\begin{scope}[rotate=0]
\draw[line width=.1pt] (180:3.18) to (0,0);
\draw[line width=.1pt] (2*19:3.18) to[bend left=26] (-2*19:3.18);
\draw[line width=.1pt] (2*27:3.18) to[bend left=40] (-2*27:3.18);
\draw[line width=.1pt] (2*33.3:3.18) to[bend left=50,looseness=1.07] (-2*33.3:3.18);
\draw[line width=.1pt] (2*38.9:3.18) to[bend left=57,looseness=1.19] (-2*38.9:3.18);
\draw[line width=.1pt] (2*43.6:3.18) to[bend left=60,looseness=1.34] (-2*43.6:3.18);
\draw[line width=.1pt] (2*48:3.18) to[bend left=62,looseness=1.51] (-2*48:3.18);
\draw[line width=.1pt] (2*52.2:3.18) to[bend left=64,looseness=1.69] (-2*52.2:3.18);
\draw[line width=.1pt] (2*56.2:3.18) to[bend left=66,looseness=1.9] (-2*56.2:3.18);
\draw[line width=.1pt] (2*60:3.18) to[bend left=68,looseness=2.13] (-2*60:3.18);
\draw[line width=.1pt] (2*63.5:3.18) to[bend left=70,looseness=2.39] (-2*63.5:3.18);
\draw[line width=.1pt] (2*67.1:3.18) to[bend left=72,looseness=2.73] (-2*67.1:3.18);
\draw[line width=.1pt] (2*70.5:3.18) to[bend left=74,looseness=3.14] (-2*70.5:3.18);
\draw[line width=.1pt] (2*73.9:3.18) to[bend left=76,looseness=3.72] (-2*73.9:3.18);
\draw[line width=.1pt] (2*77.1:3.18) to[bend left=78,looseness=4.52] (-2*77.1:3.18);
\draw[line width=.1pt] (2*80.3:3.18) to[bend left=82,looseness=5.78] (-2*80.3:3.18);
\draw[line width=.1pt] (2*83.5:3.18) to[bend left=86,looseness=8.3] (-2*83.5:3.18);
\draw[line width=.1pt] (2*86.8:3.18) to[bend left=89,looseness=16.15] (-2*86.8:3.18);
\draw[line width=.5pt,fill=white] (0,0) circle(1.41em);
\end{scope}
\draw[-, line width=.75pt,color=arccolour] (270:1.41em) to (270:3.186em);
\draw[line width=0pt,postaction={decorate}] (85:1.41em) to ++(0:.05em);
\node[font=\scriptsize, below] at (90:1.4em) {$q_1$};
\end{scope}
\begin{scope}[xshift=15em] 
\node[font=\scriptsize,left] at (-2.5,2.5) {$\mathbf S_2$};
\node[font=\scriptsize, right, align=center] at (4em, 0) {$A_2 = \Bbbk [q_2]$ \\[.5em] $\lvert q_2 \rvert = 2$};
\draw[fill=black, color=black] (90:3.186em) circle(.15em);
\draw[clip] (0,0) circle(3.18em);
\foreach \r in {-4.25,-4,...,4.25} {%
\draw[line width=.1pt] (-4.5,\r) -- (4.5,\r);
}
\draw[line width=.5pt,fill=white] (0,0) circle(1.41em);
\draw[-, line width=.75pt,color=arccolour] (270:1.41em) to (270:3.186em);
\draw[line width=0pt,postaction={decorate}] (85:1.41em) to ++(0:.05em);
\node[font=\scriptsize, below] at (90:1.4em) {$q_2$};
\end{scope}
\end{tikzpicture}
\caption{Two grading structures on the cylinder with one boundary stop and the corresponding endomorphism algebras of the generating curve}
\label{fig:cylinder}
\end{figure}
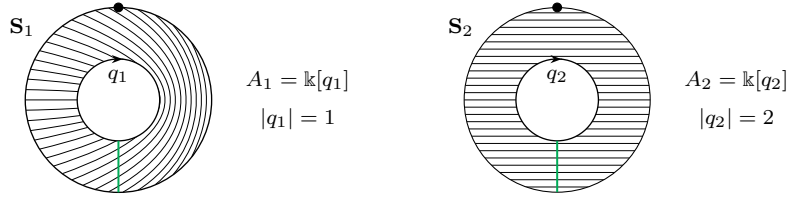

\subsubsection{The curvature problem}
\label{subsubsection:curvature}

It follows from the proof of Theorem \ref{theorem:hochschild} that we have $\dim \HH^2 (A_1, A_1) = \dim \HH^2 (A_2, A_2) = 1$ and $A_1$ and $A_2$ both have a natural $2$-cocycle $\phi_1 \colon \Bbbk \to A_1$ resp.\ $\phi_2 \colon \Bbbk \to A_2$ parametrizing a $1$-parameter family of curved deformations. These two $2$-cocycles are given by
\begin{equation}
\label{eq:phi}
\phi_1 (1) = q_1^2 \qquad \text{and} \qquad \phi_2 (1) = q_2.
\end{equation}
(Note that these cocycles are of Hochschild degree $0$, but being maps of ``internal'' degree $2$, they have total degree $2$, thus defining $2$-cocycles in the total degree.) Correspondingly, for both $A_1$ and $A_2$ we may define a family $\{ A_{1, \lambda} \}_{\lambda \in \mathbb A^1}$ resp.\ $\{ A_{2, \lambda} \}_{\lambda \in \mathbb A^1}$ of curved deformations over $\mathbb A^1$. For each $\lambda \in \Bbbk$ we set
\begin{equation}
\label{eq:curvedfamily}
A_{1, \lambda} = (A_1, \mu^0_\lambda) \qquad \text{and} \qquad A_{2, \lambda} = (A_2, \mu^0_\lambda)
\end{equation}
where for $A_{1, \lambda}$ resp.\ $A_{2, \lambda}$ the curvature term is given by $\mu^0_\lambda = \lambda q_1^2$ resp.\ $\mu^0_\lambda = \lambda q_2$. Note that for $\lambda = 0$, the curvature term vanishes and we recover $A_1 = A_{1, 0}$ and $A_2 = A_{2, 0}$.

As we will show in Section \ref{subsection:curvatureproblem}, there does not exist another classical generator of $\mathcal W (\mathbf S_1)$ or $\mathcal W (\mathbf S_2)$ such that its endomorphism algebra does not admit any curved deformations. Indeed, any classical generator must contain at least one curve connecting the two boundary components, and thus contain a copy of the curve present in the dissection in Fig.~\ref{fig:cylinder}.

\subsubsection{A weak generator and the Koszul dual}

Whereas no classical generator of $\mathrm D \mathcal W (\mathbf S_1)$ and $\mathrm D \mathcal W (\mathbf S_2)$ has an endomorphism algebra without curved deformations, by Theorem \ref{theorem:weak} both categories admit a {\it weak generator} $\Gamma^\vee$ corresponding to the curve around the boundary component as illustrated in Fig.~\ref{fig:cylinderweak}. We call $B_i = A_i^\vee = \End (\Gamma^\vee)$ the {\it weak dual} of $A_i$, which in this case coincides with the Koszul dual.

\begin{figure}
\centering
\begin{tikzpicture}[x=1em,y=1em,decoration={markings,mark=at position 0.52 with {\arrow[black]{Stealth[length=4.2pt]}}}]
\begin{scope} 
\node[font=\scriptsize,left] at (-2.5,2.5) {$\mathbf S_1$};
\node[font=\scriptsize, right, align=center] at (4em, 0) {$B_1 = \Bbbk [p_1] / (p_1^2)$ \\[.5em] $\lvert p_1 \rvert = 0$};
\draw[line width=.5pt] (0,0) circle(3.186em);
\node[font=\scriptsize] at (0,0) {$\times$};
\draw[fill=black, color=black] (90:3.186em) circle(.15em);
\draw[line width=.5pt,fill=white] (0,0) circle(1.41em);
\draw[line width=.75pt, line cap=round, color=arccolour] (245:3.186em) to[out=85, in=240, looseness=.8] (150:1.8em) arc[start angle=150, end angle=30, radius=1.8] to[out=300, in=95, looseness=.8] (295:3.186em);
\draw[line width=0pt,postaction={decorate}] (272:3.186em) to ++(0:.05em);
\node[font=\scriptsize, below] at (270:3.186em) {$p_1$};
\end{scope}
\begin{scope}[xshift=15em] 
\node[font=\scriptsize,left] at (-2.5,2.5) {$\mathbf S_2$};
\node[font=\scriptsize, right, align=center] at (4em, 0) {$B_2 = \Bbbk [p_2] / (p_2^2)$ \\[.5em] $\lvert p_2 \rvert = -1$};
\draw[fill=black, color=black] (90:3.186em) circle(.15em);
\draw (0,0) circle(3.18em);
\draw[line width=.5pt,fill=white] (0,0) circle(1.41em);
\draw[line width=.75pt, line cap=round, color=arccolour] (245:3.186em) to[out=85, in=240, looseness=.8] (150:1.8em) arc[start angle=150, end angle=30, radius=1.8] to[out=300, in=95, looseness=.8] (295:3.186em);
\draw[line width=0pt,postaction={decorate}] (272:3.186em) to ++(0:.05em);
\node[font=\scriptsize, below] at (270:3.186em) {$p_2$};
\end{scope}
\end{tikzpicture}
\caption{Weak generators of $\mathrm D \mathcal W (\mathbf S_1)$ and $\mathrm D \mathcal W (\mathbf S_2)$ given by a curve around the boundary component without stop and their endomorphism algebras}
\label{fig:cylinderweak}
\end{figure}
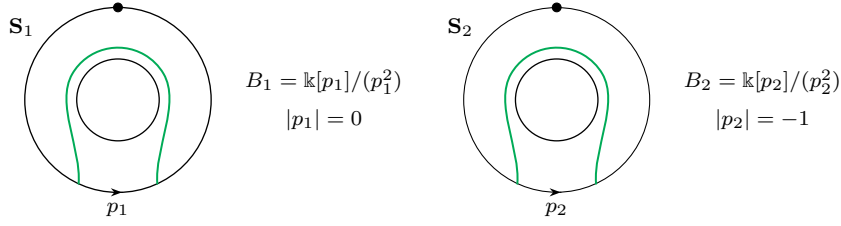

We have $B_1 \simeq {A_1}^!$ and $B_2 \simeq {A_2}^!$ and by \cite[10.5 Lemma]{keller94} triangulated equivalences
\[
\per (A_1) \simeq \mathrm D^\b (B_1) \qquad \text{and} \qquad \per (A_2) \simeq \mathrm D^\b (B_2).
\]
These triangulated equivalences can be lifted to the pretriangulated level as follows.

The algebras $B_1$ and $B_2$ are finite-dimensional graded gentle algebras (which are proper as DG algebras). Their bounded derived categories admit a natural DG enhancement given by the category $\Tw^{-, \b} (B_i)$ for $i = 1, 2$ (see Section \ref{subsection:unbounded}). In this simple example, $\Tw^{-, \b} (B_i)$ is the category generated by the unbounded one-sided twisted complex consisting of the semi-free (cofibrant) resolution of the simple $B_i$-module $B_i / \rad (B_i) \simeq \Bbbk$ (concentrated in degree $0$) on which $p_i$ acts by $0$. The superscript in $\Tw^{-, \b}$ is supposed to be reminiscent of the equivalence $\mathrm K^{-, \b} (\proj B) \simeq \mathrm D^{\b} (B)$ for a finite-dimensional zero-graded algebra $B$, in which case we indeed have $\H^0 \Tw^{-, \b} (\add B) \simeq \mathrm K^{-, \b} (\proj B)$.

Taking into account the grading, we have that $\Tw^{-, \b} (B_1)$ is generated by a single twisted complex whose corresponding DG module is given by the cochain complex
\begin{equation*}
\begin{tikzpicture}[baseline=-2.6pt]
\matrix (m) [matrix of math nodes, row sep=.75em, text height=1.5ex, column sep=1.5em, text depth=0.25ex, ampersand replacement=\&, inner sep=3pt]
{
       \& \text{\footnotesize $-5$} \& \text{\footnotesize $-4$} \& \text{\footnotesize $-3$} \& \text{\footnotesize $-2$} \& \text{\footnotesize $-1$} \& \text{\footnotesize $0$} \& \text{\footnotesize $1$} \& \\
\cdots \& (\Bbbk \, p_1) [4] \& (\Bbbk \, 1) [4] \& (\Bbbk \, p_1) [2] \& (\Bbbk \, 1) [2] \& \Bbbk \, p_1 \& \Bbbk \, 1 \& 0 \& \cdots \\
};
\path[->,line width=.4pt]
(m-2-1) edge node[font=\scriptsize, above=-.2ex] {$p_1 \cdot$} (m-2-2)
(m-2-3) edge node[font=\scriptsize, above=-.2ex] {$p_1 \cdot$} (m-2-4)
(m-2-5) edge node[font=\scriptsize, above=-.2ex] {$p_1 \cdot$} (m-2-6)
;
\end{tikzpicture}
\end{equation*}
(see \eqref{eq:veeK} for the representation as twisted complex) and $\Tw^{-, \b} (B_2)$ is generated by
\begin{equation*}
\begin{tikzpicture}[baseline=-2.6pt]
\matrix (m) [matrix of math nodes, row sep=.75em, text height=1.5ex, column sep=1.5em, text depth=0.25ex, ampersand replacement=\&, inner sep=3pt]
{
       \& \text{\footnotesize $-4$} \& \text{\footnotesize $-3$} \& \text{\footnotesize $-2$} \& \text{\footnotesize $-1$} \& \text{\footnotesize $0$} \& \text{\footnotesize $1$} \& \\
\& (\Bbbk \, 1) [4]   \& (\Bbbk \, 1) [3]   \& (\Bbbk \, 1) [2]   \& (\Bbbk \, 1) [1]   \& \Bbbk \, 1   \& \& \\[-.95em]
\cdots \& \oplus \& \oplus \& \oplus \& \oplus \& \oplus \& 0 \& \cdots \\[-.95em]
\& (\Bbbk \, p_2) [4] \& (\Bbbk \, p_2) [3] \& (\Bbbk \, p_2) [2] \& (\Bbbk \, p_2) [1] \& \Bbbk \, p_2 \& \& \\
};
\path[->,line width=.4pt]
(m-2-2.south east) edge node[font=\scriptsize, above=.2ex, pos=.7] {$p_2 \cdot$} (m-4-3.north west)
(m-2-3.south east) edge node[font=\scriptsize, above=.2ex, pos=.7] {$p_2 \cdot$} (m-4-4.north west)
(m-2-4.south east) edge node[font=\scriptsize, above=.2ex, pos=.7] {$p_2 \cdot$} (m-4-5.north west)
(m-2-5.south east) edge node[font=\scriptsize, above=.2ex, pos=.7] {$p_2 \cdot$} (m-4-6.north west)
;
\end{tikzpicture}
\end{equation*}
which in Loewy notation is nothing but the projective resolution $\cdots \to \begin{smallmatrix} 1 \\ 1 \end{smallmatrix} \to \begin{smallmatrix} 1 \\ 1 \end{smallmatrix} \to \begin{smallmatrix} 1 \\ 1 \end{smallmatrix}$ of the simple $B_2$-module $1$ (here $1$ denoting the vertex of $B_2$).

On the pretriangulated level, we thus obtain equivalences
\begin{equation}
\label{eq:equiv}
\mathcal W (\mathbf S_1) \simeq \tw (A_1) \simeq \Tw^{-, \b} (B_1) \qquad \text{and} \qquad \mathcal W (\mathbf S_2) \simeq \tw (A_2) \simeq \Tw^{-, \b} (B_2).
\end{equation}

Note that the endomorphism algebras of these two unbounded twisted complexes are isomorphic to $A_1$ and $A_2$, respectively, thus recovering the original algebra after a double dual.

\subsubsection{A solution to the curvature problem}

Phrased in terms of $A_i$ for $i = 1, 2$, we saw in Section \ref{subsubsection:curvature} that deformations of $\mathcal W (\mathbf S_i) \simeq \tw (A_i)$ are naturally subject to the curvature problem, as $A_i$ only admits curved deformations. However, the equivalences \eqref{eq:equiv} now show that deformations of $\mathrm D \mathcal W (\mathbf S_i) \simeq \mathrm D^\b (B_i)$ are equivalent to deformations of $\Tw^{-, \b} (B_i)$ (see Section \ref{subsection:solutioncurvature}) which in turn are equivalent to deformations of $B_i$, where $B_i$ has no curved deformations. Indeed, $B_1 = \Bbbk [p_1] / (p_1^2)$ with $\lvert p_1 \rvert = 0$ admits a $2$-cocycle determined by $\psi_1 (p_1 \otimes p_1) = 1$ which gives rise to a $1$-parameter family of associative deformations of $B_1$
\[
B_{1, \lambda} = \Bbbk [p_1] / (p_1^2 - \lambda).
\]
Similarly, $B_2$ admits a $2$-cocycle $\psi_2 (p_2) = 1$ giving rise to a $1$-parameter family of DG deformations of $B_2$
\[
B_{2, \lambda} = (\Bbbk [p_2] / (p_2^2), \mu^1_\lambda)
\]
where $\mu^1_\lambda$ is the differential determined by $\mu^1_\lambda (p_1) = \lambda \, 1$. Note that $\psi_1$ and $\psi_2$ are ``dual'' to $\phi_1$ and $\phi_2$ in \eqref{eq:phi} --- $\psi_1$ maps a length $2$ path to a length $0$ path whereas $\phi_1$ maps a length $0$ path to a length $2$ path and, similarly, $\psi_2$ maps a length $1$ path to a length $0$ path, whereas $\phi_2$ maps a length $0$ to a length $1$ path. 

\subsubsection{Geometric interpretation of the nontrivial deformations}

For $\lambda \in \Bbbk \smallsetminus \{ 0 \}$, the fibers $A_{i, \lambda}$ of the family \eqref{eq:curvedfamily} are curved A$_\infty$ algebras, for $i = 1, 2$, thus obstructing the usual passage to the triangulated homotopy category. In contrast, the fibers $B_{i, \lambda}$, for both $i = 1, 2$, are DG algebras. Thus the category $\Tw^{-, \b} (B_{1, \lambda})$ is still pretriangulated for all $\lambda \in \Bbbk$.

The case $\lambda = 0$ recovers $\mathcal W (\mathbf S_1)$ and $\mathcal W (\mathbf S_2)$, and for $\lambda \neq 0$, we have
\[
B_{1, \lambda} \simeq \Bbbk \times \Bbbk \qquad \text{and} \qquad B_{2, \lambda} \overset{\text{\tiny qis}}\simeq 0.
\]
In fact, we then have
\[
\Tw^{-, \b} (B_{i, \lambda}) \simeq \mathcal W (\mathbf S_{i, \lambda}), \quad i = 1, 2
\]
where for all $\lambda \neq 0$, we have
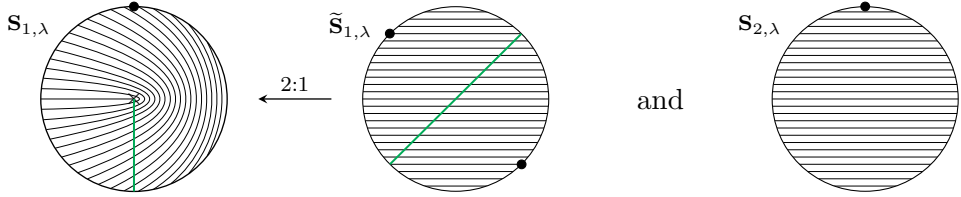
\[
\begin{tikzpicture}[x=1em,y=1em,decoration={markings,mark=at position 0.52 with {\arrow[black]{Stealth[length=4.2pt]}}}]
\begin{scope} 
\node[font=\scriptsize,left] at (-2.5,2.5) {$\mathbf S_{1, \lambda}$};
\draw[line width=.5pt] (0,0) circle(3.186em);
\node[font=\scriptsize] at (0,0) {$\times$};
\draw[fill=black, color=black] (90:3.186em) circle(.15em);
\path[<-, line width=.5pt] (4.25em,0) edge node[font=\scriptsize, above=-.2ex] {2:1} ++(2.5em,0);
\begin{scope}[rotate=0]
\draw[line width=.1pt] (180:3.18) to (0,0);
\draw[line width=.1pt] (2*19:3.18) to[bend left=26] (-2*19:3.18);
\draw[line width=.1pt] (2*27:3.18) to[bend left=40] (-2*27:3.18);
\draw[line width=.1pt] (2*33.3:3.18) to[bend left=50,looseness=1.07] (-2*33.3:3.18);
\draw[line width=.1pt] (2*38.9:3.18) to[bend left=57,looseness=1.19] (-2*38.9:3.18);
\draw[line width=.1pt] (2*43.6:3.18) to[bend left=60,looseness=1.34] (-2*43.6:3.18);
\draw[line width=.1pt] (2*48:3.18) to[bend left=62,looseness=1.51] (-2*48:3.18);
\draw[line width=.1pt] (2*52.2:3.18) to[bend left=64,looseness=1.69] (-2*52.2:3.18);
\draw[line width=.1pt] (2*56.2:3.18) to[bend left=66,looseness=1.9] (-2*56.2:3.18);
\draw[line width=.1pt] (2*60:3.18) to[bend left=68,looseness=2.13] (-2*60:3.18);
\draw[line width=.1pt] (2*63.5:3.18) to[bend left=70,looseness=2.39] (-2*63.5:3.18);
\draw[line width=.1pt] (2*67.1:3.18) to[bend left=72,looseness=2.73] (-2*67.1:3.18);
\draw[line width=.1pt] (2*70.5:3.18) to[bend left=74,looseness=3.14] (-2*70.5:3.18);
\draw[line width=.1pt] (2*73.9:3.18) to[bend left=76,looseness=3.72] (-2*73.9:3.18);
\draw[line width=.1pt] (2*77.1:3.18) to[bend left=78,looseness=4.52] (-2*77.1:3.18);
\draw[line width=.1pt] (2*80.3:3.18) to[bend left=82,looseness=5.78] (-2*80.3:3.18);
\draw[line width=.1pt] (2*83.5:3.18) to[bend left=86,looseness=8.3] (-2*83.5:3.18);
\draw[line width=.1pt] (2*86.8:3.18) to[bend left=89,looseness=16.15] (-2*86.8:3.18);
\end{scope}
\draw[-, line width=.75pt,color=arccolour] (0,0) to (270:3.186em);
\end{scope}
\begin{scope}[xshift=11em] 
\node[font=\scriptsize,left] at (-2.5,2.5) {$\widetilde{\mathbf S}_{1, \lambda}$};
\draw[fill=black, color=black] (135:3.186em) circle(.15em);
\draw[fill=black, color=black] (-45:3.186em) circle(.15em);
\node at (7em, 0) {and};
\draw[clip] (0,0) circle(3.18em);
\foreach \r in {-4.25,-4,...,4.25} {%
\draw[line width=.1pt] (-4.5,\r) -- (4.5,\r);
}
\draw[-, line width=.75pt,color=arccolour] (225:3.186em) to (45:3.186em);
\end{scope}
\begin{scope}[xshift=25em] 
\node[font=\scriptsize,left] at (-2.5,2.5) {$\mathbf S_{2, \lambda}$};
\draw[fill=black, color=black] (90:3.186em) circle(.15em);
\draw[clip] (0,0) circle(3.18em);
\foreach \r in {-4.25,-4,...,4.25} {%
\draw[line width=.1pt] (-4.5,\r) -- (4.5,\r);
}
\end{scope}
\end{tikzpicture}
\]
For $\lambda \neq 0$ the surface $\mathbf S_{1, \lambda}$ is an orbifold disk with one stop, obtained from $\mathbf S_1$ by ``partially compactifying'' the inner boundary component to a point. Since the inner boundary component of $\mathbf S_1$ has winding number $1$, this point is necessarily an orbifold point of order $2$. Here we have also drawn the double cover $\widetilde{\mathbf S}_{1, \lambda}$ which is a smooth disk with two stops. Note that $\mathcal W (\mathbf S_{1, \lambda}) \simeq (\mathcal W (\widetilde{\mathbf S}_{1, \lambda}) / \mathbb Z_2)^\natural$ (see \cite{chokim,amiotplamondon,barmeierschrollwang}). The category $\mathcal W (\widetilde{\mathbf S}_{1, \lambda})$ of the double cover contains exactly one object $X$ with $\End (X) \simeq \Bbbk$, whence its orbit category contains a single object $X$ with endomorphism algebra $\Bbbk \otimes \Bbbk [\mathbb Z_2] \simeq \Bbbk \times \Bbbk$. The idempotent completion (denoted by $-^\natural$) thus contains exactly two objects without any nontrivial morphisms between them, sometimes represented by pair of arcs, one ``tagged'' and one ``untagged''.

Similarly, for $\lambda \neq 0$ we have that $\mathbf S_{2, \lambda}$ is a (smooth) disk with one stop, again obtained from $\mathbf S_2$ by compactifying the inner boundary component to a (smooth) point. The fact that $B_{2, \lambda}$ is quasi-isomorphic to the zero algebra is also consistent with the fact that all curves in $\mathbf S_{2, \lambda}$ are homotopic to the boundary. From a more geometric perspective, any (exact) Lagrangian submanifold $\gamma$ has an acyclic partially wrapped Floer complex, since any such curve admits a boundary morphism $p_2$ of degree $-1$ which appears in a pseudoholomorphic disk defining a Floer differential $\partial (p_2) = \id_\gamma$.

This shows that the heuristic that algebraic deformations should correspond to partial compactifications can be shown to hold ``on the nose'' for strict deformations, after passing to the intermediary category $\Tw^{-, \b} (B_i)$ that avoid the curvature problem.

\subsection{The curvature problem}
\label{subsection:curvatureproblem}

If $\mathbf S$ has boundary components without stops whose winding number is $1$ or $2$, it turns out that the Hochschild cohomology of $\mathcal W (\mathbf S)$ computed via {\it any} classical generator $\Gamma$ has nontrivial $2$-cocycles parametrizing curved deformations of the endomorphism algebra of $\Gamma$.

Curvature entails that the differential on $\Hom$-spaces no longer squares to zero. In particular, the homotopy category (cf.\ Remark \ref{remark:triangulated}) is no longer well-defined. This is known as the {\it curvature problem}, as it is no longer straightforward to associate a meaningful triangulated category to such a curved deformation \cite{positselski,kellerlowennicolas,lowenvandenbergh3}.

There are several recent developments to circumvent or overcome the curvature problem, for example by considering deformations of triangulated categories or stable $\infty$-categories together with a t-structure \cite{genoveselowenvandenbergh,genoveselowensymonsvandenbergh,iwanari} or by defining new triangulated categories in the presence of curvature \cite{positselski,guanlazarev,guanholsteinlazarev,lehmann,lehmannlowen,lehmannlowen2}.

The following proposition shows that describing deformations of $\mathcal W (\mathbf S)$ via any classical generator always leads to the curvature problem.

\begin{proposition}
\label{proposition:curveddeformations}
Let $\mathbf S$ be a surface with boundary components without stops and winding number $1$ and/or $2$. Then the endomorphism algebra of any classical generator of $\mathcal W (\mathbf S)$ admits curved deformations. Consequently, the full deformation theory of $\mathcal W (\mathbf S)$ cannot be described in terms of classical generators without considering curved deformations.
\end{proposition}

\begin{proof}
Let $\partial^\V_j S$ be a boundary component with winding number $1$ and no stops. Any classical generator $\Gamma$ must contain at least one curve $\gamma$ connecting to $\partial^\V_j S$ (see Fig.~\ref{fig:dissection}). Since $\mathrm D \mathcal W (\mathbf S) \simeq (\mathrm D \mathcal W (\widetilde{\mathbf S}) / \mathbb Z_2)^\natural$ \cite{chokim,amiotplamondon,barmeierschrollwang}, we may view objects in $\mathrm D \mathcal W (\mathbf S)$ as $\mathbb Z_2$-invariant objects in the (derived) partially wrapped Fukaya category $\mathrm D \mathcal W \smash{(\widetilde{\mathbf S})}$ of a double cover $\widetilde{\mathbf S}$ of $\mathbf S$, possibly together with an idempotent (see \cite[Section 2.4]{barmeierschrollwang}) and use the classification of morphisms for the latter \cite{arnesenlakingpauksztello,opperplamondonschroll}. This classification implies that $\End_{\mathrm D \mathcal W (\mathbf S)} (\Gamma)$ is generated by boundary paths, orbifold paths and pairs of morphisms corresponding to intersections of curves in $\Gamma$ in the interior of $\mathbf S$. Now consider $\End (\gamma)$ in which case we have only boundary paths from $\gamma$ to itself and pairs of morphisms for each self-intersection of $\gamma$. We assume that one endpoint of $\gamma$ lies on $\partial^\V_j S$. There are now three situations: (1) the other endpoint lies on a different boundary component $\partial_k S \neq \partial^\V_j S$, (2) the other endpoint is an orbifold point, and (3) the other endpoint of $\gamma$ also lies on $\partial^\V_j S$. In the first and second case, let $q$ denote the boundary morphism from $\gamma$ to itself along $\partial^\V_j S$. If this boundary component has winding number $1$, we claim that the $0$-cochain $\phi_1$ determined by $\phi_1 (e_\gamma) = q^2$ defines a nonzero $2$-cocycle. Similarly, if the winding number is $2$, we claim that $\phi_2$ determined by $\phi_2 (e_\gamma) = q$ defines a nonzero $2$-cocycle. Indeed, since $q$ composes trivially with any other generator of $\End (\Gamma)$, we find that $\partial^\bullet \phi_i = 0$ in both cases $i = 1, 2$.

In the third case, let $q_1, q_2$ denote the two boundary morphisms from $\gamma$ to itself along $\partial^\V_j S$. Note that $q_1^2 = 0$ and $q_2^2 = 0$ but both $q_1 q_2$ and $q_2 q_1$ are nonzero. If the winding number along $\partial^\V_j S$ is $1$, we have that $\lvert q_1 \rvert + \lvert q_2 \rvert = 1$ and we claim that $\phi_1$ determined by $\phi_1 (e_\gamma) = q_1 q_2 q_1 q_2 + q_2 q_1 q_2 q_1$ defines a nonzero $2$-cocycle. As before, $q_1, q_2$ compose trivially with all other generators of $\End (\Gamma)$, except for $q_2, q_1$, respectively. However, also for $q_1$ and $q_2$ we have
\begin{align*}
(\partial^\bullet \phi_1) (q_1) &= - q_1 \cdot (q_1 q_2 q_1 q_2 + q_2 q_1 q_2 q_1) + (q_1 q_2 q_1 q_2 + q_2 q_1 q_2 q_1) \cdot q_1  \\
&= - q_1 q_2 q_1 q_2 q_1 + q_1 q_2 q_1 q_2 q_1 = 0
\end{align*}
where the sign difference follows from the sign in the differential $\partial_1$ (cf.\ \eqref{differentialsmaller} or \cite[Section 5.2]{barmeierwang1}).  Similarly, we have $(\partial^\bullet \phi_1) (q_2) = 0$. Since $\phi_1$ is a $0$-cochain of the underlying cochain complex, it cannot be a coboundary, and since it is a map of degree $2$, it defines a nontrivial $2$-cocycle in the Hochschild cohomology of $\End (\Gamma)$. Since $\phi_1$ is a $0$-cochain, it gives rise to a nontrivial curved deformation of $\End (\Gamma)$, by introducing a curvature term, similar to \eqref{eq:curvedfamily}.

The argument is completely analogous when the winding number of $\partial^\V_j S$ is $2$, in which case $\lvert q_1 \rvert + \lvert q_2 \rvert = 2$. Then the $0$-cochain $\phi_2$ determined by $\phi_2 (e_\gamma) = q_1 q_2 + q_2 q_1$ is a map of degree $2$ and defines a nontrivial $2$-cocycle, likewise parametrizing a curved deformation of $\End (\Gamma)$.
\end{proof}

\subsection{Weak duality}
\label{subsection:weakduality}

For $\mathcal W (\mathbf S) \simeq \tw (\add A)$ it turns out that we can avoid the curvature problem by passing from $A$ to the endomorphism algebra of a weak generator. In Section~\ref{subsubsection:weak} we gave a geometric classification of weak generators. For our purposes it suffices to consider a certain weak generator that is obtained from the standard dissection of $\mathbf S$. The passage from $A$ to the endomorphism algebra of this weak generator satisfies a type of duality, closely related to Koszul duality, that we call {\it weak duality}.

For most results in this section we assume that $A$ is a {\it locally proper} DG gentle algebra, i.e.\ $\dim \H^i (A) < \infty$ for all $i \in \mathbb Z$. In Section \ref{subsection:localization} we show how an arbitrary DG gentle algebra can be obtained from a locally proper DG gentle algebra via localization. Note that $A$ is locally proper if and only if the corresponding surface has no boundary components of winding number $0$ without stops in which case $\H^i (A)$ is infinite-dimensional only for $i = 0$. 

The following lemma will be useful to give a conceptual explanation of weak duality.

\begin{lemma}
\label{lemma:propersmooth}
Let $A = (\Bbbk Q / I, d)$ be the DG gentle algebra associated to the standard dissection and for each $i \in Q_0$ let $P_i = e_i A$ be the corresponding indecomposable object of $\add A$. Then the $P_i$'s fall into three pairwise distinct classes.
\begin{enumerate}
\item $\End (P_i)$ is proper and not homologically smooth
\item $\End (P_i)$ is homologically smooth and not proper
\item $\End (P_i)$ is both homologically smooth and proper.
\end{enumerate}
In these three cases, we have, respectively,
\begin{enumerate}
\item $\End (P_i) \simeq \Bbbk [p_i] / (p_i^2)$
\item $\End (P_i) \simeq \Bbbk [q_i]$
\item $\End (P_i) \simeq \Bbbk$.
\end{enumerate}
Moreover, if $A$ is locally proper, then $\lvert q_i \rvert \neq 0$.
\end{lemma}

\begin{proof}
The separation into the three cases for homologically smooth and/or proper follows directly from the explicit form of the endomorphism algebras which can be read off the quiver in Fig.~\ref{fig:quiver}. If $i$ is a vertex without oriented cycles starting and ending at $i$, then $\End (P_i) \simeq \Bbbk$ which is both homologically smooth and proper. If $i$ is a vertex such that there is a nonzero cycle starting and ending at $i$, then this cycle is necessarily of length $1$ or $2$. There are two types of length $1$ cycles (loops). At the vertices $i = \mathrm s (p_j^\IV)$ for $1 \leq j \leq m$ we have a loop $p_j^\IV$ with ${p_j^\IV}^2 = 0$, in which case $\End (P_i) = \Bbbk [p_j^\IV] / ({p_j^\IV})^2$ as in {\itemiii} which is proper but not homologically smooth. At the vertices $i = \mathrm s (q_j^\V)$ for $1 \leq j \leq n$ we have a loop $q_j^\V$ without relations, in which case $\End (P_i) = \Bbbk [q_j^\V]$ as in {\itemii} which is homologically smooth but not proper. Finally, if $i = \mathrm s (p^\I_j)$ or $i = \mathrm t (p^\I_j)$ for $1 \leq j \leq g$, there are length $2$ cycles $q^\I_j r^\I_j$ and $p^\I_j q^\I_j$, respectively, which due to the zero relations $r^\I_j q^\I_j = 0$ and $q^\I_j p^\I_j = 0$ also square to zero. Again, we have $\End (P_i) \simeq \Bbbk [p_i] / (p_i^2)$ as in {\itemi} which again is proper but not homologically smooth.

Since $A$ is locally proper if and only if there are no oriented cycles of degree $0$, this corresponds precisely to the case when $\lvert q_i \rvert \neq 0$, where $q_i$ corresponds to one of the loops without relations $q^\V_1, \dotsc, q^\V_n$.
\end{proof}

The set of $P_i$'s constitutes a classical generator of $\mathcal W (\mathbf S)$. The main idea of weak duality is to replace those $P_j$'s falling into type {\itemii} in Lemma \ref{lemma:propersmooth} --- which are responsible for $A$ being nonproper --- by objects $P_j^{\vee_J}$ such that the resulting collection is still a {\it weak} generator, but now with a finite-dimensional endomorphism algebra. (See Section~\ref{subsection:unbounded} below for how to recover $\mathcal W (\mathbf S)$ from such a weak generator.)

For simplicity, we only formulate this notion of weak duality for the particular DG gentle algebra associated to the standard dissection of Fig.~\ref{fig:dissection}, although a similar idea appears to work for any DG gentle algebra.

Let $A = \Bbbk Q / I$ be the DG gentle algebra associated to the standard dissection of Fig.~\ref{fig:dissection} whose quiver with relations is given in Proposition \ref{proposition:generator}.

Let $J \subset Q_0$ denote the subset of all vertices of $Q$ such that $\End (P_j)$ is not proper. (Note that by Lemma \ref{lemma:propersmooth}, this implies that $\End (P_j)$ is homologically smooth.) For each $j \in J$, let $P^{\vee_J}_j$ denote the twisted complex
\begin{equation}
\label{eq:PveeJ}
P^{\vee_J}_j = \Bigl( P_j [1 - \lvert q_j \rvert] \oplus P_j, \bigl( \begin{smallmatrix} 0 & 0 \\ q_j & 0 \end{smallmatrix} \bigr) \Bigr).
\end{equation}

Geometrically, $P_j$ corresponds to the unique curve in the dissection connecting the boundary component $\partial_0 S$ to a boundary component $\partial_j^\V S$ and $P_j^{\vee_J}$ corresponds to the curve around this boundary component as illustrated in Fig.~\ref{fig:weak}. Algebraically, passing from $P_j$ to $P_j^{\vee_J}$ turns a type {\itemii} object into a type {\itemi} object. 

\begin{lemma}
\label{lemma:EndveeJ}
Let $P^{\vee_J}_j$ be the twisted complex in \eqref{eq:PveeJ}. Then $\End_{\H^0 \tw (\add A)} (P^{\vee_J}_j) \simeq \Bbbk [p_j] / (p_j^2)$ with $\lvert p_j \rvert = 1 - \lvert q_j \rvert$, where $q_j$ is the generator of $\End (P_j)$.
\end{lemma}

\begin{proof}
An endomorphism of twisted complexes is given by an endomorphism of the underlying object (considered without differential). Such an endomorphism gives rise to a nonzero morphism in the homotopy category $\H^0 \tw (\add A)$ if it is a cocycle for the differential $\mu^1_{\tw (\add A)}$. For $P^{\vee_J}_j$ \eqref{eq:PveeJ} an endomorphism of the underlying direct sum $P_j [1 - \lvert q_j \rvert] \oplus P_j$ is given by a $2 \times 2$ matrix $\bigl( \begin{smallmatrix} f_1 & f_2 \\ f_3 & f_4 \end{smallmatrix} \bigr)$ where $f_1, \dotsc, f_4 \in \End (P_j)$.  It defines a morphism in the homotopy category if and only if it satisfies
\begin{equation*}
\begin{pmatrix}
f_1 & f_2 \\
f_3 & f_4
\end{pmatrix}
\begin{pmatrix}
0 & 0 \\
q_j & 0
\end{pmatrix}
=
\begin{pmatrix}
0 & 0 \\
q_j & 0
\end{pmatrix}
\begin{pmatrix}
f_1 & f_2 \\
f_3 & f_4
\end{pmatrix}
\end{equation*}
This shows that we must have $f_2 = 0$ and $f_1 = f_4$. Whereas $f_3$ could in principle be given by multiplication by any polynomial in $q_j$, the terms of order $\geq 1$ in $q_j$ can be removed by a coboundary 
\[
\begin{pmatrix}
0 & 0 \\
0 & q_j^k
\end{pmatrix}
\mapsto \begin{pmatrix}
0 & 0 \\
0 & q_j^k
\end{pmatrix} \begin{pmatrix}
0 & 0 \\
q_j & 0
\end{pmatrix} 
-\begin{pmatrix}
0 & 0 \\
q_j & 0
\end{pmatrix}  \begin{pmatrix}
0 & 0 \\
0 & q_j^k
\end{pmatrix}
=
\begin{pmatrix}
0 & 0 \\
q_j^{k+1} & 0
\end{pmatrix}.
\]
Therefore, we have that $\End (P^{\vee_J}_j)$ is generated by $p_j := \bigl( \begin{smallmatrix} 0 & 0 \\ 1 & 0 \end{smallmatrix} \bigr)$ which indeed has degree $1 - \lvert q_j \rvert$ as it defines a degree $0$ map $P_j [1 - \lvert q_j \rvert] \oplus P_j \to P_j [2 - 2 \lvert q_j \rvert] \oplus P_j [1 - \lvert q_j \rvert]$.
\end{proof}

We have the following important observation.

\begin{figure}
\begin{tikzpicture}[x=1em,y=1em]
\begin{scope}[decoration={markings,mark=at position 0.55 with {\arrow[black]{Stealth[length=4.2pt]}}}]
\node[font=\footnotesize, left] at (-2.5,4.5) {$\mathbf S$};
\draw[line width=.5pt] (0,3) circle(.8em);
\draw[line width=.75pt, line cap=round, color=arccolour] (0,0) to ($(0,3)+(270:.8)$);
\begin{scope}[yshift=3em]
\draw[line width=0pt, line cap=round, postaction={decorate}] (88:.79) -- (80:.806);
\node[font=\scriptsize] at (90:1.35) {$q_j$};
\end{scope}
\node[font=\scriptsize] at (.6,1.2) {$P_j$};
\draw[line width=.5pt, line cap=round] (-3,0) -- (3,0);
\draw[line width=0pt, line cap=round, postaction={decorate}] (-3,0) -- (0,0);
\draw[line width=0pt, line cap=round, postaction={decorate}] (0,0) -- (3,0);
\node[font=\scriptsize] at (-1.5,-.75) {$u$};
\node[font=\scriptsize] at (1.5,-.75) {$v$};
\draw[dash pattern=on 3.9pt off 2pt, line width=.5pt, line cap=round, looseness=2.9] (-3,0) to[bend left=90] (3,0);
\end{scope}
\begin{scope}
\begin{scope}[yshift=-4.5em]
\node[font=\footnotesize, left] at (-2.5,1.5) {$A$};
\node[shape=circle, scale=.7] (L) at (-3,0) {};
\node[shape=circle, scale=.7] (R) at (3,0) {};
\node[shape=circle, scale=.7] (M) at (0,0) {};
\draw[line width=1pt, fill=black] (0,0) circle(0.2ex);
\path[->, line width=.5pt, font=\scriptsize] (M) edge[out=130, in=50, looseness=20] node[right=-.1ex,pos=.8] {$q_j$} (M);
\path[->, line width=.5pt, font=\scriptsize] (L) edge node[below=-.3ex] {$u$} (M);
\path[->, line width=.5pt, font=\scriptsize] (M) edge node[below=-.3ex] {$v$} (R);
\draw[dash pattern=on 0pt off 1.2pt, line width=.6pt, line cap=round] ($(173:1em)$) arc[start angle=173, end angle=135, radius=1em];
\draw[dash pattern=on 0pt off 1.2pt, line width=.6pt, line cap=round] ($(7:1em)$) arc[start angle=7, end angle=45, radius=1em];
\end{scope}
\end{scope}
\begin{scope}[xshift=11em, decoration={markings,mark=at position 0.6 with {\arrow[black]{Stealth[length=4.2pt]}}}]
\node[font=\footnotesize, left] at (-2.5,4.5) {$\mathbf S$};
\draw[line width=.5pt] (0,3) circle(.8em);
\draw[line width=.75pt, line cap=round, color=arccolour] (-1,0) to[out=85, in=240, looseness=.8] ($(0,3)+(150:1.3em)$) arc[start angle=150, end angle=30, radius=1.3] to[out=300, in=95, looseness=.8] (1,0);
\draw[line width=.5pt, line cap=round] (-3,0) -- (3,0);
\draw[line width=0pt, line cap=round, postaction={decorate}] (-3,0) -- (-1,0);
\draw[line width=0pt, line cap=round, postaction={decorate}] (-1,0) -- (1,0);
\draw[line width=0pt, line cap=round, postaction={decorate}] (1,0) -- (3,0);
\node[font=\scriptsize] at (-2,-.75) {$u'$\strut};
\node[font=\scriptsize] at (0,-.75) {$p_j$\strut};
\node[font=\scriptsize] at (2,-.75) {$v'$\strut};
\node[font=\scriptsize] at (-.2,1.2) {$P_j^{\vee_J}$};
\draw[dash pattern=on 3.9pt off 2pt, line width=.5pt, line cap=round, looseness=2.9] (-3,0) to[bend left=90] (3,0);
\end{scope}
\begin{scope}[xshift=22em, decoration={markings,mark=at position 0.6 with {\arrow[black]{Stealth[length=4.2pt]}}}]
\draw[<->, line width=.5pt] (-4.5,2) -- (-6.5,2);
\node[font=\footnotesize, left] at (-2.5,4.5) {$\mathbf S^{\vee_J}$};
\draw[line width=.2em] (0,3) circle(.8em);
\draw[line width=.75pt, line cap=round, color=arccolour] (-1,0) to[out=85, in=240, looseness=.8] ($(0,3)+(150:1.3em)$) arc[start angle=150, end angle=30, radius=1.3] to[out=300, in=95, looseness=.8] (1,0);
\draw[line width=.5pt, line cap=round] (-3,0) -- (3,0);
\draw[line width=0pt, line cap=round, postaction={decorate}] (-3,0) -- (-1,0);
\draw[line width=0pt, line cap=round, postaction={decorate}] (-1,0) -- (1,0);
\draw[line width=0pt, line cap=round, postaction={decorate}] (1,0) -- (3,0);
\node[font=\scriptsize] at (-2,-.75) {$u'$\strut};
\node[font=\scriptsize] at (0,-.75) {$p_j$\strut};
\node[font=\scriptsize] at (2,-.75) {$v'$\strut};
\draw[dash pattern=on 3.9pt off 2pt, line width=.5pt, line cap=round, looseness=2.9] (-3,0) to[bend left=90] (3,0);
\end{scope}
\begin{scope}[xshift=22em, yshift=-4.5em]
\node[font=\footnotesize, left] at (-2.5,1.5) {$A^{\vee_J}$};
\node[shape=circle, scale=.7] (L) at (-3,0) {};
\node[shape=circle, scale=.7] (R) at (3,0) {};
\node[shape=circle, scale=.7] (M) at (0,0) {};
\draw[line width=1pt, fill=black] (0,0) circle(0.2ex);
\path[->, line width=.5pt, font=\scriptsize] (M) edge[out=130, in=50, looseness=20] node[right=-.1ex,pos=.8] {$p_j$} (M);
\path[->, line width=.5pt, font=\scriptsize] (L) edge node[below=-.3ex] {$u'$} (M);
\path[->, line width=.5pt, font=\scriptsize] (M) edge node[below=-.3ex] {$v'$} (R);
\draw[dash pattern=on 0pt off 1.2pt, line width=.6pt, line cap=round] ($(190:.8em)$) arc[start angle=190, end angle=350, radius=.8em];
\draw[dash pattern=on 0pt off 1.2pt, line width=.6pt, line cap=round] ($(118:.8em)$) arc[start angle=118, end angle=62, radius=.8em];
\end{scope}
\end{tikzpicture}
\caption{Modifying the standard dissection $\Gamma$ to $\Gamma^{\vee_J}$ as illustrated gives rise to a weak generator whenever $A = \End (\Gamma)$ is locally proper. Its endomorphism algebra is the DG gentle algebra $A^{\vee_J}$ whose surface model $\mathbf S^{\vee_J}$ is obtained from $\mathbf S$ by adding a full boundary stop to the boundary component without stop.}
\label{fig:weak}
\end{figure}
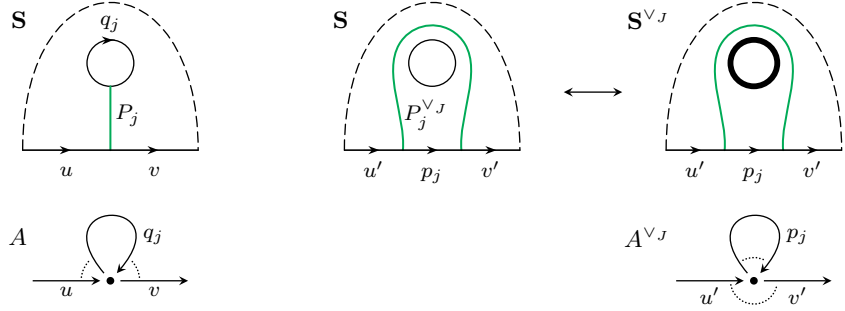

\begin{proposition}
\label{proposition:weakgenerator}
Let $A = (\Bbbk Q / I, d)$ be the DG gentle algebra associated to the standard dissection $\Gamma$ of $\mathbf S$ and let $J$ be the subset of all vertices of $Q$ such that $\End (P_j)$ is not proper. If $A$ is locally proper, then the object $\Gamma^{\vee_J} = \bigoplus_{i \in Q_0 \smallsetminus J} P_i \oplus \bigoplus_{j \in J} P_j^{\vee_J}$ is a weak generator of $\mathcal W (\mathbf S)$ and $\End (\Gamma^{\vee_J})$ is a proper DG gentle algebra.
\end{proposition}

\begin{proof}
Generalizing the situation illustrated in Fig.~\ref{fig:cylinder} and Fig.~\ref{fig:cylinderweak}, the object $\Gamma^{\vee_J}$ is given by curves in $\mathbf S$ obtained from the standard dissection by replacing each curve connecting to a boundary component without stops by a curve around this boundary component as illustrated in Fig.~\ref{fig:weak}. If $A$ is locally proper, then $\mathbf S$ contains no boundary components with winding number $0$ and without stops. Thus the curves in $\Gamma^{\vee_J}$ around the boundary components without stops satisfy the conditions of Theorem \ref{theorem:weak} and thus define a weak generator of $\mathrm D \mathcal W (\mathbf S)$.
\end{proof}

Similarly (and, as we shall see, dually), we denote by $K \subset Q_0$ the subset of all vertices of $Q$ such that $P_k$ is given by a curve around a boundary component with full boundary stop. Note that $\End (P_k)$ is not homologically smooth but proper (see Lemma \ref{lemma:propersmooth}). For each $k \in K$, let $P^{\vee_K}_k$ denote the unbounded, one-sided twisted complex
\begin{equation}
\label{eq:veeK}
P_k^{\vee_K} = \left( \dotsb \oplus P_k [2 - 2 \lvert p_k \rvert] \oplus P_k [1 - \lvert p_k \rvert] \oplus P_k, \left( \begin{smallmatrix} \raisebox{1.5pt}{\scriptsize$\cdot$}\mkern-1mu{\cdot}\mkern-1mu\raisebox{-1.5pt}{\scriptsize$\cdot$} & & & \raisebox{2.1pt}{\scriptsize$\cdot$}\mkern-5.1mu{\cdot}\mkern-5.1mu\raisebox{-2.1pt}{\scriptsize$\cdot$} \\ & 0 & 0 & 0 \\[.15em] & \mathclap{p_k} & 0 & 0 \\ \cdots & 0 & \mathclap{p_k} & 0 \end{smallmatrix} \right) \right).
\end{equation}

Geometrically, $P_k^{\vee_K}$ is represented by a curve spiralling ever closer to the fully stopped boundary component enclosed by the curve corresponding to $P_k$. Algebraically, passing from $P_k$ to $P_k^{\vee_J}$ turns a type {\itemi} object into a type {\itemii} object as mentioned in Lemma \ref{lemma:propersmooth}. 

\begin{lemma}
\label{lemma:EndveeK}
Let $P^{\vee_K}_k$ be the twisted complex in \eqref{eq:veeK} and let $p_k$ denote the generator of $\End (P_k)$. Then $\End_{\H^0 \tw (\add A)} (P^{\vee_K}_k) \simeq \Bbbk \llbracket q_k \rrbracket$ if $\lvert p_k \rvert = 1$ and else $\End_{\H^0 \tw (\add A)} (P^{\vee_K}_k) \simeq \Bbbk [q_k]$. In both cases, we have $\lvert q_k \rvert = 1 - \lvert p_k \rvert$.
\end{lemma}

\begin{proof}
Note that as a complex the cohomology of $P_k^{\vee_K}$ is $1$-dimensional concentrated in degree $0$. That is, $P_k^{\vee_K}$ is a semi-free resolution of the simple module of $\End (P_k)\simeq \Bbbk[p_i]/(p_i^2)$. Thus we have
\begin{align*}
\End_{\H^0 \tw (\add A)} (P^{\vee_K}_k)& \simeq \Hom_{\End (P_k)}(P^{\vee_K}_k, \Bbbk)  \\
& \simeq \Hom_\Bbbk \left (\bigoplus_{n=1}^\infty \Bbbk [n-n\lvert p_k\rvert], \Bbbk\right) \\
&\simeq \begin{cases}
\Bbbk\llbracket q_k \rrbracket & \text{if $\lvert p_k\rvert =1$}\\
\Bbbk[q_k] & \text{otherwise}.
\end{cases}
\end{align*}
If $\lvert p_k\rvert=1$ then the complex $\bigoplus_{n=1}^\infty \Bbbk [n-n \lvert p_k \rvert]$ is concentrated in degree $0$ which is infinite dimensional and $q_k$ is the identity map from $\Bbbk [1 - \lvert p_k \rvert]$ to $\Bbbk$ which is of degree $1 - \lvert p_k \rvert$.
\end{proof}

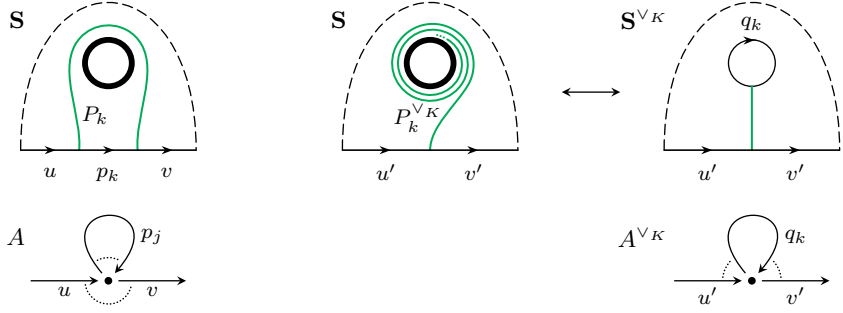
\begin{figure}
\begin{tikzpicture}[x=1em,y=1em]
\begin{scope}[xshift=-11em, decoration={markings,mark=at position 0.6 with {\arrow[black]{Stealth[length=4.2pt]}}}]
\node[font=\footnotesize, left] at (-2.5,4.5) {$\mathbf S$};
\draw[line width=.2em] (0,3) circle(.8em);
\draw[line width=.75pt, line cap=round, color=arccolour] (-1,0) to[out=85, in=240, looseness=.8] ($(0,3)+(150:1.3em)$) arc[start angle=150, end angle=30, radius=1.3] to[out=300, in=95, looseness=.8] (1,0);
\draw[line width=.5pt, line cap=round] (-3,0) -- (3,0);
\draw[line width=0pt, line cap=round, postaction={decorate}] (-3,0) -- (-1,0);
\draw[line width=0pt, line cap=round, postaction={decorate}] (-1,0) -- (1,0);
\draw[line width=0pt, line cap=round, postaction={decorate}] (1,0) -- (3,0);
\node[font=\scriptsize] at (-2,-.75) {$u$\strut};
\node[font=\scriptsize] at (0,-.75) {$p_k$\strut};
\node[font=\scriptsize] at (2,-.75) {$v$\strut};
\node[font=\scriptsize] at (-.45,1.2) {$P_k$};
\draw[dash pattern=on 3.9pt off 2pt, line width=.5pt, line cap=round, looseness=2.9] (-3,0) to[bend left=90] (3,0);
\end{scope}
\begin{scope}[xshift=-11em, yshift=-4.5em]
\node[font=\footnotesize, left] at (-2.5,1.5) {$A$};
\node[shape=circle, scale=.7] (L) at (-3,0) {};
\node[shape=circle, scale=.7] (R) at (3,0) {};
\node[shape=circle, scale=.7] (M) at (0,0) {};
\draw[line width=1pt, fill=black] (0,0) circle(0.2ex);
\path[->, line width=.5pt, font=\scriptsize] (M) edge[out=130, in=50, looseness=20] node[right=-.1ex,pos=.8] {$p_j$} (M);
\path[->, line width=.5pt, font=\scriptsize] (L) edge node[below=-.3ex] {$u$} (M);
\path[->, line width=.5pt, font=\scriptsize] (M) edge node[below=-.3ex] {$v$} (R);
\draw[dash pattern=on 0pt off 1.2pt, line width=.6pt, line cap=round] ($(190:.8em)$) arc[start angle=190, end angle=350, radius=.8em];
\draw[dash pattern=on 0pt off 1.2pt, line width=.6pt, line cap=round] ($(118:.8em)$) arc[start angle=118, end angle=62, radius=.8em];
\end{scope}
\begin{scope}[decoration={markings,mark=at position 0.55 with {\arrow[black]{Stealth[length=4.2pt]}}}]
\draw[line width=.2em] (0,3) circle(.8em);
\draw[line width=.75pt, line cap=round, color=arccolour] (0,0) to[out=90, in=-120, looseness=.8] ($(0,3)+(-30:1.5em)$) arc[start angle=-30, end angle=150, radius=1.4] arc[start angle=150, end angle=330, radius=1.3] arc[start angle=-30, end angle=150, radius=1.2] arc[start angle=150, end angle=330, radius=1.1] arc[start angle=-30, end angle=60, radius=1];
\draw[line width=.75pt, line cap=round, color=arccolour, dash pattern=on 0pt off 1.2pt] ($(0,3)+(62:1em)$) arc[start angle=62, end angle=78, radius=.9];
\node[font=\footnotesize, left] at (-2.5,4.5) {$\mathbf S$};
\node[font=\scriptsize] at (-.4,1.1) {$P_k^{\vee_K}$};
\draw[line width=.5pt, line cap=round] (-3,0) -- (3,0);
\draw[line width=0pt, line cap=round, postaction={decorate}] (-3,0) -- (0,0);
\draw[line width=0pt, line cap=round, postaction={decorate}] (0,0) -- (3,0);
\node[font=\scriptsize] at (-1.5,-.75) {$u'$};
\node[font=\scriptsize] at (1.5,-.75) {$v'$};
\draw[dash pattern=on 3.9pt off 2pt, line width=.5pt, line cap=round, looseness=2.9] (-3,0) to[bend left=90] (3,0);
\end{scope}
\begin{scope}[xshift=11em, decoration={markings,mark=at position 0.55 with {\arrow[black]{Stealth[length=4.2pt]}}}]
\draw[line width=.5pt] (0,3) circle(.8em);
\draw[line width=.75pt, line cap=round, color=arccolour] (0,0) to ($(0,3)+(270:.8)$);
\begin{scope}[yshift=3em]
\draw[line width=0pt, line cap=round, postaction={decorate}] (88:.79) -- (80:.806);
\node[font=\scriptsize] at (90:1.35) {$q_k$};
\end{scope}
\node[font=\footnotesize, left] at (-2.5,4.5) {$\mathbf S^{\vee_K}$};
\draw[<->, line width=.5pt] (-4.5,2) -- (-6.5,2);
\draw[line width=.5pt, line cap=round] (-3,0) -- (3,0);
\draw[line width=0pt, line cap=round, postaction={decorate}] (-3,0) -- (0,0);
\draw[line width=0pt, line cap=round, postaction={decorate}] (0,0) -- (3,0);
\node[font=\scriptsize] at (-1.5,-.75) {$u'$};
\node[font=\scriptsize] at (1.5,-.75) {$v'$};
\draw[dash pattern=on 3.9pt off 2pt, line width=.5pt, line cap=round, looseness=2.9] (-3,0) to[bend left=90] (3,0);
\end{scope}
\begin{scope}[xshift=11em, yshift=-4.5em]
\node[font=\footnotesize, left] at (-2.5,1.5) {$A^{\vee_K}$};
\node[shape=circle, scale=.7] (L) at (-3,0) {};
\node[shape=circle, scale=.7] (R) at (3,0) {};
\node[shape=circle, scale=.7] (M) at (0,0) {};
\draw[line width=1pt, fill=black] (0,0) circle(0.2ex);
\path[->, line width=.5pt, font=\scriptsize] (M) edge[out=130, in=50, looseness=20] node[right=-.1ex,pos=.8] {$q_k$} (M);
\path[->, line width=.5pt, font=\scriptsize] (L) edge node[below=-.3ex] {$u'$} (M);
\path[->, line width=.5pt, font=\scriptsize] (M) edge node[below=-.3ex] {$v'$} (R);
\draw[dash pattern=on 0pt off 1.2pt, line width=.6pt, line cap=round] ($(173:1em)$) arc[start angle=173, end angle=135, radius=1em];
\draw[dash pattern=on 0pt off 1.2pt, line width=.6pt, line cap=round] ($(7:1em)$) arc[start angle=7, end angle=45, radius=1em];
\end{scope}
\end{tikzpicture}
\caption{The counterpart to the situation illustrated in Fig.~\ref{fig:weak}. Replacing an object $P_k$ for $k \in K$ by the infinite twisted complex $P_k^{\vee_K}$ corresponds to a curve spiralling into a fully stopped boundary component. Its endomorphism algebra is the DG gentle algebra $A^{\vee_J}$ whose surface model $\mathbf S^{\vee_K}$ is obtained from $\mathbf S$ by adding a removing the full boundary stop and viewing $P_k^{\vee_K}$ simply as a curve connecting to this boundary component without stop.}
\label{fig:weak2}
\end{figure}

\begin{definition}
\label{definition:weakdual}
Let $A = (\Bbbk Q / I, d)$ be the DG gentle algebra associated to the standard dissection. Let $L \subset Q_0$ be any subset. Then the {\it weak dual} with respect to $L$ is the algebra
\[
A^{\vee_L} := \End \bigl( \textstyle\bigoplus_{i \in Q_0} P_i^{\vee_L} \bigr)
\]
where we set $P_i^{\vee_K} = P_i^{\vee_J} = P_i$ whenever $i \in Q_0 \smallsetminus (J \cup K)$.
\end{definition}

\begin{theorem}
\label{theorem:weakdual}
Let $A = (\Bbbk Q / I, d)$ be the DG gentle algebra associated to the standard dissection of a graded orbifold surface $\mathbf S$. If $A$ is locally proper, then for any subset $L \subset Q_0$ we have the following.
\begin{enumerate}
\item $A^{\vee_L}$ is a DG gentle algebra whose associated surface model is $\mathbf S^{\vee_L} = (S, \Sigma^{\vee_L}, \eta)$ obtained from $\mathbf S$ by removing a full boundary stop whenever the curve around the corresponding boundary component corresponds to a vertex in $L$ and by adding a full boundary stop to an unstopped boundary component if the curve connecting to this boundary component corresponds to a vertex in $L$ (see Fig.~\ref{fig:weak} and Fig.~\ref{fig:weak2}). Note that if $L \supset J \cup K$, then every full boundary stop is removed, and a full boundary stop is added for every unstopped boundary component.
\item We have the following duality
\[
(A^{\vee_L})^{\vee_L} \simeq A.
\]
\item Let $J \subset Q_0$ be the subset of all $j \in Q_0$ such that $\End (P_j)$ is not proper. Then $A^{\vee_J}$ is proper. In particular, if $A$ is proper, then $J = \varnothing$ and $A^{\vee_J} \simeq A$.
\item Let $K \subset Q_0$ be the subset of all $k \in Q_0$ such that $\End (P_k)$ is not homologically smooth. Then $A^{\vee_K}$ is homologically smooth. In particular, if $A$ is homologically smooth, then $K = \varnothing$ and $A^{\vee_K} \simeq A$.
\end{enumerate}
\end{theorem}

\begin{proof}
The relationship between $A$ and $A^{\vee_L}$ is illustrated in Fig.~\ref{fig:weak} for the case that $l \in L$ corresponds to a curve connecting to a boundary component without stop, and in Fig.~\ref{fig:weak2} for the case that $l \in L$ corresponds to a curve around a boundary component with a full boundary stop. The assumption that $A$ be locally proper implies that $P_l$ is not infinite-dimensional and concentrated in a single degree. This ensures that $A^{\vee_L}$ can be understood as the endomorphism algebra of $\Gamma^{\vee_L} = \bigoplus_{l \in L} P_l^{\vee_L} \oplus \bigoplus_{i \in Q_0 \smallsetminus L} P_i$ as follows from the illustrations in Fig.~\ref{fig:weak} and Fig.~\ref{fig:weak2} together with the following observation. Let $k \in L$ be a vertex of $Q_0$ such that $P_k$ corresponds to a curve around a boundary component with full boundary stop. Then $P_k^{\vee_L}$ is the unbounded twisted complex \eqref{eq:veeK}. It follows from the classification of morphisms in \cite{opperplamondonschroll} that the morphisms between the $P_i$'s and $P_i^{\vee_L}$'s (which by Proposition \ref{proposition:unboundedskewgentle} we may view as $\mathbb Z_2$-invariant twisted complexes for a graded gentle algebra) correspond to boundary morphisms as illustrated in Fig.~\ref{fig:weak} and Fig.~\ref{fig:weak2} with the exception that $P_k^{\vee_L}$ has an extra non-identity endomorphism as shown in Lemma \ref{lemma:EndveeK}.
\end{proof}

In terms of the surface, $A$ is locally proper if and only if $\mathbf S$ contains no boundary components with winding number $0$ and no stops. (By \cite[Corollary 9.4.5]{boothgoodbodyopper} this condition also implies that $\mathcal W (\mathbf S)$ is {\it reflexive}.) 

Weak duality is closely related to Koszul duality as the following theorem shows.

\begin{theorem}
\label{theorem:weakkoszul}
$A^{\vee_{Q_0}} = A^{\vee_{J \cup K}}$ is derived equivalent to the Koszul dual $A^!$ of $A$.
\end{theorem} 

\begin{proof}
The relationship between the surface models of a graded gentle algebra and its Koszul dual algebra was studied for smooth surfaces in \cite{opperplamondonschroll,liqiuzhou}. Concretely, the underlying surface and line field coincide, but a full boundary stop is added for every boundary components without stops and every original full boundary stop is removed. Passing from $\mathbf S$ to its smooth double cover, if necessary, one finds that the surface models for $A^{\vee_{Q_0}}$ and $A^!$ coincide, thus both algebras correspond to different dissections of the same surface model and are therefore derived equivalent by the general results on dissections of surfaces \cite{haidenkatzarkovkontsevich,opperplamondonschroll,chokim,barmeierschrollwang}.
\end{proof}

\subsection{Partially wrapped Fukaya categories as categories of unbounded twisted complexes}
\label{subsection:unbounded}

Whereas $\mathcal W (\mathbf S)$ is first and foremost described as the ``twisted completion'' $\tw (\add A)$ of a set $\Gamma$ of generating Lagrangians (see Section~\ref{section:fukayacategories}), we now proceed to describe $\mathcal W (\mathbf S)$ as a category of unbounded twisted complexes over the weak dual $A^{\vee_J}$ of the algebra $A$ associated to the standard dissection, where $J$ corresponds to the set of objects $P_j$ of $\add A$ with nonproper endomorphism algebra as in Lemma \ref{lemma:propersmooth} {\itemii}. By Proposition \ref{proposition:weakgenerator}, these objects together with the $P_i$ for $i \in Q_0 \smallsetminus J$ are a weak generator of $\mathrm D \mathcal W (\mathbf S)$. We shall prove an equivalence
\[
\mathcal W (\mathbf S) \simeq \tw (\add A) \simeq \Tw^{-_J, \b} (\add A^{\vee_J})
\]
whenever $A$ is locally proper. Trading classical generation with weak generation thus comes at the price of enlarging the category from bounded to certain unbounded twisted complexes. However, as $A^{\vee_J}$ is proper (Theorem~\ref{theorem:weakdual}) this tradeoff ends up circumventing the curvature problem when describing the deformation theory of $\mathcal W (\mathbf S)$ in terms of a classical generator (see Section~\ref{subsection:solutioncurvature}), as we saw in the proof of Theorem \ref{theorem:hochschild} that proper DG gentle algebras do not admit any curved deformations.

Since $A^{\vee_J}$ is a (proper) DG gentle algebra whose differential is of the form \eqref{eq:d2}, i.e.\ the differential maps certain arrows to a parallel arrow and all other paths to zero. It follows from Theorem \ref{theorem:formality} that $A^{\vee_J}$ is formal, i.e.\ A$_\infty$-quasi-isomorphic to its cohomology algebra, which by \cite[Section~8]{barmeierschrollwang} is a graded skew-gentle algebra (see also \cite{barmeierwang4}). In other words, $\H^\bullet (A^{\vee_J})$ is a finite-dimensional graded skew-gentle algebra which is quasi-isomorphic to $A^{\vee_J}$.

For a DG category $\mathcal C$, we let $\Tw_{\pm} (\mathcal C)$ denote the category of two-sided unbounded twisted complexes over $\mathcal C$ as in \cite{annologvinenko}. For $\mathcal C = \add A^{\vee_J}$ we denote by $\Tw^{-, \b} (\add A^{\vee_J}) \subset \Tw_{\pm} (\add A^{\vee_J})$ the full subcategory generated by the semi-free resolutions of the simple $\H^\bullet (A^{\vee_J})$-modules, realized as unbounded twisted complexes over $\add A^{\vee_J}$. Here the superscript $^{-, \b}$ is to remind us that when $B$ is an arbitrary associative algebra concentrated in degree $0$, then $\add B \simeq \proj B$ and the category $\Tw^{-, \b} (\add B)$ is a DG enhancement of the homotopy category $\mathrm K^{-, \b} (\proj B)$ of finitely generated projective $B$-modules which is triangle equivalent to the bounded derived category $\mathrm D^\b (B)$.

Likewise for $A^{\vee_J}$, the category $\Tw^{-, \b} (\add A^{\vee_J})$ is a DG enhancement of the category $\mathrm D^{\b} (A^{\vee_J})$ of DG $A^{\vee_J}$-modules with finite-dimensional cohomology, since the latter is also generated by the simple $\H^\bullet (A^{\vee_J})$-modules. Note that the isomorphism classes of the simple $\H^\bullet (A^{\vee_J})$-modules are in bijection with the vertices of the quiver for $A^{\vee_J}$ and with the indecomposable objects in $\add A^{\vee_J}$. Let us denote the objects in $\add A^{\vee_J}$ by $P_i$ for $i \in Q_0$, reminiscent of the fact that when $A^{\vee_J}$ is concentrated in degree $0$, the objects $P_i$ are precisely the indecomposable projective $A^{\vee_J}$-modules.

It follows from the classification of objects in $\mathrm D^\b (B)$ for any proper graded gentle algebra $B$ \cite{opperplamondonschroll} that every object of $\Tw^{-, \b} (\add B)$ corresponds to a curve in the surface $\mathbf S$. Whereas curves in $\tw (\add A^{\vee_J})$ correspond to closed curves or to ``open'' curves whose endpoints lie on $\partial S \smallsetminus \Sigma$, the larger category $\Tw^{-, \b} (\add B)$ also contains curves with one or two infinite ends spiralling around one resp.\ two different boundary components with full boundary stops.

This classification can be generalized to skew-gentle algebras as follows.

\begin{proposition}
\label{proposition:unboundedskewgentle}
Let $\widetilde B$ be a proper graded gentle algebra with a $\mathbb Z_2$-action and let $B \simeq \widetilde B \ast \mathbb Z_2$ be the corresponding proper graded skew-gentle algebra. Then we have quasi-equivalences
\begin{align*}
(\tw (\add \widetilde B) / \mathbb Z_2)^\natural &\simeq \tw (\add B) \\
(\Tw^{-, \b} (\add \widetilde B) / \mathbb Z_2)^\natural &\simeq \Tw^{-, \b} (\add B)
\end{align*}
In particular, every possibly unbounded twisted complex of $\Tw^{-, \b} (\add B)$ may be viewed as a $\mathbb Z_2$-invariant twisted complex in $\Tw^{-, \b} (\add \widetilde B)$.
\end{proposition}

\begin{proof}
The statement about bounded twisted complexes follows from the results of \cite{chokim,amiotplamondon,barmeierschrollwang}, thus we need to give a proof for the statement about the unbounded twisted complexes.

It follows from \cite[Lemma~2.11]{christ} that there is an adjunction between the derived $\infty$-category $\mathcal D (\widetilde B)$ of $\widetilde B$ (which is stable $\Bbbk$-linear) and its group quotient $\infty$-category $\mathcal D (\widetilde B)_{\mathbb Z_2}$, where the latter $\infty$-category can be identified with $\mathcal D (B) \simeq \mathcal D (\widetilde B \ast \mathbb Z_2) \simeq (\mathcal D (\widetilde B) / \mathbb Z_2)^\natural$ by \cite[Proposition~3.5]{christ}, where as usual we denote by $^\natural$ idempotent completion. This adjunction induces an adjunction
\[
\begin{tikzpicture}[x=1em, y=1em]
\node[left, minimum height=3ex] (L) at (0,0) {$\mathrm D (B)$\strut};
\node[right, minimum height=3ex] (R) at (4,0) {$\mathrm D (\widetilde B)$\strut};
\path[<-, line width=.5pt, transform canvas={yshift=.9ex}] (L) edge node[font=\scriptsize, above=-.3ex] {$F$} (R);
\path[->, line width=.5pt, transform canvas={yshift=-.7ex}] (L) edge node[font=\scriptsize, below=-.3ex] {$G$} (R);
\node[font=\tiny] at (2,.05) {$\perp$};
\end{tikzpicture}
\]
on the level of (triangulated) unbounded derived categories. Using this adjunction, one may check that the equivalence $(\mathrm D (\widetilde B) / \mathbb Z_2)^\natural \simeq \mathrm D (B)$ restricts to an equivalence $(\mathrm D^\b (\widetilde B) / \mathbb Z_2)^\natural \simeq \mathrm D^\b (B)$ on the level of {\it bounded} derived categories. The left adjoint functor $F$ maps any object $X$ to $X \oplus -1 \cdot X$, where $-1$ is the generator of $\mathbb Z_2 \simeq \mathbb Z^\times$. Up to the identification $(\mathrm D (\widetilde B) / \mathbb Z_2)^\natural \simeq \mathrm D (B)$ one thus sees that $F (\mathrm D^\b (\widetilde B)) \subset \mathrm D^\b (B)$ since the direct sum of two complexes with finite-dimensional cohomology has finite-dimensional cohomology. We claim that the image of this embedding split-generates $\mathrm D^\b (B)$. Since $\mathrm D^\b (B)$ is generated by the simple $B$-modules, it suffices to show that every simple $B$-module lies in the idempotent closure of $F (\mathrm D^\b (\widetilde B))$. Indeed, if $i$ is a vertex of $\widetilde Q$ that is not fixed under the $\mathbb Z_2$-action, the corresponding simple module $S_i = e_i B / \rad (e_i B)$ is mapped to a simple $B$-module at the corresponding vertex in the quiver $Q$ for $B$. If $i$ is a vertex of $\widetilde Q$ that is fixed under the $\mathbb Z_2$-action, the corresponding simple module $S_i$ is mapped to the direct sum of the two simple modules corresponding to the vertex $i$ in $Q$ endowed with an idempotent loop $\varepsilon$. Up to direct summands, we thus find that all simple $B$-modules lie in the image of $F$, proving the claim.

Lastly, since $\Tw^{-, \b} (\add \widetilde B) / \mathbb Z_2$ is a DG enhancement of $\mathrm D^\b (\widetilde B) / \mathbb Z_2$ and likewise $\Tw^{-, \b} (\add B)$ an enhancement of $\mathrm D^\b (B)$, we can represent any object in $\mathrm D^\b (B)$ by a $\mathbb Z_2$-invariant object in $\Tw^{-, \b} (\add \widetilde B)$ (together with a choice of idempotent).
\end{proof}

The description in Proposition \ref{proposition:unboundedskewgentle} shows that in the case $B = A^{\vee_J}$, every object of $\mathrm D^\b (A^{\vee_J})$ can be represented by an unbounded twisted complex containing infinitely many copies of at most two of the $P_i$'s.

We now define a subcategory of $\Tw^{-, \b} (\add A^{\vee_J})$ as follows.

\begin{definition}
\label{definition:Tw}
Let $A$ be the DG gentle algebra associated to the standard dissection of a graded orbifold surface $\mathbf S$. Let $A^{\vee_J}$ be the proper weak dual of $A$ as above. We denote by
\[
\Tw^{-_J, \b} (\add A^{\vee_J}) \subset \Tw^{-, \b} (\add A^{\vee_J})
\]
the full subcategory of twisted complexes $X$ such that the number of shifted copies of $P_i$ appearing in $X$ is finite whenever $i \not\in J$.
\end{definition}

Note that $P_j$ for $j \in J$ corresponds to the curves around the boundary components without stops in $\mathbf S$ as in Fig.~\ref{fig:weak}. Allowing twisted complexes with infinitely many copies of $P_j$ for $j \in J$, the category $\Tw^{-_J, \b} (\add A^{\vee_J})$ contains precisely those unbounded twisted complexes representing curves spiralling around these boundary components, omitting the curves that spiral around those boundary components that are already fully stopped in $\mathbf S$.

With these definitions, we can now show the following.

\begin{lemma}
$\Tw^{-_J, \b} (\add A^{\vee_J})$ is an idempotent-complete pretriangulated DG category.
\end{lemma}

\begin{proof}
By the definition of $\Tw^{-_J, \b} (\add A^{\vee_J})$, it is closed under shifts and mapping cones. Since $\Tw^{-, \b} (\add A^{\vee_J})$ is idempotent-complete, so is $\Tw^{-_J, \b} (\add A^{\vee_J})$.
\end{proof}

\begin{theorem}
\label{theorem:dualequivalence}
Let $A = (\Bbbk Q / I, d)$ be the DG gentle algebra obtained from the standard dissection of $\mathbf S$ and let $J \subset Q_0$ be the set of all $j \in Q_0$ such that $\End (P_j)$ is nonproper. If $A$ is locally proper, then we have the following equivalence
\[
\mathcal W (\mathbf S) \simeq \tw (\add A) \simeq \Tw^{-_J, \b} (\add A^{\vee_J}).
\]
\end{theorem}

\begin{proof}
The classification of objects in $\Tw^{-, \b} (\add A^{\vee_J})$ shows that each indecomposable object has at most two periodic ends which for $A^{\vee_J}$ are $2$-periodic (see Proposition \ref{proposition:unboundedskewgentle}). Each element in $\Tw^{-, \b} (\add A^{\vee_J})$ is of one of the three types which may be represented diagrammatically as follows
\begin{equation*}
\begin{tikzpicture}[x=4em, y=.75em]
\draw[line cap=round, line width=.5pt, dash pattern=on 3pt off 1pt, rounded corners=.2em] (-.1,1) rectangle (1.5,-1);
\node at (.755,0) {$X$};
\node[left] at (-5,0) {\it first type\strut};
\begin{scope}[yshift=-2.75em]
\node[left] at (-5,0) {\it second type\strut};
\node (M0) at (-4,0) {$\dotsb$};
\node (M1) at (-3,0) {$P_l$};
\node (M2) at (-2,0) {$P_l$};
\node (M3) at (-1,0) {$P_l$};
\node (M4) at (0,0) {\phantom{$P_l$}};
\path[->] (M0) edge (M1);
\path[->] (M1) edge node[above=-.3ex, font=\scriptsize] {$p_l$} (M2);
\path[->] (M2) edge node[above=-.3ex, font=\scriptsize] {$p_l$} (M3);
\path[->] (M3) edge (M4);
\draw[line cap=round, line width=.5pt, dash pattern=on 3pt off 1pt, rounded corners=.2em] (-.1,1) rectangle (1.5,-1);
\node at (.755,0) {$Y$};
\end{scope}
\begin{scope}[yshift=-5.95em]
\node[left] at (-5,0) {\it third type\strut};
\node (T0) at (-4,1) {$\dotsb$};
\node (T1) at (-3,1) {$P_l$};
\node (T2) at (-2,1) {$P_l$};
\node (T3) at (-1,1) {$P_l$};
\node (T4) at (0,1) {\phantom{$P_l$}};
\node (B0) at (-4,-1) {$\dotsb$};
\node (B1) at (-3,-1) {$P_m$};
\node (B2) at (-2,-1) {$P_m$};
\node (B3) at (-1,-1) {$P_m$};
\node (B4) at (0,-1) {\phantom{$P_l$}};
\path[->] (T0) edge (T1);
\path[->] (T1) edge node[above=-.3ex, font=\scriptsize] {$p_l$} (T2);
\path[->] (T2) edge node[above=-.3ex, font=\scriptsize] {$p_l$} (T3);
\path[->] (T3) edge (T4);
\path[->] (B0) edge (B1);
\path[->] (B1) edge node[above=-.3ex, font=\scriptsize] {$p_m$} (B2);
\path[->] (B2) edge node[above=-.3ex, font=\scriptsize] {$p_m$} (B3);
\path[->] (B3) edge (B4);
\draw[line cap=round, line width=.5pt, dash pattern=on 3pt off 1pt, rounded corners=.2em] (-.1,1.6) rectangle (1.5,-1.6);
\node at (.755,0) {$Z$};
\end{scope}
\end{tikzpicture}
\end{equation*}
Here the dashed boxes represent bounded twisted complexes $X, Y, Z \in \tw (\add A^{\vee_J})$ and for simplicity we have suppressed the shifts from the notation. The last two types of objects belong to $\Tw^{-_J,\b} (\add A^{\vee_J})$ if and only if $l$ resp.\ $l, m$ belong to $J$.

Now note that the infinite ends can be represented by $P_j^{\vee_J}$ for $j \in J$. In particular, every object is a (finite) iterated mapping cone of the $P_i$'s for all $i$ together with the $P_j^{\vee_J}$'s for $j \in J$. This shows that $\Tw^{-_J, \b} (\add A^{\vee_J})$ is generated by the $P_i$'s together with the $P_j^{\vee_J}$'s for $j \in J$, i.e.\ $\bigoplus_{j \in J} P_j^{\vee_J} \oplus \bigoplus_{i \in Q_0} P_i$ is a generator. But for each $j \in J$, the object $P_j$ is a mapping cone of a map $P_j^{\vee_J} \to P_j^{\vee_J} [1]$, so the $P_j$'s are superfluous and $\Gamma^{\vee_J} := \bigoplus_{j \in J} P_j^{\vee_J} \oplus \bigoplus_{i \in Q_0 \smallsetminus J} P_i$ also generates. We now have that $\End (\Gamma^{\vee_J}) \simeq A$, so that $\Tw^{-_J, \b} (\add A^{\vee_J}) \simeq \tw (\add A)$.
\end{proof}

\subsection{Deformation equivalence and solution to the curvature problem}
\label{subsection:solutioncurvature}

\begin{proposition}
\label{proposition:Binfinityalgebra}
Let $A^{\vee_J}$ be the (proper) weak dual of a locally proper DG algebra $A$. Then $A$ and $A^{\vee_J}$ have isomorphic Hochschild cochain complexes in the homotopy category of B$_\infty$ algebras.
\end{proposition}

\begin{proof}
Viewing $A$ and $A^{\vee_J}$ as DG categories as usual, we have the following inclusions of DG categories
\begin{align*}
A          &\subset \eqmakebox[Twtw]{$\tw (\add A)$}                    \subset \mathrm{DGMod} (A) \\
A^{\vee_J} &\subset \eqmakebox[Twtw]{$\Tw^{-_J, \b} (\add A^{\vee_J})$} \subset \mathrm{DGMod} (A^{\vee_J})
\end{align*}
where $\mathrm{DGMod} (A)$ is the DG category formed by all (right) DG $A$-modules. In both cases, the first inclusion is given by viewing an object as a one-term twisted complex and the second is induced by the Yoneda embedding. Then by \cite[Theorem 4.4.1]{lowenvandenbergh1} it follows that the induced restriction maps between Hochschild complexes
\begin{equation}
\label{eq:qis}
\begin{tikzpicture}[baseline=-2.6pt]
\matrix (m) [matrix of math nodes, row sep=.75em, text height=1.5ex, column sep=1.5em, text depth=0.25ex, ampersand replacement=\&, inner sep=3pt]
{
\mathrm C^\bullet (\mathrm{DGMod} (A))          \& \mathrm C^\bullet (\tw (\add A))               \& \mathrm C^\bullet (A) \\
\mathrm C^\bullet (\mathrm{DGMod} (A^{\vee_J})) \& \mathrm C^\bullet (\Tw^{-_J, \b} (A^{\vee_J})) \& \mathrm C^\bullet (A^{\vee_J}) \\
};
\path[->,line width=.4pt]
(m-1-1) edge (m-1-2)
(m-1-2) edge (m-1-3)
(m-2-1) edge (m-2-2)
(m-2-2) edge (m-2-3)
;
\end{tikzpicture}
\end{equation}
are all quasi-isomorphisms of B$_\infty$ algebras. (Here we have written $\mathrm C^\bullet (\mathcal C)$ for $\mathrm C^\bullet (\mathcal C, \mathcal C)$ for brevity.) By \cite{keller03} the quasi-equivalence $\tw (\add A) \simeq \Tw^{-_J, \b} (\add A^{\vee_J})$ in Theorem \ref{theorem:dualequivalence} induces a quasi-isomorphism $\mathrm C^\bullet (\tw (\add A)) \to \mathrm C^\bullet (\Tw^{-_J, \b} (\add A^{\vee_J}))$ so that also the rows of \eqref{eq:qis} are isomorphic in the homotopy category of B$_\infty$ algebras. 
\end{proof}

Collecting the various results obtained above, we can now show that weak duality gives a solution to the curvature problem for $\mathbf S$. 

\begin{theorem}
\label{theorem:solution}
Let $\mathbf S$ be any graded orbifold surface with stops in the boundary. If $\mathcal W (\mathbf S)$ is locally proper, then there exists a proper DG gentle algebra $B$ such that
\begin{itemize}
\item $\mathcal W (\mathbf S) \simeq \Tw^{-_J, \b} (\add B)$ as pretriangulated A$_\infty$ categories
\item $\Tw^{-_J, \b} (\add B)$ and $B$ have equivalent deformation theories
\item $B$ admits a semi-universal family $\widetilde B = \{ B_\lambda \}_{\lambda \in \mathbb A^d}$ of (uncurved) DG deformations.
\end{itemize}
In particular, the full deformation theory $\mathcal W (\mathbf S)$ can be described via (uncurved) DG deformations of $B$.
\end{theorem}

\begin{proof}
We may take $A$ to be the DG gentle algebra associated to the standard dissection of $\mathbf S$ and let $B = A^{\vee_J}$ be its (proper) weak dual which is a DG gentle algebra. The first assertion is proved in Theorem \ref{theorem:dualequivalence}. The second assertion is proved in Proposition \ref{proposition:Binfinityalgebra}. It thus remains to prove the third assertion.

As in the proof of Theorem \ref{theorem:hochschild}, let $L$ denote the number of boundary components with one boundary stop and winding number $1$, and let $M_1$ and $M_2$ denote the number of boundary components with a full boundary stop and winding number $1$ and $2$, respectively. By Theorem \ref{theorem:hochschild} we have $L + M_1 + M_2 = \dim \HH^2 (B, B) =: d$. We now construct a family $\widetilde B$ of DG algebras from the $2$-cocycles for $B$ given in the proof of Theorem \ref{theorem:hochschild}. For simplicity, we shall relabel the $2$-cocycles as follows
\[
(\phi_1, \dotsc, \phi_d) = (\phi^\III_{j_1}, \dotsc, \phi^\III_{j_L}, \phi^\IV_{j_1, 1}, \dotsc, \phi^\IV_{j_1, M_1}, \phi^\IV_{j_2, 1}, \dotsc, \phi^\IV_{j_2, M_2})
\]
and similarly for the arrows $p^\III_{j_i}, q^\III_{j_i}, p^\IV_{j_1, i}, p^\IV_{j_2, i}$.

Let $R = \Bbbk [x_1, \dotsc, x_d]$ and let $\widetilde B = (B \otimes R, \widetilde \mu^1, \widetilde \mu^2)$ be the DG $R$-algebra which is defined on the same graded quiver as $B$ (see Proposition \ref{proposition:generator}), but with the relations and differential modified as follows:
\begin{align*}
\widetilde \mu^1 (p_i) &= x_i q_i & 1 &\leq i \leq L \\
\widetilde \mu^2 (p_i \otimes p_i) &= x_i e_i & L + 1 &\leq i \leq L + M_1 \\
\widetilde \mu^1 (p_i) &= x_i e_i & L + M_1 + 1 &\leq i \leq L + M_1 + M_2.
\end{align*}
Note that $\widetilde B$ is a DG gentle $R$-algebra. Here $x_1, \dotsc, x_d$ are considered the coordinate functions on $\mathbb A^d$. A (closed) point $\lambda \in \mathbb A^d \simeq \Spec \Bbbk [x_1, \dotsc, x_d]$ corresponds to the maximal ideal $\mathfrak m_\lambda = (x_1 - \lambda_1, \dotsc, x_d - \lambda_d)$. We write $\Bbbk_\lambda = R / \mathfrak m_\lambda$. 
For any (closed) point $\lambda = (\lambda_1, \dotsc, \lambda_d)$ of $\mathbb A^d$, given by evaluating $x_i \mapsto \lambda_i$, we thus obtain a ``strict'' deformation $B_\lambda = \widetilde B \otimes \Bbbk_\lambda$ of $B$ corresponding to the cocycle $\phi_\lambda = \sum_{1 \leq i \leq d} \lambda_i \phi_i$. Consider the basis $t_1, \dotsc, t_d$ of $\mathrm T_\lambda \mathbb A^d \simeq \Hom (\mathfrak m_\lambda / \mathfrak m_\lambda^2, \Bbbk)$ given by $t_j (x_i - \lambda_i + \mathfrak m_\lambda^2) = \delta_{ij}$ for $1 \leq i, j \leq d$. In this basis the Kodaira--Spencer map $\mathrm T_\lambda \mathbb A^d \to \HH^2 (B_\lambda, B_\lambda)$ is given by $t_i \mapsto [\phi_i]$, where $[\phi_i] \in \HH^2 (B_\lambda, B_\lambda)$ is the cohomology class of the $2$-cocycle $\phi_i$. The semi-universality at $\lambda = 0$ follows by construction since $[\phi_1], \dotsc, [\phi_d]$ span all of $\HH^2 (B, B)$. The versality at any $\lambda$ follows from the observation that $B_\lambda$ is a deformation of $B$, and all of the $2$-cocycles $\phi_1, \dotsc, \phi_d$ for $B$ continue to be $2$-cocycles for $B_\lambda$, but when $\lambda_i \neq 0$, the cocycle $\phi_i$ becomes a coboundary. Thus the Kodaira--Spencer is surjective but for $\lambda \neq 0$ no longer an isomorphism.
\end{proof}

\begin{remark}[Curved Morita deformations]
For any DG (or A$_\infty$) category $\mathcal A$ and a commutative Noetherian $\Bbbk$-algebra $R$, there is a natural map \cite{kellerlowen, lowenvandenbergh3} 
\[
\Theta\colon \mathrm{Def}_{\mathrm{tMor}}(\mathcal A, R) \to \mathrm{MC}(\mathcal A, R)
\]
where $\mathrm{Def}_{\mathrm{tMor}}(\mathcal A, R)$ is the set of Morita $R$-deformations of $\mathcal A$ up to torsion Morita equivalence of deformations, and $\mathrm{MC}(\mathcal A, R)$ is the set of Maurer--Cartan elements in the Hochschild cochain complex of $\mathcal A$ over $R$ up to gauge equivalence. 

Precisely, let $(\mathcal B, \mathcal M, \widetilde{\mathcal B})$ be a Morita $R$-deformation of $\mathcal A$. That is, $\mathcal B$ is a $\Bbbk$-linear DG category, $\mathcal M$ is a Morita $\mathcal A$-$\mathcal B$-bimodule and $\widetilde{\mathcal B}$ is a $R$-linear A$_\infty$ category with $\widetilde{\mathcal B} \otimes_R \Bbbk = \mathcal B$, see \cite[Definition 5.2]{lowenvandenbergh3}. The $R$-linear A$_\infty$ structure on $\widetilde{\mathcal B}$ naturally defines an element $\varphi_{\widetilde{\mathcal B}} \in  \mathrm{MC}(\mathcal B, R)$. The Morita bimodule $\mathcal M$ induces a canonical isomorphism 
\[
\Phi_{\mathcal M} \colon \mathrm{MC}(\mathcal B, R) \to  \mathrm{MC}(\mathcal A, R).
\]
Define $\Theta((\mathcal B, \mathcal M, \widetilde{\mathcal B})) = \Phi_{\mathcal M}(\varphi_{\widetilde{\mathcal B}})$. It is shown in \cite[Theorem 5.4]{lowenvandenbergh3} that $\Theta$ is injective for $R = \Bbbk \llbracket t \rrbracket$. 

In recent work \cite{lehmannlowen,lehmann} the above injection $\Theta$ is extended into a bijection for $R = \Bbbk [t] / (t^n)$
\[
\mathrm{cDef}_{\mathrm{tMor}}(\mathcal A, R) \to \mathrm{MC}(\mathcal A, R)
\]
where $\mathrm{cDef}_{\mathrm{tMor}}(\mathcal A, R)$ is the set of curved Morita $R$-deformations of $\mathcal A$ up to torsion Morita equivalence of deformations.

In our case, let $A$ be a locally proper DG gentle algebra associated to the standard dissection of $\mathbf S$ let $B = A^{\vee_J}$ be its proper weak dual as in Theorem \ref{theorem:solution}. Now let $\mathcal A = \tw(\add A)$ and $\mathcal B = \tw(\add B)$ and $\mathcal B' = \Tw^{-_J, \b} (\add B)$. Since $B = A^{\vee_J}$ is proper and hence admits no curved deformations, it follows from \cite{lowenvandenbergh3,lehmannlowen} that the maps $\mathrm{Def}_{\mathrm{tMor}} (-, R) \to \mathrm{cDef}_{\mathrm{tMor}} (-, R) \to \mathrm{MC} (-, R)$ are bijective when applied to both $\mathcal B$ and $\mathcal B'$. Therefore, we obtain the following diagram
\[
\begin{tikzpicture}[baseline=-2.6pt]
\matrix (m) [matrix of math nodes, row sep=2em, text height=1.5ex, column sep=3em, text depth=0.25ex, ampersand replacement=\&, inner sep=3pt]
{
\mathrm{Def}_{\mathrm{tMor}}(\mathcal A, R)  \& \mathrm{cDef}_{\mathrm{tMor}}(\mathcal A,  R) \& \mathrm{MC} (\mathcal A, R) \\
\mathrm{Def}_{\mathrm{tMor}}(\mathcal B', R) \& \mathrm{cDef}_{\mathrm{tMor}}(\mathcal B', R) \& \mathrm{MC} (\mathcal B', R) \\
\mathrm{Def}_{\mathrm{tMor}}(\mathcal B, R)  \& \mathrm{cDef}_{\mathrm{tMor}}(\mathcal B, R)  \& \mathrm{MC} (\mathcal B, R) \\
};
\path[->,line width=.4pt]
(m-1-2) edge node[font=\footnotesize, above=-.75ex, pos=.45] {$\sim$} (m-1-3)
(m-2-1) edge node[font=\footnotesize, above=-.75ex, pos=.45] {$\sim$} (m-2-2)
(m-2-2) edge node[font=\footnotesize, above=-.75ex, pos=.45] {$\sim$} (m-2-3)
(m-3-1) edge node[font=\footnotesize, above=-.75ex, pos=.45] {$\sim$} (m-3-2)
(m-3-2) edge node[font=\footnotesize, above=-.75ex, pos=.45] {$\sim$} (m-3-3)
(m-1-2) edge node[font=\footnotesize, right=-.55ex, pos=.45] {\rotatebox{90}{$\sim$}} (m-2-2)
(m-1-3) edge node[font=\footnotesize, right=-.55ex, pos=.45] {\rotatebox{90}{$\sim$}} (m-2-3)
(m-2-1) edge node[font=\footnotesize, right=-.55ex, pos=.45] {\rotatebox{90}{$\sim$}} (m-3-1)
(m-2-2) edge node[font=\footnotesize, right=-.55ex, pos=.45] {\rotatebox{90}{$\sim$}} (m-3-2)
(m-2-3) edge node[font=\footnotesize, right=-.55ex, pos=.45] {\rotatebox{90}{$\sim$}} (m-3-3)
;
\path[{Hooks[right]}->, line width=.4pt]
(m-1-1) edge (m-1-2)
;
\end{tikzpicture}
\]
The top row of vertical maps is induced  by the quasi-equivalence $\mathcal A \simeq \mathcal B'$. In this way, we obtain a bijection $\mathrm{cDef}_{\mathrm{tMor}}(\mathcal A, R) \simeq \mathrm{Def}_{\mathrm{tMor}}(\mathcal B, R)$.   
\end{remark}

\begin{remark}[Curved deformations and derived categories of the second kind]
Although we have circumvented the curvature problem for $\mathcal W (\mathbf S)$, it is equally possible to construct a semiuniversal family $\{ A_\lambda \}_{\lambda \in \mathbb A^d}$ of curved deformations for any graded gentle algebra from the natural basis of $2$-cocycles for $A$ given in Theorem \ref{theorem:hochschild} --- including the cocycles giving curvature terms. It would be interesting to understand if there is a family of categories naturally associated to this family of curved algebras --- such as the coderived categories \cite{positselski} or the compactly supported derived categories of the second kind \cite{guanlazarev,guanholsteinlazarev} --- which recovers the family $\{ \Tw^{-_J, \b} (B_\lambda) \}_{\lambda \in \mathbb A^d}$ directly from $\{ A_\lambda \}_{\lambda \in \mathbb A^d}$.
\end{remark}

\subsection{Deformation, localization and completion}
\label{subsection:localization}

In order to tackle the description of the deformation theory of $\mathcal W (\mathbf S)$ in full generality, we need to discuss one final issue that already appeared in the statements of the results of Sections \ref{subsection:weakduality}--\ref{subsection:solutioncurvature} some of which assumed that $A$ is locally proper. Namely, when $\mathbf S$ has both boundary components of winding number $0$ without stops as well as boundary components of winding number $1$ or $2$ without stops, the latter type of boundary component leads to the curvature problem (Section \ref{subsection:curvatureproblem}). However, using weak duality to bypass the curvature problem as in Section \ref{subsection:solutioncurvature}, we would obtain an equivalence (Theorem \ref{theorem:dualequivalence})
\begin{equation}
\label{eq:completion}
\tw (\add \widehat A) \simeq \Tw^{-_J, \b} (\add A^{\vee_J}).
\end{equation}
The former type of boundary component now implies that $A$ is not locally proper so that the completion on the left-hand side is nontrivial, that is, $A \not\simeq \widehat A$. Thus the equivalence \eqref{eq:completion} only recovers $\mathcal W (\mathbf S) \simeq \tw (\add A)$ from the weak dual $A^{\vee_J}$ ``up to completion''. (See \cite{opperplamondonschroll} for further results in this direction for the case of graded gentle algebras.)

In order to recover $\mathcal W (\mathbf S)$ ``on the nose", we observe that $\mathcal W (\mathbf S)$ can be obtained as the localization of another partially wrapped Fukaya category $\mathcal W (\bar{\mathbf S})$ as follows. Writing $\mathbf S = (S, \Sigma, \eta)$ as usual, let $\bar{\mathbf S} = (S, \Sigma \cup \bar\Sigma, \eta)$ be the surface obtained from $\mathbf S$ by adding exactly one boundary stop to each boundary component of winding number $0$ without stop. That is, $\bar \Sigma = \{ \bar \sigma_i \}_i$ where $i$ runs over the boundary components of $S$ such that $\partial_i S \cap \Sigma = \varnothing$ and $\mathrm w_\eta (\partial_i S) = 0$.

Let $\Gamma$ and $\bar \Gamma$ be standard dissections of $\mathbf S$ and $\bar{\mathbf S}$, respectively, and let $A = \End (\Gamma)$ and $\bar A = \End (\bar\Gamma)$ as usual. Since $\mathbf S$ and $\bar{\mathbf S}$ only differ in boundary components of winding number $0$, the difference between $\Gamma$ and $\bar \Gamma$ and $A$ and $\bar A$ can be seen locally, as shown on the left and right of the following picture:
\begin{equation}
\label{eq:loc}
\begin{tikzpicture}[x=1em,y=1em,decoration={markings,mark=at position 0.54 with {\arrow[black]{Stealth[length=4.2pt]}}}, baseline=2em]
\node[font=\footnotesize, left] at (-2.5,4.5) {$\bar A$};
\draw[line width=.5pt] (0,3) circle(.8em);
\draw[line width=.5pt, fill=black] (0,2.2) circle(.15em);
\draw[line width=.75pt, line cap=round, color=arccolour] (-1,0) to ($(0,3)+(240:.8)$) (1,0) to ($(0,3)+(300:.8)$);
\begin{scope}[yshift=3em]
\draw[line width=0pt, line cap=round, postaction={decorate}] (88:.79) -- (80:.806);
\node[font=\scriptsize] at (90:1.35) {$q$};
\node[font=\scriptsize] at (-90:1.35) {$\bar \sigma_i$};
\end{scope}
\draw[line width=.5pt, line cap=round, postaction={decorate}] (-3,0) -- (3,0);
\draw[line width=0pt, line cap=round, postaction={decorate}] (-3,0) -- (-1.25,0);
\draw[line width=0pt, line cap=round, postaction={decorate}] (1.75,0) -- (3,0);
\node[font=\scriptsize] at (0,-.75) {$p$};
\node[font=\scriptsize] at (-2.25,-.75) {$u$};
\node[font=\scriptsize] at (2.25,-.75) {$v$};
\draw[dash pattern=on 3.9pt off 2pt, line width=.5pt, line cap=round, looseness=2.9] (-3,0) to[bend left=90] (3,0);
\begin{scope}[xshift=11em]
\node[font=\footnotesize, left] at (-2.5,4.5) {$\bar A [W^{-1}]$};
\draw[line width=.5pt] (0,3) circle(.8em);
\draw[line width=.75pt, line cap=round, color=arccolour] (-1,0) to ($(0,3)+(240:.8)$) (1,0) to ($(0,3)+(300:.8)$);
\begin{scope}[yshift=3em]
\draw[line width=0pt, line cap=round, postaction={decorate}] (88:.79) -- (80:.806);
\begin{scope}[rotate=177]
\draw[line width=0pt, line cap=round, postaction={decorate}] (88:.79) -- (80:.806);
\end{scope}
\node[font=\scriptsize] at (90:1.35) {$q$};
\node[font=\scriptsize] at (-90:1.35) {$p\mathrlap{^{-1}}$};
\end{scope}
\draw[line width=.5pt, line cap=round, postaction={decorate}] (-3,0) -- (3,0);
\draw[line width=0pt, line cap=round, postaction={decorate}] (-3,0) -- (-1.25,0);
\draw[line width=0pt, line cap=round, postaction={decorate}] (1.75,0) -- (3,0);
\node[font=\scriptsize] at (0,-.75) {$p$};
\node[font=\scriptsize] at (-2.25,-.75) {$u$};
\node[font=\scriptsize] at (2.25,-.75) {$v$};
\draw[dash pattern=on 3.9pt off 2pt, line width=.5pt, line cap=round, looseness=2.9] (-3,0) to[bend left=90] (3,0);
\node at (5em,2em) {$\overset{\text{\tiny Morita}}\sim$};
\end{scope}
\begin{scope}[xshift=21em]
\node[font=\footnotesize, left] at (-2.5,4.5) {$A$};
\draw[line width=.5pt] (0,3) circle(.8em);
\draw[line width=.75pt, line cap=round, color=arccolour] (0,0) to ($(0,3)+(270:.8)$);
\begin{scope}[yshift=3em]
\draw[line width=0pt, line cap=round, postaction={decorate}] (88:.79) -- (80:.806);
\node[font=\scriptsize] at (90:1.35) {$q$};
\end{scope}
\draw[line width=.5pt, line cap=round] (-3,0) -- (3,0);
\draw[line width=0pt, line cap=round, postaction={decorate}] (-3,0) -- (0,0);
\draw[line width=0pt, line cap=round, postaction={decorate}] (0,0) -- (3,0);
\node[font=\scriptsize] at (-1.5,-.75) {$u$};
\node[font=\scriptsize] at (1.5,-.75) {$v$};
\draw[dash pattern=on 3.9pt off 2pt, line width=.5pt, line cap=round, looseness=2.9] (-3,0) to[bend left=90] (3,0);
\end{scope}
\end{tikzpicture}
\end{equation}
Letting $W$ denote the set of arrows of $\bar A$ given by the arrows $p$ in the left picture, the middle picture is explained in the following result.

\begin{proposition}
\label{proposition:localization}
Let $A$ be the DG gentle algebra associated to a standard dissection of a graded orbifold surface $\mathbf S$ and let $\bar A$ be corresponding dissection of $\bar{\mathbf S}$.
There is a localization functor
\[
l \colon \add \bar A \to \add \bar A [W^{-1}]
\]
such that $\bar A [W^{-1}]$ is Morita equivalent to $A$. Moreover, this localization induces a functor
\[
L \colon \mathcal W (\bar{\mathbf S}) \to \mathcal W (\mathbf S)
\]
which may be identified with the stop removal functor, removing $\bar \Sigma$.
\end{proposition}

\begin{proof}
The existence of the localization functor follows from the fundamental constructions of localizations for DG and A$_\infty$ categories, see for example \cite{drinfeld,toen,lyubashenkoovsienko} as well as \cite{ohtanaka,pascaleff,canonacoornaghistellari}. The Morita equivalence between $\bar A [W^{-1}]$ and $A$ follows either by direct computation, or from \cite[Theorem~7.13]{barmeierschrollwang} by observing as in \eqref{eq:loc} that they correspond to different dissections of the orbifold surface $\mathbf S$.
\end{proof}

(See \cite[Section 6]{ganatrapardonshende2} for stop removal functors in general partially wrapped Fukaya categories.)

\begin{theorem}
\label{theorem:localization}
The localization functor $l \colon \add \bar A \to \add \bar A [W^{-1}]$ gives rise to maps between Hochschild complexes
\begin{equation}
\label{eq:hochschildcomplexes}
\mathrm C^\bullet (\bar A, \bar A) \to \mathrm C^\bullet (\bar A, \bar A [W^{-1}]) \leftarrow \mathrm C^\bullet (\bar A [W^{-1}], \bar A [W^{-1}]) \simeq \mathrm C^\bullet (A, A)
\end{equation}
which induce isomorphisms
\[
\HH^n (\bar A, \bar A) \simeq \HH^n (A, A)
\]
for all $n \geq 2$. In particular, every first-order deformation of $\mathcal W (\mathbf S)$ is induced by a first-order deformation of $\mathcal W (\bar{\mathbf S})$.
\end{theorem}

\begin{proof}
Note that both $\bar A$ and $A$ are DG gentle algebras. Since $\bar A [W^{-1}]$ and $A$ are Morita-equivalent by Proposition \ref{proposition:localization}, they have isomorphic Hochschild complexes (see Section \ref{subsection:keller}). The maps in \eqref{eq:hochschildcomplexes} between Hochschild complexes are induced by the pre- and postcomposition maps given by the localization, i.e.\ any cochain $\bar A^{\otimes n} \to \bar A$ induces a cochain $\bar A^{\otimes n} \to \bar A \to \bar A [W^{-1}]$ and any cochain $(\bar A [W^{-1}])^{\otimes n} \to \bar A [W^{-1}]$ induces a cochain $\bar A^{\otimes n} \to (\bar A [W^{-1}])^{\otimes n} \to \bar A [W^{-1}]$.

The statement about the induced isomorphism on $\HH^n$ for $n \geq 2$ now follows from the computation in the proof of Theorem \ref{theorem:hochschild} using the following three observations. Firstly, the arrows in $W$ --- denoted $p$ in \eqref{eq:loc} --- are of degree $0$ and in $\bar A$ they do not appear in any overlaps. Secondly, these arrows have a unique parallel path --- denoted $q$ in \eqref{eq:loc} --- which is likewise of degree $0$. Thirdly, the formal inverses of the arrows in $W$ --- denoted $p^{-1}$ in \eqref{eq:loc} --- do not have any parallel paths. Moreover, the surfaces and quivers for $\bar A$ and $\bar A [W^{-1}]$ are identical away from the paths involving $p$, $q$ and $p^{-1}$. The only difference between the computation of $\HH^n$ for $n \geq 2$ now lies in the existence of $2$-cocycles $\phi_p$ which are determined by
\begin{align*}
\phi_p \colon p^{-1} p &\mapsto e_{\mathrm s (p)} \\
p p^{-1} &\mapsto e_{\mathrm t (p)}.
\end{align*}
Because the relations $p^{-1} p = e_{\mathrm s (p)}$ and $p p^{-1} = e_{\mathrm t (p)}$ are already present in $\bar A [W^{-1}]$, these $2$-cocycles are coboundaries (see \cite[Example 9.18]{barmeierwang1}). These observations now imply that the difference between the Hochschild cohomology of $\bar A$ and of $\bar A [W^{-1}]$ differs at most in cochains of total degree $1$, yielding the isomorphisms for $\HH^n$ for all $n \geq 2$. 
\end{proof}

When $\mathcal W (\mathbf S)$ is subject to the curvature problem, we may view $\mathcal W (\mathbf S)$ as the localization of the locally proper category $\mathcal W (\bar{\mathbf S})$ and deform the latter in order to avoid the need to consider completions of DG gentle algebras that are not locally proper.

\begin{remark}[An algebro-geometric perspective]
The passage from $\mathbf S$ to $\bar{\mathbf S}$ and back via localization can be understood from a B~side perspective by adding or removing a point ``at infinity''. For example, consider the case when $\mathbf S$ is the cylinder/annulus with one stop on one of the boundary components, equipped with a line field of winding number $0$ around both boundary components. Then $\bar{\mathbf S}$ is given by a cylinder with one boundary stop on both boundary components. We then have equivalences
\begin{align*}
\mathrm D \mathcal W (\mathbf S) &\simeq \mathrm D^\b (\mathbb A^1) \\
\mathrm D \mathcal W (\bar{\mathbf S}) &\simeq \mathrm D^\b (\mathbb P^1)
\end{align*}
where the localization is induced by the inclusion $\mathbb A^1 \subset \mathbb P^1$.
\end{remark}

\section{Deformations of partially wrapped Fukaya categories}
\label{section:deformation}

In his ICM 2002 address ``Fukaya categories and deformations'' \cite{seidel1} P.~Seidel outlines a programme for studying Fukaya categories of compact symplectic manifolds via A$_\infty$ deformations of wrapped Fukaya categories of (noncompact) exact symplectic manifolds. Seidel shows how partial compactifications of exact symplectic manifolds give rise to Hochschild cocycles parametrizing A$_\infty$ deformations of their Fukaya categories. Much progress has since been made on both the algebraic and geometric aspects of this programme. On the algebraic side, there is a well-developed theory of deformations of $\Bbbk$-linear categories in various setups: for example, for Abelian categories \cite{lowenvandenbergh1,lowenvandenbergh2,barmeierwang2}, DG and A$_\infty$ categories \cite{keller03,kontsevichsoibelman,genoveselowenvandenbergh} and $\Bbbk$-linear $\infty$-categories \cite{lurie,blanckatzarkovpandit,iwanari}. In each setup, deformations are controlled by a version of the Hochschild complex.

On the geometric side, the relationship of deformations Fukaya categories and partial compactifications envisioned by Seidel have been studied in the context of relative Fukaya categories, see for example \cite{seidel3,sheridan1,sheridan2,perutzsheridan}.

In the context of moduli of quadratic differentials and stability conditions, the relationship between partial compactifications and deformations of $3$-Calabi--Yau categories associated to surfaces has also been studied in \cite{qiu}. In the $\mathbb Z_2$-graded and fully wrapped setting, deformations of gentle algebras and their mirror partners have been studied in \cite{bocklandtvandekreeke1,bocklandtvandekreeke2}.

\subsection{Deformations of \texorpdfstring{$\mathcal W (\mathbf S)$}{W(S)}}
\label{subsection:deformationwrapped}

We are now able to prove our main theorem about the deformation theory of partially wrapped Fukaya categories of graded orbifold surfaces. This theorem shows that any abstract A$_\infty$ deformation has a geometric interpretation as the partially wrapped Fukaya category of another graded orbifold surface.

\begin{theorem}
\label{theorem:deformationwrapped}
Let $\mathbf S$ be a graded orbifold surface with stops and denote by $d = \dim \HH^2 (\mathcal W (\mathbf S), \mathcal W (\mathbf S))$. Then there is a family $\{ \mathcal W_\lambda \}_{\lambda \in \mathbb A^d}$ of idempotent-complete pretriangulated DG categories induced by a family $\{ B_\lambda \}_{\lambda \in \mathbb A^d}$ of DG algebras with the following properties:
\begin{enumerate}
\item $\mathcal W_0 = \mathcal W (\mathbf S)$
\item \label{item2:deformationwrapped} For every $\lambda \in \mathbb A^d$, there is a surface $\mathbf S_\lambda$ obtained from $\mathbf S$ by partial compactification such that $\mathcal W_\lambda = \mathcal W (\mathbf S_\lambda)$. More precisely, the surface $\mathbf S_\lambda$ is obtained from $\mathbf S$ by replacing
\[
\begin{tikzpicture}
\begin{scope}
\node[left, font=\small] at (-2.75em,0) {$\w_\eta (\partial^\III_{j_i} S) = 1$};
\draw[dash pattern=on 3.1pt off 1.6pt, line width=.4pt, line cap=round] (1.25em,0) arc[start angle=0, end angle=360, radius=1.25em];
\draw[line width=.5pt, line cap=round] (0,0) circle(.5em);
\draw[fill=black] (0,-.5em) circle(.15em);
\node[left, font=\scriptsize, fill=white, shape=ellipse, inner sep=0] at (90:.88em) {$\partial_{j_i}^\III \! S$};
\node at (2.5em,0) {$\mapsto$};
\end{scope}
\begin{scope}[xshift=5em]
\node[font=\scriptsize] at (0,0) {$\times$};
\draw[dash pattern=on 3.1pt off 1.6pt, line width=.4pt, line cap=round] (1.25em,0) arc[start angle=0, end angle=360, radius=1.25em];
\node[right] at (2.5em,0) {\strut\textup{if $\lambda^\III_{j_i} \neq 0$}};
\node[right] at (8.5em,0) {\strut\textup{for $1 \leq i \leq L$}};
\end{scope}
\begin{scope}[yshift=-3.5em]
\begin{scope}
\node[left, font=\small] at (-2.75em,0) {$\w_\eta (\partial^\IV_{j_1,i} S) = 1$};
\draw[dash pattern=on 3.1pt off 1.6pt, line width=.4pt, line cap=round] (1.25em,0) arc[start angle=0, end angle=360, radius=1.25em];
\draw[line width=.15em, line cap=round] (0,0) circle(.5em);
\node[left, font=\scriptsize, fill=white, shape=ellipse, inner sep=0] at (90:.88em) {$\partial_{j_1, i}^\IV S$};
\node at (2.5em,0) {$\mapsto$};
\end{scope}
\begin{scope}[xshift=5em]
\node[font=\scriptsize] at (0,0) {$\times$};
\draw[dash pattern=on 3.1pt off 1.6pt, line width=.4pt, line cap=round] (1.25em,0) arc[start angle=0, end angle=360, radius=1.25em];
\node[right] at (2.5em,0) {\strut\textup{if $\lambda^\IV_{j_1, i} \neq 0$}};
\node[right] at (8.5em,0) {\strut\textup{for $1 \leq i \leq M_1$}};
\end{scope}
\end{scope}
\begin{scope}[yshift=-7em]
\begin{scope}
\node[left, font=\small] at (-2.75em,0) {$\w_\eta (\partial^\IV_{j_2,i} S) = 2$};
\draw[dash pattern=on 3.1pt off 1.6pt, line width=.4pt, line cap=round] (1.25em,0) arc[start angle=0, end angle=360, radius=1.25em];
\draw[line width=.15em, line cap=round] (0,0) circle(.5em);
\node[left, font=\scriptsize, fill=white, shape=ellipse, inner sep=0] at (90:.88em) {$\partial_{j_2, i}^\IV S$};
\node at (2.5em,0) {$\mapsto$};
\end{scope}
\begin{scope}[xshift=5em]
\draw[dash pattern=on 3.1pt off 1.6pt, line width=.4pt, line cap=round] (1.25em,0) arc[start angle=0, end angle=360, radius=1.25em];
\node[right] at (2.5em,0) {\strut\textup{if $\lambda^\IV_{j_2, i} \neq 0$}};
\node[right] at (8.5em,0) {\strut\textup{for $1 \leq i \leq M_2$}};
\end{scope}
\end{scope}
\begin{scope}[yshift=-10.5em]
\begin{scope}
\node[left, font=\small] at (-2.75em,0) {$\w_\eta (\partial^\V_{j_1,i} S) = 1$};
\draw[dash pattern=on 3.1pt off 1.6pt, line width=.4pt, line cap=round] (1.25em,0) arc[start angle=0, end angle=360, radius=1.25em];
\draw[line width=.5pt, line cap=round] (0,0) circle(.5em);
\node[left, font=\scriptsize, fill=white, shape=ellipse, inner sep=0] at (90:.88em) {$\partial_{j_1, i}^\V S$};
\node at (2.5em,0) {$\mapsto$};
\end{scope}
\begin{scope}[xshift=5em]
\node[font=\scriptsize] at (0,0) {$\times$};
\draw[dash pattern=on 3.1pt off 1.6pt, line width=.4pt, line cap=round] (1.25em,0) arc[start angle=0, end angle=360, radius=1.25em];
\node[right] at (2.5em,0) {\strut\textup{if $\lambda^\V_{j_1, i} \neq 0$}};
\node[right] at (8.5em,0) {\strut\textup{for $1 \leq i \leq N_1$}};
\end{scope}
\end{scope}
\begin{scope}[yshift=-14em]
\begin{scope}
\node[left, font=\small] at (-2.75em,0) {$\w_\eta (\partial^\V_{j_2,i} S) = 2$};
\draw[dash pattern=on 3.1pt off 1.6pt, line width=.4pt, line cap=round] (1.25em,0) arc[start angle=0, end angle=360, radius=1.25em];
\draw[line width=.5pt, line cap=round] (0,0) circle(.5em);
\node[left, font=\scriptsize, fill=white, shape=ellipse, inner sep=0] at (90:.88em) {$\partial_{j_2, i}^\V S$};
\node at (2.5em,0) {$\mapsto$};
\end{scope}
\begin{scope}[xshift=5em]
\draw[dash pattern=on 3.1pt off 1.6pt, line width=.4pt, line cap=round] (1.25em,0) arc[start angle=0, end angle=360, radius=1.25em];
\node[right] at (2.5em,0) {\strut\textup{if $\lambda^\V_{j_2, i} \neq 0$}};
\node[right] at (8.5em,0) {\strut\textup{for $1 \leq i \leq N_2$}};
\end{scope}
\end{scope}
\end{tikzpicture}
\]
where
$$\lambda = (\lambda^\III_{j_1}, \dotsc, \lambda^\III_{j_L}, \lambda^\IV_{j_1, 1}, \dotsc, \lambda^\IV_{j_1, M_1}, \lambda^\IV_{j_2, 1}, \dotsc, \lambda^\IV_{j_2, M_2},  \lambda^\V_{j_1, 1}, \dotsc, \lambda^\V_{j_1, N_1}, \lambda^\V_{j_2, 1}, \dotsc, \lambda^\V_{j_2, N_2}).$$
\item The family $\{ B_\lambda \}_{\lambda \in \mathbb A^d}$ is versal at each $\lambda \in \mathbb A^d$ and semi-universal at $0 \in \mathbb A^d$, i.e.\ the Kodaira--Spencer map $\mathrm T_\lambda \mathbb A^d \to \HH^2 (B_\lambda, B_\lambda)$ is surjective for all $\lambda$ and an isomorphism for $\lambda = 0$.
\item For each $\lambda \in \mathbb A^d$, the DG algebra $B_\lambda$ is formal.
\item If $\mathbf S$ has a boundary component that does not contribute to $\HH^2 (\mathcal W (\mathbf S), \mathcal W (\mathbf S))$, then $\mathcal W_\lambda$ is generically rigid, i.e.\ $\HH^2 (\mathcal W_\lambda, \mathcal W_\lambda) = 0$ for generic $\lambda$. 
\end{enumerate}
\end{theorem}

See Fig.~\ref{fig:examplepartialcompactification} for an illustration of the effect of deformations of $\mathcal W (\mathbf S)$ described in Theorem \ref{theorem:deformationwrapped}.

\begin{figure}
\begin{tikzpicture}[x=1.3em,y=1.3em,decoration={markings,mark=at position 0.99 with {\arrow[black]{Stealth[length=4.8pt]}}}, scale=.6]
\draw[line width=.5pt] (0,0) ++ (180:4 and 1) arc[start angle=180,end angle=360,x radius=4,y radius=1] to[out=90,in=190,looseness=1] ++(2.5,4) arc[start angle=-90,end angle=90,x radius=.25,y radius=1] to[out=170,in=290,looseness=1] ++(-1,1) arc[start angle=20,end angle=135,radius=5] to[out=225,in=0,looseness=1] (-11,6) ++ (0,-2) to[out=0,in=90] (-9,0) arc[start angle=180,end angle=360,x radius=1,y radius=.25] arc[start angle=180,end angle=0,radius=1.5];
\begin{scope}
\node[font=\tiny, align=center] at (-11.2,7.5) {winding \\ number $1$};
\node[font=\small, align=center] at (-5,8.5) {$\mathbf S$};
\node[font=\tiny, align=center] at (-8,-1.7) {winding \\ number $1$};
\node[font=\tiny, align=center] at (7.4,2.4) {winding \\ number $2$};
\path[->,line width=.6pt] (4,10.4) edge[bend left=30] node[above, font=\small] {deformation of $\mathcal W (\mathbf S)$} (15,10.4);
\draw[line width=.5pt,line cap=round] (5.5,7) ++(200:5) ++(0:2.5) arc[start angle=0,end angle=90,radius=2.5];
\draw[line width=.5pt,line cap=round] (5.5,7) ++(200:5) ++(0:2.5) ++(90:2.5) ++(280:2.5) arc[start angle=280,end angle=170,radius=2.5];
\end{scope}
\begin{scope}[shift={(-2.35,-2.35)}]
\draw[line width=.5pt,line cap=round] (5.5,7) ++(200:5) ++(0:2.5) arc[start angle=0,end angle=90,radius=2.5];
\draw[line width=.5pt,line cap=round] (5.5,7) ++(200:5) ++(0:2.5) ++(90:2.5) ++(280:2.5) arc[start angle=280,end angle=170,radius=2.5];
\end{scope}
\draw[line width=.6pt,dash pattern=on 0pt off 1.5pt,line cap=round] (0,0) ++ (180:4 and 1) arc[start angle=180,end angle=0,x radius=4,y radius=1] (6.5,4) arc[start angle=270,end angle=90,x radius=0.25,y radius=1] (-9,0) arc[start angle=180,end angle=0,x radius=1,y radius=.25];
\draw[line width=.1em,color=black] (-11,6) arc[start angle=90,end angle=450,x radius=.25,y radius=1];
\draw[fill=black, color=black] (85:4 and 1) circle(.15em);
\draw[fill=black, color=black] (235:4 and 1) circle(.15em);
\draw[fill=black, color=black] (325:4 and 1) circle(.15em);
\draw[fill=black, color=black] (-8,0) ++(260:1 and .25) circle(.15em);
\begin{scope}[xshift=30em]
\node[font=\small, overlay, align=center] at (-5.4,8.5) {$\mathbf S_\lambda$};
\draw[line width=.5pt] (0,0) ++ (180:4 and 1) arc[start angle=180,end angle=360,x radius=4,y radius=1] to[out=90,in=270,looseness=.9] ++(2,5) arc[start angle=0,end angle=135,radius=5.5] to[out=225,in=0,looseness=1] (-11,5) to[out=0,in=90] (-8,0) arc[start angle=180,end angle=0,radius=2];
\node[font=\tiny] at (-11,5) {$\times$};
\node[font=\tiny] at (-8,0) {$\times$};
\begin{scope}
\draw[line width=.5pt,line cap=round] (5.5,7) ++(200:5) ++(0:2.5) arc[start angle=0,end angle=90,radius=2.5];
\draw[line width=.5pt,line cap=round] (5.5,7) ++(200:5) ++(0:2.5) ++(90:2.5) ++(280:2.5) arc[start angle=280,end angle=170,radius=2.5];
\end{scope}
\begin{scope}[shift={(-2.35,-2.35)}]
\draw[line width=.5pt,line cap=round] (5.5,7) ++(200:5) ++(0:2.5) arc[start angle=0,end angle=90,radius=2.5];
\draw[line width=.5pt,line cap=round] (5.5,7) ++(200:5) ++(0:2.5) ++(90:2.5) ++(280:2.5) arc[start angle=280,end angle=170,radius=2.5];
\end{scope}
\draw[line width=.6pt,dash pattern=on 0pt off 1.5pt,line cap=round] (0,0) ++ (180:4 and 1) arc[start angle=180,end angle=0,x radius=4,y radius=1]; 
\draw[fill=black, color=black] (85:4 and 1) circle(.15em);
\draw[fill=black, color=black] (235:4 and 1) circle(.15em);
\draw[fill=black, color=black] (325:4 and 1) circle(.15em);
\end{scope}
\end{tikzpicture}
\caption{Illustration of the effect of deformation of the partially wrapped Fukaya category $\mathcal W (\mathbf S)$ of a graded surface $\mathbf S$ (Theorem \ref{theorem:deformationwrapped}). Here $\dim \HH^2 (\mathcal W (\mathbf S), \mathcal W (\mathbf S)) = 3$, coming from the three boundary components with specified winding number (Theorem \ref{theorem:hochschild}). For generic values $\lambda = (\lambda_1, \lambda_2, \lambda_3)$ of the deformation parameters, the resulting deformed category is the partially wrapped Fukaya category $\mathcal W (\mathbf S_\lambda)$ of an orbifold surface $\mathbf S_\lambda$, where two deformation parameters correspond to partial compactifications of the boundary components with winding number $1$ to orbifold points, and the third parameter to the partial compactification of a boundary component with winding number $2$.}
\label{fig:examplepartialcompactification}
\end{figure}
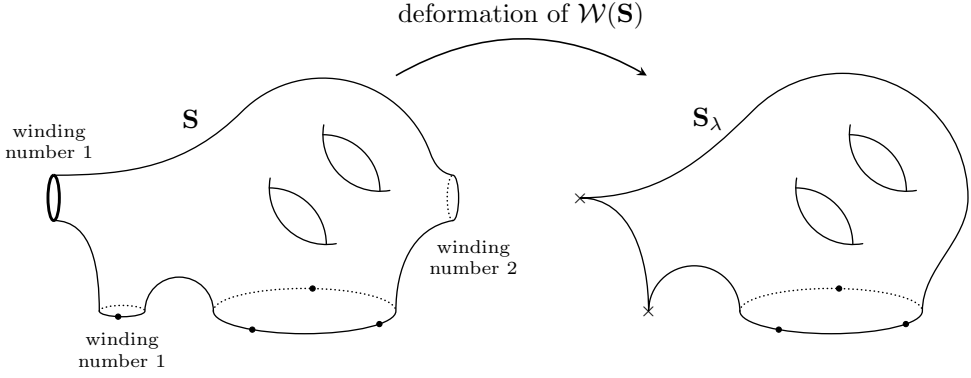

\begin{proof}
Let $A$ denote the algebra associated to a standard dissection of $\mathbf S$. If $A$ is proper, then its family of deformations is constructed in Theorem \ref{theorem:solution}. It remains to deduce the general case and give the geometric interpretation of the deformations. First, let us assume that $A$ is locally proper. Then its weak dual $B = A^{\vee_J}$ is a proper DG gentle algebra which again by Theorem \ref{theorem:solution} admits a semiuniversal family $\widetilde B$ of DG deformations and {\itemiii} is already shown in the proof of this theorem. Writing $B_\lambda = \widetilde B \otimes \Bbbk_\lambda$ as in the same proof, we set $\mathcal W_\lambda = \Tw^{-_J, \b} (\add B_\lambda)$. We have $B_0 \simeq B$ and by Theorem \ref{theorem:dualequivalence} an equivalence $\mathcal W_0 \simeq \Tw^{-_J, \b} (\add B_0) \simeq \tw (\add A) \simeq \mathcal W (\mathbf S)$, showing {\itemi}. 

The surface $\mathbf S^{\vee_J}$ of $B = A^{\vee_J}$ is obtained from $\mathbf S$ by replacing all boundary components without stops into boundary components with a full boundary stop illustrated as follows:
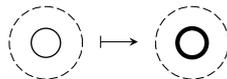
\[
\begin{tikzpicture}
\begin{scope}[yshift=-10.5em]
\begin{scope}
\draw[dash pattern=on 3.1pt off 1.6pt, line width=.4pt, line cap=round] (1.25em,0) arc[start angle=0, end angle=360, radius=1.25em];
\draw[line width=.5pt, line cap=round] (0,0) circle(.5em);
\node at (2.5em,0) {$\mapsto$};
\end{scope}
\begin{scope}[xshift=5em]
\draw[line width=.15em, line cap=round] (0,0) circle(.5em);
\draw[dash pattern=on 3.1pt off 1.6pt, line width=.4pt, line cap=round] (1.25em,0) arc[start angle=0, end angle=360, radius=1.25em];
\end{scope}
\end{scope}
\end{tikzpicture}
\]

For the rest of the proof we shall assume that $\mathbf S$ has one boundary component $\partial_0 S$ with either more than boundary stop, or with winding number $\w_\eta (\partial_0 S) \neq 1$. (The special case when all boundary components with boundary stops have exactly one stop and winding number $1$ will be dealt with in Section~\ref{section:pillowcase} where we will also complete the proof for this case.)

To prove {\itemii} and {\itemiv}, we let $\mathbf S_\lambda$ be the modification of $\mathbf S$ as in the statement of the theorem and let $\mathbf S_\lambda^\vee$ be the surface model obtained from $\mathbf S_\lambda$ by adding a full boundary stop to every boundary component of $\mathbf S_\lambda$ without stop. We claim that $B_\lambda$ is a DG gentle algebra with surface model $\mathbf S_\lambda^\vee$. Since $B = B_0$ is given by the dissection $\Delta$ as in Fig.~\ref{fig:dissection} of $\mathbf S^{\vee_J} = \mathbf S_0^\vee$, we can find a dissection $\Delta_\lambda$ of $\mathbf S_\lambda$ obtained from $\Delta$ by the following modifications
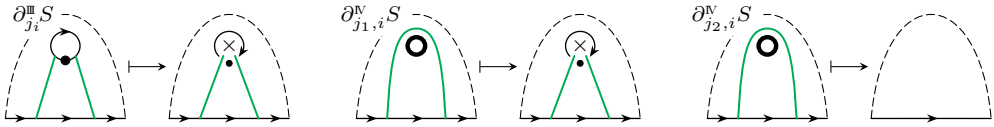
\[
\begin{tikzpicture}[decoration={markings,mark=at position 0.55 with {\arrow[black]{Stealth[length=4.2pt]}}}]
\begin{scope}
\draw[line width=.5pt, line cap=round] (0,0) circle(.5em);
\draw[fill=black] (0,-.5em) circle(.15em);
\node at (2.8em,-.75em) {$\mapsto$};
\draw[line width=0pt,postaction={decorate}] (-1.75em,-2.5em) -- (-1em,-2.5em);
\draw[line width=0pt,postaction={decorate}] (-1em,-2.5em) -- (1.25em,-2.5em);
\draw[line width=0pt,postaction={decorate}] (1.5em,-2.5em) -- (2em,-2.5em);
\draw[line width=0pt,postaction={decorate}] (70:.52em) -- ++(-2:.001em);
\draw[line width=.5pt, line cap=round] (-2.05em,-2.5em) -- (2.05em,-2.5em);
\draw[dash pattern=on 3.1pt off 1.6pt, line width=.4pt, line cap=round] (-2.05em,-2.5em) edge[out=85, in=95, looseness=3] (2.05em,-2.5em);
\draw[line width=.75pt, line cap=round, color=arccolour] (225:.5em) edge (-1em,-2.5em);
\draw[line width=.75pt, line cap=round, color=arccolour] (315:.5em) edge (1em,-2.5em);
\node[left, font=\scriptsize, fill=white, shape=ellipse, inner sep=0] at (95:.95em) {$\partial_{j_i}^\III \! S$};
\end{scope}
\begin{scope}[xshift=5.6em]
\node[font=\scriptsize] at (0,0) {$\times$};
\draw[line width=.5pt, line cap=round] (225:.5em) arc[start angle=225, end angle=-30, radius=.5em];
\draw[fill=black] (0,-.6em) circle(.1em);
\draw[line width=0pt,postaction={decorate}] (-1.75em,-2.5em) -- (-1em,-2.5em);
\draw[line width=0pt,postaction={decorate}] (-1em,-2.5em) -- (1.25em,-2.5em);
\draw[line width=0pt,postaction={decorate}] (1.5em,-2.5em) -- (2em,-2.5em);
\draw[line width=0pt,postaction={decorate}] (-49:.5em) -- ++(237:.001em);
\draw[line width=.5pt, line cap=round] (-2.05em,-2.5em) -- (2.05em,-2.5em);
\draw[dash pattern=on 3.1pt off 1.6pt, line width=.4pt, line cap=round] (-2.05em,-2.5em) edge[out=85, in=95, looseness=3] (2.05em,-2.5em);
\draw[line width=.75pt, line cap=round, color=arccolour] (240:.4em) edge (-1em,-2.5em);
\draw[line width=.75pt, line cap=round, color=arccolour] (300:.4em) edge (1em,-2.5em);
\end{scope}
\begin{scope}[xshift=12em]
\begin{scope}
\draw[line width=.15em, line cap=round] (0,0) circle(.3em);
\node at (2.8em,-.75em) {$\mapsto$};
\draw[line width=0pt,postaction={decorate}] (-1.75em,-2.5em) -- (-1em,-2.5em);
\draw[line width=0pt,postaction={decorate}] (-1em,-2.5em) -- (1.25em,-2.5em);
\draw[line width=0pt,postaction={decorate}] (1.5em,-2.5em) -- (2em,-2.5em);
\draw[line width=.5pt, line cap=round] (-2.05em,-2.5em) -- (2.05em,-2.5em);
\draw[dash pattern=on 3.1pt off 1.6pt, line width=.4pt, line cap=round] (-2.05em,-2.5em) edge[out=85, in=95, looseness=3] (2.05em,-2.5em);
\draw[line width=.75pt, line cap=round, color=arccolour] (-1em,-2.5em) to[out=85, in=225, looseness=.6] (135:.6em) arc[start angle=135, end angle=45, radius=.6em] to[out=-45, in=95, looseness=.6] (1em,-2.5em);
\node[left, font=\scriptsize, fill=white, shape=ellipse, inner sep=0] at (95:.95em) {$\partial_{j_1, i}^\IV S$};
\end{scope}
\begin{scope}[xshift=5.6em]
\node[font=\scriptsize] at (0,0) {$\times$};
\draw[line width=.5pt, line cap=round] (225:.5em) arc[start angle=225, end angle=-30, radius=.5em];
\draw[fill=black] (0,-.6em) circle(.1em);
\draw[line width=0pt,postaction={decorate}] (-1.75em,-2.5em) -- (-1em,-2.5em);
\draw[line width=0pt,postaction={decorate}] (-1em,-2.5em) -- (1.25em,-2.5em);
\draw[line width=0pt,postaction={decorate}] (1.5em,-2.5em) -- (2em,-2.5em);
\draw[line width=0pt,postaction={decorate}] (-49:.5em) -- ++(237:.001em);
\draw[line width=.5pt, line cap=round] (-2.05em,-2.5em) -- (2.05em,-2.5em);
\draw[dash pattern=on 3.1pt off 1.6pt, line width=.4pt, line cap=round] (-2.05em,-2.5em) edge[out=85, in=95, looseness=3] (2.05em,-2.5em);
\draw[line width=.75pt, line cap=round, color=arccolour] (240:.4em) edge (-1em,-2.5em);
\draw[line width=.75pt, line cap=round, color=arccolour] (300:.4em) edge (1em,-2.5em);
\end{scope}
\end{scope}
\begin{scope}[xshift=24em]
\begin{scope}
\draw[line width=.15em, line cap=round] (0,0) circle(.3em);
\node at (2.8em,-.75em) {$\mapsto$};
\draw[line width=0pt,postaction={decorate}] (-1.75em,-2.5em) -- (-1em,-2.5em);
\draw[line width=0pt,postaction={decorate}] (-1em,-2.5em) -- (1.25em,-2.5em);
\draw[line width=0pt,postaction={decorate}] (1.5em,-2.5em) -- (2em,-2.5em);
\draw[line width=.5pt, line cap=round] (-2.05em,-2.5em) -- (2.05em,-2.5em);
\draw[dash pattern=on 3.1pt off 1.6pt, line width=.4pt, line cap=round] (-2.05em,-2.5em) edge[out=85, in=95, looseness=3] (2.05em,-2.5em);
\draw[line width=.75pt, line cap=round, color=arccolour] (-1em,-2.5em) to[out=85, in=225, looseness=.6] (135:.6em) arc[start angle=135, end angle=45, radius=.6em] to[out=-45, in=95, looseness=.6] (1em,-2.5em);
\node[left, font=\scriptsize, fill=white, shape=ellipse, inner sep=0] at (95:.95em) {$\partial_{j_2, i}^\IV S$};
\end{scope}
\begin{scope}[xshift=5.6em]
\draw[line width=0pt,postaction={decorate}] (-1em,-2.5em) -- (1.25em,-2.5em);
\draw[line width=.5pt, line cap=round] (-2.05em,-2.5em) -- (2.05em,-2.5em);
\draw[dash pattern=on 3.1pt off 1.6pt, line width=.4pt, line cap=round] (-2.05em,-2.5em) edge[out=85, in=95, looseness=3] (2.05em,-2.5em);
\end{scope}
\end{scope}
\end{tikzpicture}
\]
whenever the corresponding entry of $\lambda$ is nonzero.

With these modifications, $\Delta_\lambda$ is a formal dissection of $\mathbf S_\lambda$ in the sense of \cite[Section~8]{barmeierschrollwang} (see also Theorem \ref{theorem:formality}) and we claim that the endomorphism algebra of the corresponding generator $\Gamma_\lambda$ is Morita equivalent to $B_\lambda$. Firstly, note that by rescaling $p^\III_{j_i}$, $p^\IV_{j_1, i}$ or $p^\IV_{j_2, i}$ we may assume that the entries of $\lambda$ are either $0$ or $1$. To see that $\End (\Gamma_\lambda)$ is Morita equivalent to $B_\lambda$, we use the following three observations.

\paragraph{\it Observation 1.} There is an algebra isomorphism 
\begin{equation}
\label{eq:obs1}
\begin{tikzpicture}[x=2.5em,y=1.5em,baseline=.5em]
\node[left] at (1,.5) {$\Bbbk \Biggl($};
\node[shape=circle, scale=.7, inner sep=1pt] (1) at (1,0) {$1$};
\node[shape=circle, scale=.7, inner sep=1pt] (2) at (2,0) {$2$};
\node[shape=circle, scale=.7, inner sep=1pt] (3) at (3,0) {$3$};
\path[->, line width=.5pt, font=\scriptsize] (1) edge node[below=-.3ex, pos=.3] {$v_1$} (2);
\path[->, line width=.5pt, font=\scriptsize] (2) edge[out=130, in=50, looseness=14] node[pos=.25, left=-.5ex] {$p$} (2);
\path[->, line width=.5pt, font=\scriptsize] (2) edge node[below=-.3ex, pos=.7] {$v_2$} (3);
\node[right] at (3,.5) {$\Biggr) \biggm/ \!(p^2 - e_2)$};
\draw[dash pattern=on 0pt off 1.2pt, line width=.6pt, line cap=round] ($(2,0)+(-172:.9em)$) arc[start angle=-172, end angle=-6, radius=.9em];
\begin{scope}[xshift=-15.7em]
\node[left] at (1,.5) {$\Bbbk \biggl($};
\node[shape=circle, scale=.7, inner sep=1pt] (1) at (1,.5) {$1$};
\node[shape=circle, scale=.7, inner sep=1pt] (2+) at (2,1.25) {$2^+$};
\node[shape=circle, scale=.7, inner sep=1pt] (2-) at (2,-.25) {$2^-$};
\node[shape=circle, scale=.7, inner sep=1pt] (3) at (3,.5) {$3$};
\path[->, line width=.5pt, font=\scriptsize] (1) edge node[above=-.3ex,pos=.35] {$a$} (2+);
\path[->, line width=.5pt, font=\scriptsize] (1) edge node[below=-.3ex,pos=.35] {$c$} (2-);
\path[->, line width=.5pt, font=\scriptsize] (2+) edge node[above=-.3ex] {$b$} (3);
\path[->, line width=.5pt, font=\scriptsize] (2-) edge node[below=-.5ex] {$d$} (3);
\node[right] at (3,.5) {$\biggr) \Bigm/ \!(b a - d c) \;\toarg{f}$};
\end{scope}
\end{tikzpicture}
\end{equation}
given by
\begin{align*}
f (e_1) &= e_1 & f (e_{2^+}) &= \tfrac12 (e_2 + p) & f (a) &= \tfrac12 (e_2 + p) v_1 & f (b) &= \tfrac12 v_2 (e_2 + p) \\
f (e_3) &= e_3 & f (e_{2^-}) &= \tfrac12 (e_2 - p) & f (c) &= \tfrac12 (e_2 - p) v_1 & f (d) &= \tfrac12 v_2 (e_2 + p).
\end{align*}

\paragraph{\it Observation 2.} There is a quasi-isomorphism
\begin{equation}
\label{eq:obs2}
\begin{tikzpicture}[x=2.5em,y=1.5em,baseline=.5em]
\node[left] at (1,.5) {$\Bbbk \biggl($};
\node[shape=circle, scale=.7, inner sep=1pt] (1) at (1,.5) {$1$};
\node[shape=circle, scale=.7, inner sep=1pt] (2+) at (2,1.25) {$2^+$};
\node[shape=circle, scale=.7, inner sep=1pt] (2-) at (2,-.25) {$2^-$};
\node[shape=circle, scale=.7, inner sep=1pt] (3) at (3,.5) {$3$};
\path[->, line width=.5pt, font=\scriptsize] (1) edge node[above=-.3ex,pos=.35] {$a$} (2+);
\path[->, line width=.5pt, font=\scriptsize] (1) edge node[below=-.3ex,pos=.35] {$c$} (2-);
\path[->, line width=.5pt, font=\scriptsize] (2+) edge node[above=-.3ex] {$b$} (3);
\path[->, line width=.5pt, font=\scriptsize] (2-) edge node[below=-.5ex] {$d$} (3);
\node[right] at (3,.5) {$\biggr) \! \Bigm/ \!(b a - d c) \; \toarg{F_1}$};
\begin{scope}[xshift=15.5em]
\node[left] at (1,.5) {$\biggl( \Bbbk \Bigl($};
\node[shape=circle, scale=.7, inner sep=1pt] (1) at (1,.25) {\strut $1$};
\node[shape=circle, scale=.7, inner sep=0pt] (2-) at (2.15,.25) {\strut $2^-$};
\node[shape=circle, scale=.7, inner sep=0pt] (2+) at (3.3,.25) {\strut $2^+$};
\node[shape=circle, scale=.7, inner sep=1pt] (3) at (4.45,.25) {\strut $3$};
\path[->, line width=.5pt, font=\scriptsize] (1) edge node[below=-.3ex] {$v_1$} (2-);
\path[->, line width=.5pt, font=\scriptsize] (2-) edge node[below=-.3ex] {$p$} (2+);
\path[->, line width=.5pt, font=\scriptsize] (2-) edge[bend left=40] node[above=-.3ex] {$q$} (2+);
\path[->, line width=.5pt, font=\scriptsize] (2+) edge node[below=-.3ex] {$v_2$} (3);
\node[right] at (4.45,.5) {$\Bigr), d (p) = q \biggr)$};
\draw[dash pattern=on 0pt off 1.2pt, line width=.6pt, line cap=round] ($(2.05,.25)+(172:.9em)$) arc[start angle=172, end angle=40, radius=.9em];
\draw[dash pattern=on 0pt off 1.2pt, line width=.6pt, line cap=round] ($(3.2,.25)+(122:.95em)$) arc[start angle=122, end angle=5, radius=.95em];
\end{scope}
\end{tikzpicture}
\end{equation}
given by
\begin{align*}
F_1 (e_1) &= e_1 & F_1 (e_{2^+}) &= e_{2^+} & F_1 (a) &= p v_1 & F_1 (b) &= v_2 \\
F_1 (e_3) &= e_3 & F_1 (e_{2^-}) &= e_{2^-} & F_1 (c) &= v_1   & F_1 (d) &= v_2 p
\end{align*}
for any grading satisfying the constraints $\lvert p \rvert = 0$, $\lvert q \rvert = 1$, $\lvert a \rvert = \lvert c \rvert = \lvert v_1 \rvert$ and $\lvert b \rvert = \lvert d \rvert = \lvert v_2 \rvert$. One easily verifies that $F_1$ is a morphism of DG algebras, where the left-hand side is considered as a DG algebra with trivial differential, inducing an isomorphism with the cohomology of the right-hand side. (Note for example, that $q$ is a coboundary and $p$ is not a cocycle and indeed neither $p$ nor $q$ lie in the image of $F_1$.) Combining $F_1$ and $f$ we also obtain a quasi-isomorphism from the right-hand side of \eqref{eq:obs1} to the right-hand side of \eqref{eq:obs2}.

\paragraph{\it Observation 3.} There is a quasi-isomorphism
\begin{equation}
\label{eq:obs3}
\begin{tikzpicture}[x=2.5em,y=1.5em,baseline=.5em]
\node[left] at (1,.5) {$\Bbbk \bigl($};
\node[shape=circle, scale=.7, inner sep=1pt] (1) at (1,.5) {$1$};
\node[shape=circle, scale=.7, inner sep=1pt] (3) at (2,.5) {$3$};
\path[->, line width=.5pt, font=\scriptsize] (1) edge node[above=-.3ex] {$a$} (3);
\node[right] at (2,.5) {$\bigr) \; \toarg{F_2}$};
\begin{scope}[xshift=7.92em]
\node[left] at (1,.5) {$\Biggl( \Bbbk \biggl($};
\node[shape=circle, scale=.7, inner sep=1pt] (1) at (1,0) {$1$};
\node[shape=circle, scale=.7, inner sep=1pt] (2) at (2,0) {$2$};
\node[shape=circle, scale=.7, inner sep=1pt] (3) at (3,0) {$3$};
\path[->, line width=.5pt, font=\scriptsize] (1) edge node[below=-.3ex, pos=.3] {$v_1$} (2);
\path[->, line width=.5pt, font=\scriptsize] (2) edge[out=130, in=50, looseness=14] node[pos=.25, left=-.5ex] {$p$} (2);
\path[->, line width=.5pt, font=\scriptsize] (2) edge node[below=-.3ex, pos=.7] {$v_2$} (3);
\node[right] at (3,.5) {$\biggr), d (p) = e_2 \Biggr)$};
\draw[dash pattern=on 0pt off 1.2pt, line width=.6pt, line cap=round] ($(2,0)+(-172:.9em)$) arc[start angle=-172, end angle=-6, radius=.9em];
\draw[dash pattern=on 0pt off 1.2pt, line width=.6pt, line cap=round] ($(2,0)+(117:.9em)$) arc[start angle=115, end angle=61, radius=.9em];
\end{scope}
\end{tikzpicture}
\end{equation}
given by
\begin{align*}
F_2 (e_1) &= e_1 & F_2 (e_3) &= e_3 & F_2 (a) = v_2 p v_1
\end{align*}
for any grading satisfying $\lvert p \rvert = -1$ and $\lvert a \rvert = \lvert v_1 \rvert + \lvert v_2 \rvert$.

These three observations assemble into a quasi-isomorphism $\End (\Gamma_\lambda) \simeq B_\lambda$. Namely, in \eqref{eq:obs1} $p$ takes on the role of $p^\IV_{j_1, i}$ for $1 \leq i \leq M_1$ whenever $\lambda^\IV_{j_1, i} \neq 0$ and $v_1, v_2$ take on the role of any path ending resp.\ starting at the source/target of $p^\IV_{j_1, i}$. Here we also use the quasi-isomorphism to the right-hand side of \eqref{eq:obs2}. Similarly, in \eqref{eq:obs2} $p$ and $q$ take on the role of $p^\III_{j_i, 1}$ and $q^\III_{j_i}$ for $1 \leq i \leq L$ whenever $\lambda^\III_{j_i} \neq 0$ and $v_1, v_2$ are again paths. And in \eqref{eq:obs3} $p$ takes on the role of $p^\IV_{j_2, i}$ for $1 \leq i \leq M_2$ whenever $\lambda^\IV_{j_2, i} \neq 0$. In particular, $B_\lambda$ is quasi-isomorphic to a DG gentle algebra, where the differential is again given by sending some arrows to parallel arrows, vanishing on all other paths. In particular, the definition of the category $\Tw^{-_J, \b} (\add B_\lambda)$ of certain unbounded twisted complexes (Definition \ref{definition:Tw}) remains well-defined. Since $\H^\bullet (B_\lambda)$ is again a graded skew-gentle algebra by \cite[Section~8]{barmeierschrollwang} (see also Section ), the results of Section \ref{subsection:unbounded} also apply to $B_\lambda$. In particular, we have that $\tw (\add B_\lambda) \simeq \mathcal W (\mathbf S_\lambda^\vee)$ and $\Tw^{-_J, \b} (\add B_\lambda) \simeq \mathcal W (\mathbf S_\lambda)$. This proves {\itemii} for the case that $\mathcal W (\mathbf S)$ is locally proper.

When $\mathcal W (\mathbf S)$ is not locally proper, we may view it as a localization of a locally proper category $\mathcal W (\bar{\mathbf S})$ (see Section \ref{subsection:localization}). Since by Theorem \ref{theorem:localization} every first-order deformation of $\mathcal W (\mathbf S)$ is induced by a deformation of $\mathcal W (\bar{\mathbf S})$, the family for deformations of $\mathcal W (\bar{\mathbf S})$ just constructed gives rise by localization to a family of deformations of $\mathcal W (\mathbf S) = \mathcal W (\bar{\mathbf S}) [W^{-1}]$.
\end{proof}

\begin{remark}
The deformed surface $\mathbf S_\lambda$ may not necessarily contain a boundary stop. For instance, if all boundary components of $\mathbf S$ with boundary stops have exactly $1$ stop and winding number $1$, but there are further boundary components with full boundary stops which cannot be deformed. A particular case where $\mathbf S$ has exactly $4$ boundary components with boundary stops in $\mathbf S$ have exactly $1$ stop and winding number $1$, then the deformed surface $\mathbf S_\lambda$ is compact without boundary. This is discussed in detail in the next section.
\end{remark}

\section{Fukaya categories of pillowcases}
\label{section:pillowcase}

The results in Section~\ref{subsection:deformationwrapped} show that deformations of $\mathcal W (\mathbf S)$ correspond to compactifications of certain boundary components of $\mathbf S$ with winding number $1$ or $2$. It is thus possible that through deformation {\it all} boundary components of $\mathbf S$ are compactified, resulting in a compact orbifold surface with empty boundary.

In this case, it turns out that the resulting surface is a {\it pillowcase}, a topological sphere with four orbifold points of order $2$. Besides the smooth torus, it is the only compact surface admitting a global line field and a pillowcase can be viewed as a $\mathbb Z_2$-quotient of a torus by an elliptic involution. A qualitative difference with the case of surfaces with boundary is that pillowcases naturally arise in a $1$-parameter family which from a symplectic viewpoint is explained by the existence of a parameter depending on the total volume of the sphere. From a mirror-symmetric viewpoint, we can view the B~side of a pillowcase as a weighted projective line with four weights of order $2$, each corresponding to one of the orbifold points. After fixing the first weighted points to be $\{ 0, 1, \infty \} \in \mathbb P^1$ by an automorphism of $\mathbb P^1$, the moduli of the position of the fourth orbifold point on $\mathbb P^1 \simeq \mathrm S^2$ is a number $\lambda \in \Bbbk \smallsetminus \{ 0, 1 \}$. Its double cover is an elliptic curve $\mathbb C / (\mathbb Z \oplus \lambda \mathbb Z)$.

Let $\mathbf S_\lambda$ be a pillowcase with parameter $\lambda \in \Bbbk \smallsetminus \{ 0, 1 \}$. Just as for orbifold surfaces with boundary, whose Fukaya categories were introduced in \cite{barmeierschrollwang}, one may define the Fukaya category $\mathcal W (\mathbf S_\lambda)$ of a pillowcase $\mathbf S_\lambda$ as the category of global sections of a cosheaf of pretriangulated A$_\infty$ categories, the cosheaf being defined on the ribbon graph (or ribbon complex) dual to a dissection of $\mathbf S_\lambda$ into polygons. The parameter $\lambda$ is related to the area of the polygons appearing in the A$_\infty$ relations (cf.\ Remark \ref{remark:volume}).

We now define a family of DG algebras which we use to give a direct definition of the Fukaya category $\mathcal W (\mathbf S_\lambda)$ of a pillowcase. This definition is based on a particular choice of dissection, illustrated on the right-hand side of Fig.~\ref{fig:pillowcase}. We will show independence of the choice of dissection and further properties of $\mathcal W (\mathbf S_\lambda)$ in forthcoming work. 

\begin{definition}
Consider the graded gentle algebra $A$
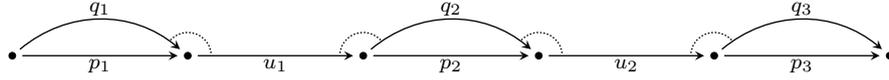
\[
\begin{tikzpicture}[x=2em,y=1.5em,baseline=.5em]
\centering
\foreach \label/\x in {0/0,3/3,6/6,9/9,12/12,15/15} {
\node[shape=circle, scale=.7] (\label) at (\x,0) {};
}
\foreach \x in {0,3,6,9,12,15} {
\draw[line width=1pt, fill=black] (\x,0) circle(0.2ex);
}
\path[->, line width=.5pt, font=\scriptsize] (0) edge node[below=-.5ex] {$p_1$} (3);
\path[->, line width=.5pt, font=\scriptsize, bend left=40] (0) edge node[above=-.5ex] {$q_1$} (3);
\path[->, line width=.5pt, font=\scriptsize] (3) edge node[below=-.5ex] {$u_1$} (6);
\path[->, line width=.5pt, font=\scriptsize] (6) edge node[below=-.5ex] {$p_2$} (9);
\path[->, line width=.5pt, font=\scriptsize, bend left=40] (6) edge node[above=-.5ex] {$q_2$} (9);
\path[->, line width=.5pt, font=\scriptsize] (9) edge node[below=-.5ex] {$u_2$} (12);
\path[->, line width=.5pt, font=\scriptsize] (12) edge node[below=-.5ex] {$p_3$} (15);
\path[->, line width=.5pt, font=\scriptsize, bend left=40] (12) edge node[above=-.5ex] {$q_3$} (15);

\draw[dash pattern=on 0pt off 1.2pt, line width=.6pt, line cap=round] ($(3,0)+(8:.8em)$) arc[start angle=8, end angle=140, radius=.8em];
\draw[dash pattern=on 0pt off 1.2pt, line width=.6pt, line cap=round] ($(6,0)+(172:.8em)$) arc[start angle=172, end angle=40, radius=.8em];
\draw[dash pattern=on 0pt off 1.2pt, line width=.6pt, line cap=round] ($(9,0)+(8:.8em)$) arc[start angle=8, end angle=140, radius=.8em];
\draw[dash pattern=on 0pt off 1.2pt, line width=.6pt, line cap=round] ($(12,0)+(172:.8em)$) arc[start angle=172, end angle=40, radius=.8em];
\end{tikzpicture}
\]
where $|q_1|=|q_2|=|q_3|=1$ and all other arrows are of degree $0$. The surface corresponding to $A$ is given on the left side in Fig.~\ref{fig:pillowcase}.

\begin{figure}
\begin{tikzpicture}[x=1em,y=1em,decoration={markings,mark=at position 0.7 with {\arrow[black]{Stealth[length=4.2pt]}}}]
\draw[line width=.5pt] circle(4.5em);
\draw[line width=0pt,postaction={decorate}] (211:4.5) -- (212:4.5);
\node[font=\scriptsize] at (210:5.25) {$u_1$};
\draw[line width=0pt,postaction={decorate}] (331:4.5) -- (332:4.5);
\node[font=\scriptsize] at (330:5.25) {$u_2$};
\draw[fill=black] (90:4.5) circle(.15em);
\begin{scope}[rotate=-120]
\begin{scope}[yshift=-2.5em]
\draw[line width=0pt,postaction={decorate}] (85:0.74) -- (75:.77);
\draw[line width=.5pt] circle(.75em);
\draw[fill=black] (0,-.75) circle(.15em);
\draw[line width=.75pt, line cap=round, color=arccolour] (225:.75) to[out=245, in=85, looseness=.5] (-.8,-1.9);
\draw[line width=.75pt, line cap=round, color=arccolour] (315:.75) to[out=295, in=95, looseness=.5] (.8,-1.9);
\draw[line width=0pt,postaction={decorate}] (0,-2) -- (.4,-2);
\node[font=\scriptsize] at (90:1.25) {$q_1$};
\node[font=\scriptsize] at (0,-2.75) {$p_1$};
\end{scope}
\end{scope}
\begin{scope}[rotate=0]
\begin{scope}[yshift=-2.5em]
\draw[line width=0pt,postaction={decorate}] (85:0.74) -- (75:.77);
\draw[line width=.5pt] circle(.75em);
\draw[fill=black] (0,-.75) circle(.15em);
\draw[line width=.75pt, line cap=round, color=arccolour] (225:.75) to[out=245, in=85, looseness=.5] (-.8,-1.9);
\draw[line width=.75pt, line cap=round, color=arccolour] (315:.75) to[out=295, in=95, looseness=.5] (.8,-1.9);
\draw[line width=0pt,postaction={decorate}] (0,-2) -- (.4,-2);
\node[font=\scriptsize] at (90:1.25) {$q_2$};
\node[font=\scriptsize] at (0,-2.75) {$p_2$};
\end{scope}
\end{scope}
\begin{scope}[rotate=120]
\begin{scope}[yshift=-2.5em]
\draw[line width=0pt,postaction={decorate}] (85:0.74) -- (75:.77);
\draw[line width=.5pt] circle(.75em);
\draw[fill=black] (0,-.75) circle(.15em);
\draw[line width=.75pt, line cap=round, color=arccolour] (225:.75) to[out=245, in=85, looseness=.5] (-.8,-1.9);
\draw[line width=.75pt, line cap=round, color=arccolour] (315:.75) to[out=295, in=95, looseness=.5] (.8,-1.9);
\draw[line width=0pt,postaction={decorate}] (0,-2) -- (.4,-2);
\node[font=\scriptsize] at (90:1.25) {$q_3$};
\node[font=\scriptsize] at (0,-2.75) {$p_3$};
\end{scope}
\end{scope}
\draw[->,line width=.5pt] (6.5em,0) -- (11.5em,0);
\node[font=\scriptsize] at (9em,.6em) {compactification};
\begin{scope}[xshift=18em]
\draw[line width=.3pt] circle(4.75em);
\node[font=\scriptsize] at (0,0) {$\times$};
\draw[fill=black] (0,.6) circle(.12em);
\begin{scope}[rotate=-120]
\node[font=\scriptsize] at (270:2.75) {$\times$};
\draw[fill=black] (0,-3.4) circle(.12em);
\draw[line width=.75pt, line cap=round, color=arccolour] (267:3) to[out=267, in=255+90, looseness=.5] (255:4.75);
\draw[line width=.75pt, line cap=round, color=arccolour] (273:3) to[out=273, in=285-90, looseness=.5] (285:4.75);
\draw[dash pattern=on 0pt off 1.5pt, line width=.85pt, line cap=round, color=arccolour] (240:.35) to[out=240, in=255-90, looseness=.5] (255:4.75);
\draw[dash pattern=on 0pt off 1.5pt, line width=.85pt, line cap=round, color=arccolour] (300:.35) to[out=300, in=285+90, looseness=.5] (285:4.75);
\draw[->, line width=.5pt, line cap=round] (246:1.8em) arc[start angle=246, end angle=294, radius=1.8em];
\node[font=\scriptsize] at (0,-1.2) {$p_1$};
\draw[->, line width=.5pt, line cap=round] (298:1.8em) arc[start angle=298, end angle=363, radius=1.8em];
\node[font=\scriptsize] at (-30:1.25) {$u_1$};
\begin{scope}[yshift=-2.8em]
\draw[line width=.5pt, line cap=round] (225:.5em) arc[start angle=225, end angle=-30, radius=.5em];
\draw[line width=0pt,postaction={decorate}] (-49:.5em) -- ++(237:.001em);
\node[font=\scriptsize] at (.85,-.6) {$q_1$};
\end{scope}
\end{scope}
\begin{scope}[rotate=0]
\node[font=\scriptsize] at (270:2.75) {$\times$};
\draw[fill=black] (0,-3.4) circle(.12em);
\draw[line width=.75pt, line cap=round, color=arccolour] (267:3) to[out=267, in=255+90, looseness=.5] (255:4.75);
\draw[line width=.75pt, line cap=round, color=arccolour] (273:3) to[out=273, in=285-90, looseness=.5] (285:4.75);
\draw[dash pattern=on 0pt off 1.5pt, line width=.85pt, line cap=round, color=arccolour] (240:.35) to[out=240, in=255-90, looseness=.5] (255:4.75);
\draw[dash pattern=on 0pt off 1.5pt, line width=.85pt, line cap=round, color=arccolour] (300:.35) to[out=300, in=285+90, looseness=.5] (285:4.75);
\draw[->, line width=.5pt, line cap=round] (246:1.8em) arc[start angle=246, end angle=294, radius=1.8em];
\node[font=\scriptsize] at (0,-1.4) {$p_2$};
\draw[->, line width=.5pt, line cap=round] (298:1.8em) arc[start angle=298, end angle=363, radius=1.8em];
\node[font=\scriptsize] at (-30:1.25) {$u_2$};
\begin{scope}[yshift=-2.8em]
\draw[line width=.5pt, line cap=round] (225:.5em) arc[start angle=225, end angle=-30, radius=.5em];
\draw[line width=0pt,postaction={decorate}] (-49:.5em) -- ++(237:.001em);
\node[font=\scriptsize] at (.85,-.6) {$q_2$};
\end{scope}
\end{scope}
\begin{scope}[rotate=120]
\node[font=\scriptsize] at (270:2.75) {$\times$};
\draw[fill=black] (0,-3.4) circle(.12em);
\draw[line width=.75pt, line cap=round, color=arccolour] (267:3) to[out=267, in=255+90, looseness=.5] (255:4.75);
\draw[line width=.75pt, line cap=round, color=arccolour] (273:3) to[out=273, in=285-90, looseness=.5] (285:4.75);
\draw[dash pattern=on 0pt off 1.5pt, line width=.85pt, line cap=round, color=arccolour] (240:.35) to[out=240, in=255-90, looseness=.5] (255:4.75);
\draw[dash pattern=on 0pt off 1.5pt, line width=.85pt, line cap=round, color=arccolour] (300:.35) to[out=300, in=285+90, looseness=.5] (285:4.75);
\draw[->, line width=.5pt, line cap=round] (246:1.8em) arc[start angle=246, end angle=294, radius=1.8em];
\node[font=\scriptsize] at (0,-1.3) {$p_3$};
\begin{scope}[yshift=-2.8em]
\draw[line width=.5pt, line cap=round] (225:.5em) arc[start angle=225, end angle=-30, radius=.5em];
\draw[line width=0pt,postaction={decorate}] (-49:.5em) -- ++(237:.001em);
\node[font=\scriptsize] at (.85,-.6) {$q_3$};
\end{scope}
\end{scope}
\end{scope}
\end{tikzpicture}
\caption{A genus $0$ surface with four boundary components of winding number $1$ (left) and its compactification to a pillowcase (right) modelled by a sphere with three orbifold points on the front hemisphere and one orbifold point on the back (compactified from to the outer boundary component on the left picture)}
\label{fig:pillowcase}
\end{figure}
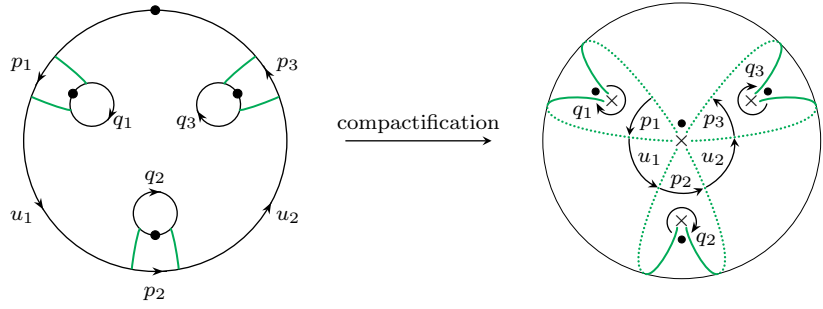

Note that $\dim \HH^2 (A, A) = 4$ with the standard cocycles 
\begin{align*}
\mu^1 (\s p_i) & = \lambda_i \s q_i  & \text{for $i =1, 2, 3$}\\
\mu^5 (\s q_3 \otimes \s u_2 \otimes \s q_2 \otimes \s u_1 \otimes \s q_1) & = \lambda_4 \s p_3 u_2 p_2 u_1 p_1
\end{align*}
and we obtain a $4$-parameter family $\{ A_{\lambda_1,\lambda_2,\lambda_3,\lambda_4} \}$ of A$_\infty$ deformations of $A = A_{0,0,0,0}$ given by the above $\mu^1, \mu^5$ together with the following extra products  
\begin{equation}
\label{eq:higher}
\begin{aligned}
\mu^4 (\s p_3 u_2 \otimes \s q_2 \otimes \s u_1 \otimes \s q_1) &= \lambda_3 \lambda_4 \, \s p_3 u_2 p_2 u_1 p_1 \\
\mu^4 (\s q_3 \otimes \s u_2 \otimes \s p_2 u_1 \otimes \s q_1) &= \lambda_2 \lambda_4 \, \s p_3 u_2 p_2 u_1 p_1 \\
\mu^4 (\s q_3 \otimes \s u_2 \otimes \s q_2 \otimes \s u_1 p_1) &= \lambda_1 \lambda_4 \, \s p_3 u_2 p_2 u_1 p_1 \\
           \mu^3 (\s p_3 u_2 \otimes \s q_2 \otimes \s u_1 p_1) &= \lambda_1 \lambda_3 \lambda_4 \, \s p_3 u_2 p_2 u_1 p_1 \\
           \mu^3 (\s q_3 \otimes \s u_2 \otimes \s p_2 u_1 p_1) &= \lambda_1 \lambda_2 \lambda_4 \, \s p_3 u_2 p_2 u_1 p_1 \\
                      \mu^2 (\s p_3 u_2 p_2 \otimes \s u_1 p_1) &= (1 - \lambda_1 \lambda_2 \lambda_3 \lambda_4) \s p_3 u_2 p_2 u_1 p_1.
\end{aligned}
\end{equation}
\end{definition}

From an algebraic perspective, the extra products \eqref{eq:higher} are necessary in order to obtain a well-defined A$_\infty$ algebra. From a geometric perspective, these products can be viewed as counting pseudo-holomorphic disks in a pillowcase. The counts of pseudo-holomorphic disks in a collection of Lagrangians define an A$_\infty$ category, where the area $a$ of a pseudo-holomorphic disk gives rise to coefficients $T^a$ of the higher multiplications $\mu^i$, $i \geq 1$, where $T$ is the (formal) Novikov parameter. In the case of the pillowcase, the coefficients $\lambda_1, \lambda_2, \lambda_3, \lambda_4$ in \eqref{eq:higher} can be viewed as $\lambda_i = \mathrm e^{-a_i}$, obtained from $T^{a_i}$ by evaluating the Novikov parameter $T$ to $\mathrm e^{-1}$, where $a_i$ is the area of the disk, similar to \cite[Section~4]{polishchukzaslow}. In particular, the case $\lambda_i = 1$ corresponds to the limiting case when the area of the corresponding pseudo-holomorphic disk tends to $0$. However, as a pillowcase has a (fixed) total volume, in a dissection into four polygons, such as one hexagon and three digons as in Fig.~\ref{fig:pillowcase}, only three of their areas can be taken to $0$ while keeping the total volume fixed. This leaves one degree of freedom for a pillowcase --- the total (complexified) volume of the sphere --- which is finite and nonzero.

A straightforward computation gives the following result.

\begin{lemma}
The cohomology algebra $\H^\bullet (A_{\lambda_1,\lambda_2,\lambda_3,\lambda_4})$ is concentrated in degree $0$ and may be given by the quiver
\[
\begin{tikzpicture}[x=5em,y=4em,baseline=.5em]
\begin{scope}[xshift=15.7em]
\node[left] at (1.24,.5) {$\Bbbk \left( \phantom{\frac{\displaystyle\sum_i^i}{2}}\right.$};
\node[shape=circle, scale=.7, inner sep=1pt] (1+) at (1,1) {$1^+$};
\node[shape=circle, scale=.7, inner sep=1pt] (1-) at (1,0) {$1^-$};
\node[shape=circle, scale=.7, inner sep=1pt] (2+) at (2,1) {$2^+$};
\node[shape=circle, scale=.7, inner sep=1pt] (2-) at (2,0) {$2^-$};
\node[shape=circle, scale=.7, inner sep=1pt] (3+) at (3,1) {$3^+$};
\node[shape=circle, scale=.7, inner sep=1pt] (3-) at (3,0) {$3^-$};
\path[->, line width=.5pt, font=\scriptsize] (1+) edge node[above=-.3ex] {$[u_1p_1]$} (2+);
\path[->, line width=.5pt, font=\scriptsize] (1+) edge node[left=.35ex,pos=.4] {$[p_2u_1p_1]$} (2-);
\path[->, line width=.5pt, font=\scriptsize] (1-) edge node[above=.5ex,pos=0] {$[u_1]$} (2+);
\path[->, line width=.5pt, font=\scriptsize] (1-) edge node[below=-.3ex] {$[p_2u_1]$} (2-);
\path[->, line width=.5pt, font=\scriptsize] (2+) edge node[above=-.3ex] {$[u_2p_2]$} (3+);
\path[->, line width=.5pt, font=\scriptsize] (2+) edge node[left=.35ex,pos=.4] {$[p_3u_2p_2]$} (3-);
\path[->, line width=.5pt, font=\scriptsize] (2-) edge node[above=.5ex,pos=0] {$[u_2]$} (3+);
\path[->, line width=.5pt, font=\scriptsize] (2-) edge node[below=-.5ex] {$[p_3u_2]$} (3-);
\node[right] at (2.7,.5) {$\left.\phantom{\frac{\displaystyle\sum_i^i}{2}}\right)$};
\end{scope}
\end{tikzpicture}
\]
with one square skew-commuting:
\[
[p_3 u_2 p_2] [u_1 p_1] = (1 - \lambda_1 \lambda_2 \lambda_3 \lambda_4) [p_3 u_2] [p_2 u_1 p_1]
\]
and the other squares being strictly commutative:
\[
[u_2 p_2] [u_1 p_1] = [u_2] [p_2 u_1 p_1], \quad [p_3 u_2 p_2] [u_1] = [p_3 u_2] [p_2 u_1], \quad [u_2 p_2] [u_1] = [u_2] [p_2 u_1]. 
\]
In particular, $A_{\lambda_1,\lambda_2,\lambda_3,\lambda_4}$ is formal and $\H^\bullet (A_{\lambda_1,\lambda_2,\lambda_3,\lambda_4})$ is a minimal model of $A_{\lambda_1,\lambda_2,\lambda_3,\lambda_4}$.
\end{lemma}

If $\lambda_1, \lambda_2, \lambda_3 \neq 0$, then by an isomorphism may rescale these coefficients to $1$. We denote by $A_\lambda = \H^\bullet (A_{1,1,1,\lambda})$.

\begin{definition}[The Fukaya category of a pillowcase]
\label{definition:fukayapillowcase}
Let $\lambda \in \Bbbk \smallsetminus \{ 0, 1 \}$ and let $\mathbf S_\lambda$ be a pillowcase with parameter $\lambda$. The {\it Fukaya category of a pillowcase} is the ($\mathbb Z$-graded) pretriangulated DG category
\[
\mathcal W (\mathbf S_\lambda) := \tw (\add A_\lambda)^\natural.
\]
\end{definition}

\begin{remark}
Whereas $A_\lambda$ is naturally associated to a particular dissection of $\mathbf S_\lambda$, we show in forthcoming work that (up to quasi-equivalence), $\mathcal W (\mathbf S_\lambda)$ does not depend on the dissection and thus not directly on $A_\lambda$.

It would be interesting to understand the relationship of this definition to the Fukaya categories of pillowcases considered for their relations to knot homology in \cite{heddenheraldkirk1,heddenheraldkirk2,heddenheraldhogancampkirk}.
\end{remark}

\begin{remark}
\label{remark:volume}
In Definition \ref{definition:fukayapillowcase} we exclude the cases $\lambda = 0, 1$. Although the algebras $A_0$ and $A_1$ can naturally be added to the family $\{ A_\lambda \}$, there is a geometric reason for these two exceptions. The algebra $A_0$ is isomorphic to the skew-gentle algebra
\begin{equation}
\label{eq:skewgentle}
\begin{tikzpicture}[x=3.5em,y=1.5em,baseline=.5em]
\node[left] at (.62,.5) {$\Bbbk \Biggl($};
\node[shape=circle, scale=.7, inner sep=1pt] (1) at (1,0) {$1$};
\node[shape=circle, scale=.7, inner sep=1pt] (2) at (2,0) {$2$};
\node[shape=circle, scale=.7, inner sep=1pt] (3) at (3,0) {$3$};
\path[->, line width=.5pt, font=\scriptsize] (1) edge[out=130, in=50, looseness=14] node[pos=.25, left=-.5ex] {$p_1$} (1);
\path[->, line width=.5pt, font=\scriptsize] (2) edge[out=130, in=50, looseness=14] node[pos=.25, left=-.5ex] {$p_2$} (2);
\path[->, line width=.5pt, font=\scriptsize] (3) edge[out=130, in=50, looseness=14] node[pos=.25, left=-.5ex] {$p_3$} (3);
\path[->, line width=.5pt, font=\scriptsize] (1) edge node[below=-.3ex, pos=.5] {$u_1$} (2);
\path[->, line width=.5pt, font=\scriptsize] (2) edge node[below=-.3ex, pos=.5] {$u_2$} (3);
\node[right] at (3.2,.5) {$\Biggr) \biggm/ \!\bigl( p_i^2 - e_i \bigr)_{1 \leq i \leq 3}$};
\draw[dash pattern=on 0pt off 1.2pt, line width=.6pt, line cap=round] ($(2,0)+(-172:.9em)$) arc[start angle=-172, end angle=-6, radius=.9em];
\end{tikzpicture}
\end{equation}
whose surface model is a disk with three orbifold points and one stop in the boundary, corresponding to the fact that $\lambda = \mathrm e^{-\mathrm{vol} (\mathbf S_\lambda)} \to 0$ can be viewed as the limiting case $\lvert \mathrm{vol} (\mathbf S_\lambda) \rvert \to \infty$. Thus the corresponding surface is noncompact and it (or rather its ``inner Liouville domain'') should be represented by a surface with nonempty boundary.

The algebra $A_1$ can be defined as a degeneration of the family $A_\lambda$ which is the only algebra in the family with a monomial relation. Again, understanding $\lambda$ as $\mathrm e^{-\mathrm{vol} (\mathbf S_\lambda)}$, the case $\lambda = 1$ would correspond to a pillowcase with zero volume.
\end{remark}

We can now state our last main result, which shows how Fukaya categories of (compact) pillowcases can arise as deformations of partially wrapped Fukaya categories of certain genus $0$ surfaces with boundary. 

\begin{theorem}
\label{theorem:pillowcase}
Let $\mathbf S$ be a graded orbifold surface with
\begin{itemize}
\item $K \geq 0$ orbifold points
\item $L \geq 1$ boundary components with one boundary stop and winding number $1$
\item $M_1 \geq 0$ boundary components with a full boundary stop and winding number $1$
\item $M_2 \geq 0$ boundary components with a full boundary stop and winding number $2$
\item $N_1 \geq 0$ boundary components without stops and winding number $1$
\item $N_2 \geq 0$ boundary components without stops and winding number $2$
\end{itemize}
and no further boundary components. Then $K + L + M_1 + N_1 = 4$ and
\[
d = \dim \HH^2 (\mathcal W (\mathbf S), \mathcal W (\mathbf S)) = L + M_1 + M_2 + N_1 + N_2 = 4 - K + M_2 + N_2.
\]
For generic $\lambda \in \mathbb A^d$, the category $\mathcal W_\lambda \simeq \mathcal W (\mathbf S_\lambda)$ in the family $\{ \mathcal W_\lambda \}_{\lambda \in \mathbb A^d}$ of Theorem \ref{theorem:deformationwrapped} is equivalent to the Fukaya category of a pillowcase. 
\end{theorem}

\begin{proof}
By the Poincaré--Hopf index theorem we have 
\begin{equation*}
-K - L - M_1 - 2 M_2 - N_1 - 2 N_2 = 4 - 4 g - 2 b
\end{equation*}
where $b$ is the number of boundary components. Since $b = L + M_1 + M_2 + N_1 + N_2$ we get
\begin{equation}
\label{eq:index}
K + L + M_1 + N_1 = 4 - 4 g.
\end{equation}
Since $L \geq 1$ and the other numbers are $\geq 0$ it follows that we must have $g = 0$ and $K + L + M_1 + N_1 = 4$. In particular, there are no constraints on the numbers $M_2$ and $N_2$. The dimension of $\HH^2 (\mathcal W (\mathbf S), \mathcal W (\mathbf S))$ follows from Theorem \ref{theorem:hochschild}. 

By Theorem \ref{theorem:deformationwrapped}, the orbifold surface $\mathbf S_\lambda$, for generic $\lambda \in \mathbb A^d$,  is obtained from $\mathbf S$ by partial compactification, so that  there are $K+L+M_1+N_1=4$ orbifold points and all the other boundary components are compactified to smooth points. That is, $\mathbf S_\lambda$ is a pillowcase and the category $\mathcal W_\lambda$ is equivalent to the Fukaya category of a pillowcase $S_\lambda$. 
\end{proof}

\addtocontents{toc}{\SkipTocEntry}
\section*{Acknowledgements}

We would like to thank Merlin Christ, Sheel Ganatra, Fabian Haiden, Julian Holstein, Gustavo Jasso, Dmitry Kaledin, Kyoungmo Kim, Alessandro Lehmann, Sebastian Opper, Yu Qiu and Kyungmin Rho for helpful discussions and comments during the writing of this article.

This work originated during the Junior Trimester Program ``New Trends in Representation Theory'' at the Hausdorff Research Institute for Mathematics in Bonn, funded by the Deutsche Forschungsgemeinschaft (DFG, German Research Foundation) under Germany's Excellence Strategy -- EXC-2047/1 -- 390685813, and we would like to thank both the organisers and the institute for providing an excellent research environment. The third author was also supported by the National Key R\&D Program of China (2024YFA1013803), the Fundamental Research Funds for the Central Universities (No.\ 020314380037) and the NSFC (Nos.\ 13004005 and 12371043).

\end{document}